\documentclass[10pt]{article}

\usepackage[parfill]{parskip}
\usepackage{caption,graphicx}
\usepackage[justification=centering, singlelinecheck=false, margin=0pt]{subfig}
\usepackage{multicol}
\usepackage{array}
\usepackage{amsbsy,amsfonts,amsmath}
\usepackage[text={6.75in,9.5in},centering,letterpaper]{geometry}
\usepackage{mathabx}
\usepackage[comma,square,sort&compress,numbers]{natbib}
\usepackage{afterpage}
\usepackage{placeins}
\usepackage{booktabs}
\usepackage{ntheorem}
\usepackage{color}
\usepackage{graphicx}
\usepackage{mathdots}
\usepackage{url}
\usepackage{multirow}
\usepackage{rotating}
\usepackage{longtable}

\setcaptionmargin{0.25in}

\newtheorem{Lemma}{Lemma}[section]

\newtheorem{Proposition}[Lemma]{Proposition}
\newtheorem{Remark}[Lemma]{Remark}

\newtheorem{Claim}[Lemma]{Claim}
\newtheorem{Theorem}{Theorem}

\theoremstyle{break}

\newenvironment{proof}[1][.]%
 {\begin{trivlist}\item[]\textbf{Proof#1 }}%
 {\hspace*{\fill}$\rule{0.3\baselineskip}{0.35\baselineskip}$\end{trivlist}}

\newenvironment{Acknowledgment}%
 {\begin{trivlist}\item[]\textbf{Acknowledgments }}{\end{trivlist}}

\newcolumntype{C}{>{\centering\arraybackslash}p{0.5\linewidth}}
\newcolumntype{R}{>{\centering\arraybackslash}p{0.25\linewidth}}
\newcolumntype{L}{>{\centering\arraybackslash}p{0.15\linewidth}}

\makeatletter\@addtoreset{equation}{section}\makeatother
\renewcommand{\theequation}{\arabic{section}.\arabic{equation}}

\renewcommand{\theLemma}{\arabic{section}.\arabic{Lemma}}

\newcommand{\N}{\mathbb{N}}             
\newcommand{\R}{\mathbb{R}}             
\newcommand{\Z}{\mathbb{Z}}             

\newcommand{\calO}{{\cal{O}}}           
\newcommand{\rmd}{\mathrm{d}}           
\newcommand{\rme}{\mathrm{e}}           
\newcommand{\erfi}{\mathrm{erfi}}           
\newcommand{\Li}{\mathrm{Li}}           
\newcommand{\alg}{\mathrm{alg}}           

\newcommand{\Span}[1]{\mathrm{span}\left\{#1\right\}}
\newcommand{\sech}{\mathrm{sech}}
\newcommand{\csch}{\mathrm{csch}}

\newcommand{\pmat}[1]{\begin{pmatrix}#1\end{pmatrix}}
\DeclareMathOperator*{\argmax}{\arg\!\max}

\title{An explanation of metastability in the viscous Burgers equation with periodic boundary conditions via a spectral analysis}
\author{\vspace{-2ex}Kelly McQuighan\\\vspace{-2ex} Department of Mathematics and Statistics\\\vspace{-2ex}Boston University\\Boston, MA  02215, USA \and \vspace{-2ex}C. Eugene Wayne\\\vspace{-2ex} Department of Mathematics and Statistics\\\vspace{-2ex}Boston University\\Boston, MA  02215, USA}

\begin{document}
\maketitle
\begin{abstract}\noindent
A ``metastable solution" to a differential equation typically refers to a family of solutions for which solutions with initial data near the family converge to the family much faster than evolution along the family. Metastable families have been observed both experimentally and numerically in various contexts; they are believed to be particularly relevant for organizing the dynamics of fluid flows. In this work we propose a candidate metastable family for the Burgers equation with periodic boundary conditions. Our choice of family is motivated by our numerical experiments. We furthermore explain the metastable behavior of the family without reference to the Cole--Hopf transformation, but rather by linearizing the Burgers equation about the family and analyzing the spectrum of the resulting operator. We hope this may make the analysis more readily transferable to more realistic systems like the Navier--Stokes equations. Our analysis is motivated by ideas from singular perturbation theory and Melnikov theory. 
\end{abstract}
\section{Introduction}
In the study of differential equations one often is interested in understanding the long-term asymptotic behavior of solutions; the long term behavior could include, for example, convergence to a periodic orbit or a steady-state. One typical approach is to prove the existence of a particular solution and then to argue that nearby initial data converge to that solution; in the case of a steady-state or periodic orbit, such arguments often involve computations of the linear spectrum. 

In this work we address a slightly different question, which arises when the asymptotic state only emerges after a ``long" time; in this case, it may be that the intermediate transient behavior of the system is physically relevant. In other words, we are not interested in \underline{what} the asymptotic state is, but \underline{how} solutions with a wide class of initial data approach it. To address this question we analyze what are known as ``metastable" solutions. The term metastable solution often refers to a family of profiles with the following properties: (1) a profile within this family evolves within the family and tends asymptotically toward the long-time asymptotic state (which is typically a boundary point of the metastable family); (2) solutions with ``nearby" initial data remain near the family for all forward times; and (3) the timescale on which solutions with nearby initial data approach the family is much faster than the evolution within the family towards the asymptotic state. Property (3) is what makes metastable solutions of physical interest.

Metastable solution families are of particular interest in fluid dynamics. For example, in the Navier--Stokes equation with periodic boundary conditions
\begin{align}\label{eq:NS}
&\partial_t \vec u =\nu\Delta\vec u-\vec u\cdot\nabla\vec u+\nabla p, \qquad \nabla \cdot \vec u = 0, \qquad \vec u\in\R^2,\; \nu\ll 1\nonumber\\
&\vec u(x, y, t) = \vec u(x+2\pi, y, t),\quad\text{and}\quad\vec u(x,y, t) = \vec u(x, y+2\pi, t),
\end{align}
which describes two-dimensional viscous fluid flows, metastable vortex pairs known as ``dipoles" were numerically observed \cite{Matthaeus1991, Yin2003}; the dipoles emerge quickly and persist for long times before eventually converging to the trivial state. The metastable states described in \cite{Matthaeus1991, Yin2003} are characterized in terms of their vorticity $\omega$, defined as $\omega := \nabla\times\vec u$. In \cite{Yin2003} a second metastable family known as ``bar" states---solutions with constant vorticity in one spatial direction and periodic vorticity in the other---were observed; which of the two candidate metastable families dominates the dynamics depends on the initial data. 

A related context in which metastability has been observed and studied is Burgers equation.
Although the Burgers equation is unphysical, it is nevertheless relevant to fluid dynamics since it is, in some sense, the one-dimensional simplified analog of the Navier--Stokes equation. Thus, one often uses the Burgers equation as a test case for Navier--Stokes: one hopes that by first observing and analyzing some phenomenon in the Burgers equation, that insight can be translated into an understanding of related phenomena in Navier--Stokes.  Metastable solutions in Burgers equation were observed numerically
in the viscous Burgers equation on an unbounded domain \cite{Kim2001} in the so-called ``scaling variables"
\begin{align}\label{eq:scaledBurgers}
\partial_\tau w = \nu \partial_\xi^2w+\frac12\partial_\xi(\xi w)-ww_\xi\qquad w\in \R, \;\nu\ll1.
\end{align}
The scaling variables 
\begin{align*}
\xi = \frac x{\sqrt{1+t}},\qquad\tau = \ln(t+1),\quad\text{and}\quad u(x,t) = \frac1{\sqrt{1+t}}w\left(\frac x{\sqrt{1+t}}, \ln(1+t) \right)
\end{align*}
have been defined so that a diffusion wave--a strictly positive triangular profile which approaches zero for $|x|\to\infty$---is a steady state solution to (\ref{eq:scaledBurgers}) (otherwise, all solutions to Burgers equation in the unscaled variables $\partial_t u = \nu \partial_x^2u-uu_x$ approach the zero solution as $t\to+\infty$). In \cite{Kim2001} the authors observe that ``diffusive N-waves"---profiles with a negative triangular region immediately followed by a positive triangular region so that the profile resembles a lopsided backwards ``N"---quickly emerge before the solution converges to a diffusion wave. 

Burgers equation is much more amenable to analysis than the Navier-Stokes equation and there has
been a fair amount of theoretical work to explain the observations of \cite{Kim2001}.  Already in \cite{Kim2001}, the
authors used the Cole-Hopf transformation to derive an analytical expression for the diffusive
N-waves.   In \cite{Beck2009}, the authors provide a more dynamical systems motivated explanation of
metastability.  First they constructed a center-manifold for (1.2) consisting of the diffusion waves, denoted $A_M(\xi)$, which is parametrized by the solution mass.  Each of these diffusion waves represents the
long-time asymptotic state of all integrable solutions with initial mass $M$ and they are also fixed points
in the scaling variables.  Through each of these fixed points there is a one-dimensional manifold,
parameterized by $\tau$, consisting of exactly the diffusive $N$-waves.  Then, using the Cole-Hopf
transformation, the authors show that solutions converge toward the manifold of $N$-waves on a time scale
of order $\tau = {\cal O}(|\ln \nu||)$, that solutions remain near $w_N(\xi,\tau)$ for all future times, and that the evolution along $w_N(\xi,\tau)$ towards $A_M(\xi)$ is on a time scale of the order $\tau = \calO(1/\nu)$. In particular, convergence to the family is faster than the subsequent evolution along the family. We emphasize that their analysis makes strong use of the Cole-Hopf transformation.

In \cite{Beck2013} the authors proposed an explanation of the metastability of the bar-states of (\ref{eq:NS}) as follows. They first
propose as candidates for the metastable family the exact solutions of the Navier-Stokes equations
with vorticity distribution
\[
\omega^b(x,y,t) = \rme^{-\nu t}\cos(x)\footnote{Alternatively, the bar state could be $\widetilde w^b(x,y,t) = \rme^{-\nu t}\sin(x)$, or the solution could instead be periodic in the $y$ direction and constant in the $x$ direction.},
\]
which is again parametrized by time. Solutions in this family converge 
to the long-time limit (which is the zero solution in this case) on the viscous time scale 
$t \sim \frac{1}{\nu}$. In order to understand the convergence of solutions with nearby initial data to the metastable family, the authors linearize the vorticity formulation of (\ref{eq:NS})
\begin{align}\label{eq:vorticity}
\partial_t\omega = \nu\Delta\omega-\vec u\cdot\nabla\omega, \qquad \vec u = (-\partial_y\Delta^{-1} \omega, \partial_x\Delta^{-1}\omega).
\end{align}
about $\omega^b(x,y,t)$. The linearization results in a nonlocal time-dependent linear operator \[\mathcal L(t) = \nu\Delta-a\rme^{-\nu t}\sin x\partial_y(1+\Delta^{-1}).\] Using hypercoercivity techniques motivated by the work of Villani \cite{Villani2009} and Gallagher, Gallay, and Nier \cite{Gallagher2009}, the authors show that solutions to a modified operator $\mathcal L^a(t) = \nu\Delta-a\rme^{-\nu t}\sin x\partial_y$, which differs from $\mathcal L(t)$ by removing the non-local, but relatively compact, term, decay with rate at least $\rme^{-\sqrt{\nu}t}$. Additionally, they provide numerical evidence that the real part of the least negative eigenvalue for the nonlocal operator $\mathcal L(t)$ is proportional to $\sqrt\nu$. These arguments, in combination with the fact that the rate of decay of solutions to (\ref{eq:vorticity}) to zero is given by the much slower viscous time scale  provides a mathematical explanation for the metastable behavior of the family of bar states. 

What is notable is that the mechanism for metastability as well as the relevant time scales are different in each case \cite{Beck2009} versus \cite{Beck2013}. Thus, the goal of this work is to re-visit the Burgers equation, albeit with periodic boundary conditions so that the boundary conditions are more similar to those of (\ref{eq:NS}), in order to devise a mathematical explanation for metastability which is more easily transferable to Navier--Stokes. To that end, we intentionally avoid the Cole--Hopf transformation and instead use spectral analysis from the linearization about the candidate metastable family. We find that the spectrum, to leading order, does not depend on the viscosity $\nu$, even though our analysis depends on the presence of the viscosity term in the equation (and thus the calculations below do not apply to the inviscid equation). This is in contrast to the results from \cite{Beck2013} for the Navier--Stokes equation in which the rate of approach toward the metastable solutions occurs at a $\nu$ dependent rate, albeit a much faster rate than the $\nu$ dependent time of approach toward the final asymptotic state. More generally, the linear operator that we analyze is not self-adjoint. Such operators arise frequently, for example, in weakly viscous fluid dynamics and we hope that the methods develop in this work could be applied to wide class of non-self-adjoint spectral problems.

From a technical perspective, the linearization about the metastable states leads to a singularly perturbed eigenvalue problem, in which the perturbation parameter is the viscosity $\nu$. Our strategy is to construct the eigenfunction-eigenvalue pairs in each of two spatial scaling regimes (denoted the ``slow" and ``fast" scales) and then to glue the eigenfunction pieces together in an appropriate ``overlap" region (see Figure~\ref{fig:tildePhi} for a schematic representation). We show, in fact, that the eigenvalues are given, to leading order, by the slow-scale eigenvalues; the rigorous ``gluing" of the fast and slow solutions is done with the aid of a Melnikov-like computation which gives the first order correction of the eigenvalues. The use of such Melnikov-like computations for piecing together solutions has a long tradition, generally called Lin's method \cite{Lin90}, which has been applied to the construction of eigenfunctions in, for example, \cite{Sandstede98}. The idea of piecing together slow and fast eigenfunctions in a singularly perturbed eigenvalue problem follows, for example, from \cite{Doelman98}.

It is worth noting another context in which singularly perturbed eigenvalue problems have
arisen in connection with a slightly different type of
metastability, including in variants of Burger's equation.  In
\cite{Ward95, Ward99} metastability refers to the very slow motion of internal layers in
nearly steady states of  reaction diffusion equations and diffusively perturbed conservation laws.
While different in details and physical context, the notion of metastability in these papers is
similar in spirit to our discussion in that it also describes the slow motion along a family
of solutions (in this cases, solutions in which the internal layer occurs at different positions) before
the system reaches its final state.
The motion of those internal layers is explained by an exponentially small shift in the zero eigenvalue of the 
operator describing the  equation linearized about a stationary state.  In contrast, in our problem, the 
zero eigenvalue is unchanged, regardless of which member of the family of metastable solutions
we linearize around, but the remaining eigenvalues (or at least the four additional eigenvalues
that we compute here) undergo exponentially small shifts.

Another recent study of metastability in the Navier--Stokes equation, which is similar to our work in context, but very different in methods
is the study of the inviscid limit of the Navier--Stokes equations in the neighborhood of the Couette flow, by Bedrossian, Masmoudi and
Vicol \cite{Bedrossian15} (see also  \cite{Bedrossian13}). In this paper the authors prove an enhanced stability of the Couette
flow by using carefully chosen energy functionals.  They prove that for times less than $\calO({Re}^{1/3})$, the system approaches
the Couette flow in a way governed by the inviscid limit (i.e. the Euler equations) while for time scales longer than this viscosity effects
dominate; here ${Re}$ is the Reynold's number of the flow.  Since our results show that our metastable family attracts nearby solutions at
a rate which is, to leading order, independent of the viscosity, we believe that they are analogous to the initial phase of the evolution
analyzed in \cite{Bedrossian15} in which inviscid effects dominate.  It would be interesting to see if the transition to viscosity dominated
evolution could be observed in this Burgers equation context as well.
\section{Set-up and statement of main results}\label{sec:setup}
In this section we discuss our candidate family of metastable solutions, denoted $W(x,t;\nu, \Delta x, c)$, to the viscous Burgers equation with periodic boundary conditions
\begin{alignat}{3}\label{eq:Burgers}
\partial_t u =& \nu\partial_x^2u-uu_x\qquad&& \nu\ll 1,\; x\in\R,\; t\in\R^+\nonumber\\
u(x,0) =& u_0(x)\qquad&&u_0\in H_{per}^1([0,2\pi))\nonumber\\
u(x+2\pi, t) =& u(x,t).
\end{alignat}
We also present numerical and analytical justification for our choice. The analytical justification given in Section~\ref{sec:CH} relies, again, heavily on the Cole-Hopf transformation. Thus, although it provides powerful evidence for the behavior of solutions near $W(x,t;\nu, \Delta x, c)$, the result provides no insight into techniques one might use to analyze Navier--Stokes. Thus we provide an alternative explanation which relies on information about the spectrum of the linear operator obtained from linearizing (\ref{eq:Burgers}) about the metastable family $W(x,t;\nu, \Delta x, c)$; the statement and discussion of these results can be found in Sections~\ref{sec:outline} and \ref{sec:justify}.
In what follows we make the technical assumption that the primitive of $(u_0(x)-\overline u)$ attains a unique global maximum on $[0,2\pi)$, where $\overline u =\frac1{2\pi} \int_0^{2\pi} u_0(x)dx$. We remark that this assumption is generic since if the primitive of $u_0(x)$ does not attain a global maximum on $[0,2\pi)$ then for all $\varepsilon>0$ there exists a function $v(x)$ with $\|v\|_{\mathrm H^1_{per}}\le\varepsilon$ such that the primitive of $u_0(x)+v(x)$ does attain a global maximum, where
\[
\|v\|_{\mathrm H^1_{per}}^2 = \int_0^{2\pi} \left[v(x)^2+v'(x)^2\right]dx
\]
is the usual periodic $\mathrm H^1$ norm.

\subsection{Family of metastable solutions}\label{sec:Whit}
It is well known, using the Cole-Hopf transformation, that
\begin{align}\label{eq:CH}
u(x,t) = -2\nu\frac{\psi_x(x,t)}{\psi(x,t)}
\end{align}
is a solution to Burgers on the real line if $\psi(x,t)$ satisfies the heat equation
\begin{align}\label{eq:heat}
\psi_t =& \nu\psi_{xx}\qquad \nu\ll 1,\; x\in\R,\; t\in\R^+.
\end{align}
A family of periodic solutions to (\ref{eq:heat}) can be constructed by placing heat sources on the real line spaced $2\pi$ apart centered at $x=\pi(2n-1)$ 
\begin{align}\label{eq:psi}
\psi^W(x,t; \nu):=\frac 1{\sqrt{4\pi \nu t}}\sum_{n\in\Z}\exp\left[\frac{-(x+\pi-2n\pi)^2}{4\nu t}\right].
\end{align}
Then every function in the family
\begin{align}\label{eq:uW}
W_0(x,t; \nu) :=-2\nu\frac{\psi^W_x}{\psi^W}= \frac 1t \frac{\sum_{n\in\Z}(x+\pi-2n\pi)\exp\left[\frac{-(x+\pi-2n\pi)^2}{4\nu t}\right]}{\sum_{n\in\Z}\exp\left[\frac{-(x+\pi-2n\pi)^2}{4\nu t}\right]}
\end{align}
is $2\pi$-periodic and hence a solution to (\ref{eq:Burgers}). We have denoted solutions (\ref{eq:uW}) by $W_0$ to indicate the fact that one can find them in, for example, the classic text by G.B. Whitham \cite[\S 4.6]{Whitham}. 
Using formula (\ref{eq:uW}) one can check that $W_0(n\pi,t; \nu) =0$ and that $W_0$ is an odd function about $n\pi$, for $n\in\Z$.

The family of solutions (\ref{eq:uW}) is parametrized by $t$. We can extend the family to include two additional parameters as follows. Firstly, we can replace $x$ by $x-\Delta x$, effectively shifting the origin of the $x$-axis.  Next, suppose $u(x,t)$ is a solution to (\ref{eq:Burgers}). Then $u_c(x,t) := c+u(x-ct, t)$ solves (\ref{eq:Burgers}) as well since
\begin{align*}
\partial_t u_c = \partial_t u-c\partial_x u=\nu\partial_x^2 u_c - (u_c-c)\partial_xu_c -c\partial_x u = \nu\partial_x^2 u_c - u_c(u_c)_x.
\end{align*}
Thus we define an extension of (\ref{eq:uW}) by $W(x, t; \nu, \Delta x, c):=c+W_0(x-\Delta x-ct,t; \nu)$. We remark that if $\psi(x,t)$ is periodic,
\[
\int_{-\pi}^\pi -2\nu\partial_x\psi(x,t)dx = 0
\]
and thus, since
\[
\int_{-\pi}^\pi W(x,t;\nu, \Delta x, c)dx = 2\pi c,
\]
$W(x, t; \nu, \Delta x, c)$ can not be obtained via the Cole-Hopf transformation of a periodic function unless $c=0$. 

We will need the following estimates of $W_0$ and its derivatives.
\begin{Proposition}\label{prop:estW}
Fix $\nu > 0$, $0<\varepsilon_0\ll1$. Then there exists $0<C(\varepsilon_0)<\infty$ such that
\begin{align}\label{est:W}
&\sup_{|x|\le\pi}\left|W_0(x, t; \nu) - \frac1t\left[x-\pi\tanh\left(\frac{\pi x}{2\nu t} \right)\right]\right|\le \frac{C(\varepsilon_0)}t\rme^{-1/\nu t }\nonumber\\
&\sup_{|x|\le\pi}\left|\partial_xW_0(x,t; \nu) - \frac1t\left[1-\frac{\pi^2}{2\nu t}\sech^2\left(\frac{\pi x}{2\nu t} \right)\right]\right|\le \frac {C(\varepsilon_0)}{t^2}\rme^{-1/\nu t }\nonumber\\
&\sup_{|x|\le\pi}\left|\partial_tW_0(x,t; \nu) -  \frac1{t^2}\left[-x+\pi\tanh\left(\frac{\pi x}{2\nu t} \right)+\frac{\pi^2 x}{2\nu t}\sech^2\left(\frac{\pi x}{2\nu t} \right)\right]\right|\le\frac {C(\varepsilon_0)}{t^3}\rme^{-1/\nu t }
\end{align}
for all $0<\nu t<\varepsilon_0$.
\end{Proposition}
We remark that since $W_0(x,t;\nu)$ is periodic, these $\mathrm L^\infty$ estimates can be converted into $\mathrm L^p_{per}$ estimates for any $1\le p<\infty$.
\begin{proof}
Due to the fact that $W_0(x,t; \nu)$ is an odd function centered about $x=0$, we prove the estimates for $x\in[0,\pi]$. 
Define
\begin{align*}
S(x,t; \nu):=&-1+\frac{2\sum_{n\in\Z} n\exp\left[\frac{-(x+\pi-2n\pi)^2}{4\nu t}\right]}{\sum_{n\in\Z}\exp\left[\frac{-(x+\pi-2n\pi)^2}{4\nu t}\right]}
\end{align*}
so that
\begin{align*}
W_0(x, t; \nu)=&\frac xt -\frac{\pi}tS(x,t; \nu)\\
\partial_xW_0(x, t; \nu)=&\frac 1t -\frac{\pi}tS_x(x,t; \nu)\\
\partial_tW_0(x, t; \nu)=&-\frac x{t^2} +\frac{\pi}{t^2}S(x,t; \nu)-\frac\pi t S_t(x, t; \nu).
\end{align*}
Thus it remains to estimate $S(x, t; \nu)$ and its derivatives. We factor $\exp[-(x+\pi)^2/4\nu t]$ out of both the numerator and denominator, define
\begin{align}\label{eq:expbounds}
\exp_n(x; t, \nu)&:=\exp\left[-\pi[-nx+n^2\pi-n\pi]/\nu t\right] =\left\{\begin{array}{ccc}
\exp\left[\frac{-\pi n[(n-1)\pi-x]}{\nu t}\right]&:&n\ge 0\\
\exp\left[\frac{\pi n[(-n+1)\pi+x]}{\nu t}\right]&:&n\le 0
\end{array}\right\},
\end{align}
and rearrange to get
\begin{align*}
=&\frac{-\exp\left[\frac{-\pi x}{2\nu t} \right]+\exp\left[\frac{\pi x}{2\nu t} \right]+\exp\left[\frac{-\pi x}{2\nu t} \right]\sum_{n\neq0, 1} (2n-1)\exp_n(x; t, \nu)}{\exp\left[\frac{-\pi x}{2\nu t} \right]+\exp\left[\frac{\pi x}{2\nu t} \right]+\exp\left[\frac{-\pi x}{2\nu t} \right]\sum_{n\neq0,1}\exp_n(x; t, \nu)}\nonumber\\
=&\tanh\left(\frac{\pi x}{2\nu t} \right)+\mathcal R(x;\nu, t)
\end{align*}
where
\begin{align*}
\mathcal R(x;\nu, t):=&
\frac{\exp\left[\frac{-\pi x}{2\nu t} \right]\sum_{n\neq0, 1} \left[2n-1-\tanh\left(\frac{\pi x}{2\nu t} \right)\right]\exp_n(x; t, \nu)}{\exp\left[\frac{-\pi x}{2\nu t} \right]\sum_{n\in\Z}\exp_n(x; t, \nu)}
\end{align*}
Define $r:=\exp\left[-\pi^2/\nu t\right]$; we have that $0\le r<1$ for all $0\le\nu t\le\varepsilon_0$. Then, using (\ref{eq:expbounds}), we see that for all $x\in[0,\pi]$
\[
\exp_{n}(x; t, \nu) \le  r^{|n|} \qquad \forall n\neq 0, 1, 2 
\]
and
\[
\exp\left[\frac{-\pi x}{2\nu t}\right]\exp_2(x; t, \nu) = \exp\left[\frac{-\pi(4\pi-3x)}{2\nu t}\right]\le \exp\left[\frac{-\pi^2}{2\nu t}\right]=r^{1/2}.
\]
Using the fact that the denominator of $\mathcal R$ is greater than or equal to one since it is a sum of positive terms and the leading term
\begin{align*}
\exp\left[\frac{-\pi x}{2\nu t}\right]\exp_1(x;\nu, t) = \exp\left[\frac{\pi x}{2\nu t}\right] \ge 1\qquad\forall x\in[0,\pi],
\end{align*}
we find
\begin{align*}
\left|\mathcal R(x;\nu, t)\right|\le&4r^{1/2}+\exp\left[\frac{-\pi x}{2\nu t} \right]\sum_{n\neq0, 1,2 } 2(|n|+1)r^{|n|}\\
\le&4r^{1/2}+4\frac{r(2-r)}{(1-r)^2}.
\end{align*}
Thus, there exists $0<C(\varepsilon_0)<\infty$ such that $|\mathcal R(x; \nu, t)|\le C(\varepsilon_0)\rme^{-\pi^2/2\nu t}$ for all $0\le \nu t\le\varepsilon_0$ and $x\in[0,\pi]$.
The same transformations and estimates give
\begin{align*}
\left|S_x(x,t; \nu)-\frac{\pi}{2\nu t}\sech^2\left(\frac{\pi x}{2\nu t} \right)\right|\le \frac{C(\varepsilon_0)}t\rme^{-1/\nu t} \quad\text{and}\quad
\left|S_t(x,t; \nu) +\frac{\pi x}{2\nu t^2} \sech^2\left(\frac{\pi x}{2\nu t} \right)\right|\le \frac{C(\varepsilon_0)}{t^2}\rme^{-1/\nu t}
\end{align*}
after potentially making $C(\varepsilon_0)$ larger. 
\end{proof}
\subsection{Solutions via the Cole-Hopf transformation}\label{sec:CH}
Based on our numerical simulations (see Section~\ref{sec:numerics}), we anticipate that solutions to (\ref{eq:Burgers}) rapidly approach a profile in the family $W(x, t; \nu, \Delta x, c)$, and that the specific member in the family that the solution approaches depends on the initial data $u_0(x)$. In Section~\ref{sec:Whit} we discussed the Cole-Hopf transformation but did not take the initial data into account; we address the initial value problem now and show how the initial data can be used to determine which specific profile $W(x, t; \nu, \Delta x, c)$ the solution is expected to approach.

A solution $u(x,t)$ given by the Cole-Hopf transformation (\ref{eq:CH}) will satisfy the Burgers equation on the real line with initial data $u_0(x)$ provided $\psi(x,t)$ satisfies the initial value problem
\begin{alignat}{3}\label{eq:heat2}
\psi_t =& \nu\psi_{xx}\qquad&& \nu\ll 1,\; x\in\R,\; t\in\R^+\nonumber\\
\psi(x,0) = \psi_0(x) =& \rme^{\frac 1{2\nu}F(x; u_0)},&\quad& F(x; u_0):=-\int_0^x u_0(s)\rmd s.
\end{alignat}
Solutions to (\ref{eq:heat2}) can be expressed as a convolution with the heat kernel $G_t:\R\to\R^+$
\begin{align*}
\psi(x,t) =&\int_{-\infty}^\infty \psi_0(y)G_t(x-y)\rmd y= \frac 1{\sqrt{4\pi\nu t}}\int_{-\infty}^\infty \rme^{\frac 1{2\nu}\left[F(y; u_0)-\frac 1{2t}(x-y)^2\right]}\rmd y.
\end{align*}
As was argued in \cite{Pelinovsky2012}, if one additionally assumes that $\int_0^{2\pi} u_0(s)\rmd s=0$ then $\psi_0(x)$ is $2\pi$-periodic, and hence so are $\psi(x,t)$ and
\[
u_0^{CH}(x,t; \nu, u_0) := -2\nu\frac{\psi_x(x,t)}{\psi(x,t)}=\frac1t\frac{\int_{-\infty}^\infty (x-y) \exp\left[\frac 1{2\nu}\left(-\frac {(x-y)^2}{2t}+F(y; u_0)\right)\right]\rmd y}{\int_{-\infty}^\infty \exp\left[\frac 1{2\nu}\left(-\frac {(x-y)^2}{2t}+F(y; u_0)\right)\right]\rmd y}.
\]
Thus $u_0^{CH}(x,t; \nu, u_0)$ is a solution to the periodic problem (\ref{eq:Burgers}) with initial data $u_0^{CH}(x,t; \nu, u_0) = u_0(x)$. We assume that $F(y; u_0)$ has a single global maximum in the interval $y\in[-\pi,\pi)$ located at $y=y_0$
\[
y_0 = \argmax_{y\in[-\pi, \pi]}\left(-\int_0^y u_0(s) \rmd s\right).
\]
Then the solution $u_0^{CH}$ can be estimated as 
\begin{align}\label{est:CH}
u_0^{CH}(x,t;\nu, u_0) = \frac1t\left[x- y_0-\pi-\pi\tanh\left(\frac {\pi(x-y_0-\pi) }{2\nu t}\right)+\calO\left(\sqrt\nu+\frac1t\right)\right],
\end{align}
which can be seen by using, for example, Laplace's method; since the goal of this work is to get away from the Cole-Hopf transformation, we leave the details to the reader. Comparison of (\ref{est:CH}) with (\ref{est:W}) indicates that solutions to (\ref{eq:Burgers}) will asymptotically approach $W_0(x,t; \nu, \Delta x)$, and that $\Delta x$ is close to $y_0+\pi$, where $y_0$ depends on the initial data. If $c:=\frac1{2\pi}\int_0^{2\pi} u_0(s)\rmd s \neq 0$ then
\[
u^{CH}(x,t; \nu, u_0, c) = c+u_0^{CH}(x-ct,t; \nu, u_0-c).
\]

\subsection{Numerical results}\label{sec:numerics}
The discussion in Sections~\ref{sec:Whit} and \ref{sec:CH} indicates that $W(x, t; \nu, \Delta x, c)$ should be our candidate metastable solution. Numerical simulations indicate the same result. We numerically computed solutions to (\ref{eq:Burgers}) in Python using Gudonov's scheme for conservative PDEs. Letting $h=dx$ and $k=dt$, the CFL condition is
\[
k = \mathrm{min}\left\{\frac {\lambda h}{\max[u(x,0)]}, \lambda h^2 \right\}
\]
for $\lambda<1$. We used $\lambda=0.5$. The initial condition $u(x,0)$ was given by 
\[
u(x, 0) = a_0+\sum_{n=1}^m [a_n\sin(jx)+b_n\cos(jx)],
\]
where $m$ is the number of modes and the coefficients $a_n$ were randomly generated. Due to the symmetry of the modes for $j\ge 1$, the mean of $u(x,0)$, denoted $\overline{u(x,0)}$, is given by $a_0$; furthermore, due to the periodic boundary conditions the mean of any solution is preserved since
\begin{align*}
\frac\rmd{\rmd t} \overline u = \int_{-\pi}^\pi u_t\rmd x = \int_{-\pi}^\pi [\nu u_{xx}-uu_x]\rmd x = \left[\nu u_x-\frac12 u^2\right]\bigg|_{-\pi}^\pi = 0.
\end{align*}
The time series for a solution with $a_0=0$ is shown in Figure~\ref{fig:numerics}. We find that $u(x,t)$ rapidly approaches a solution $W_0(x-\Delta x,t;\nu)$, defined in (\ref{eq:uW}); for all future times, the solution converges to zero in a manner resembling the behavior of $W_0(x-\Delta x,t;\nu)$. When $a_0\neq0$ we find that the solution is vertically centered around $a_0$ moves to the left for $a_0<0$ and to the right for $a_0>0$; consistent with the solution \[W(x, t; \nu, \Delta x, a_0):=a_0+W_0(x-\Delta x-a_0t,t; \nu)\] defined immediately before Proposition~\ref{prop:estW}. Although we show only one sample time series here, we ran multiple experiments with different initial conditions; our results indicate that the evolution of a wide class of initial data evolve in a qualitatively similar fashion to that shown in Figure~\ref{fig:numerics}.

\begin{figure}[h]
\centering
\subfloat[\small $t=0.00$]{\includegraphics[width = 0.3\textwidth]{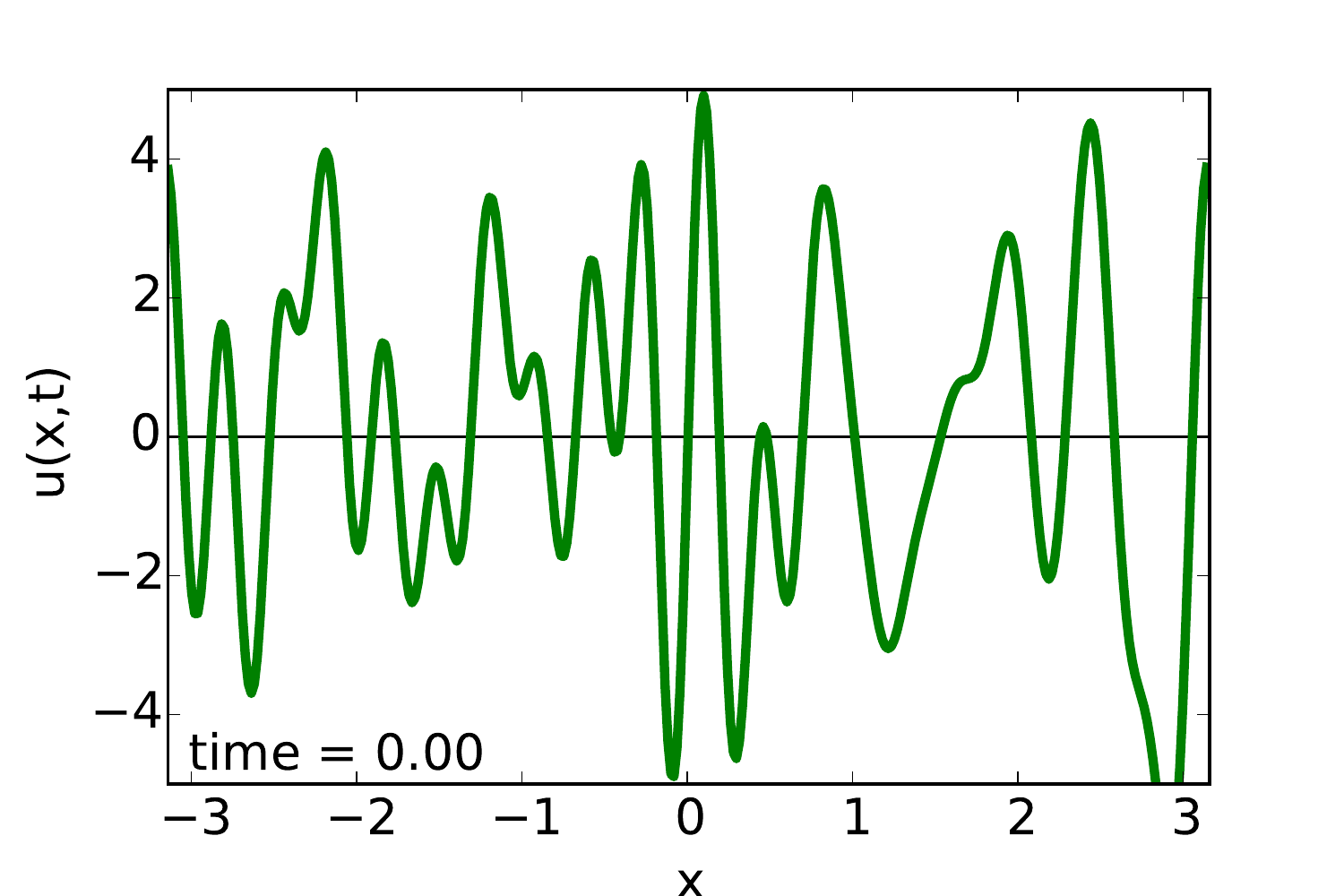}}
\hspace{0.01\textwidth}
\subfloat[\small $t=0.48$]{\includegraphics[width = 0.3\textwidth]{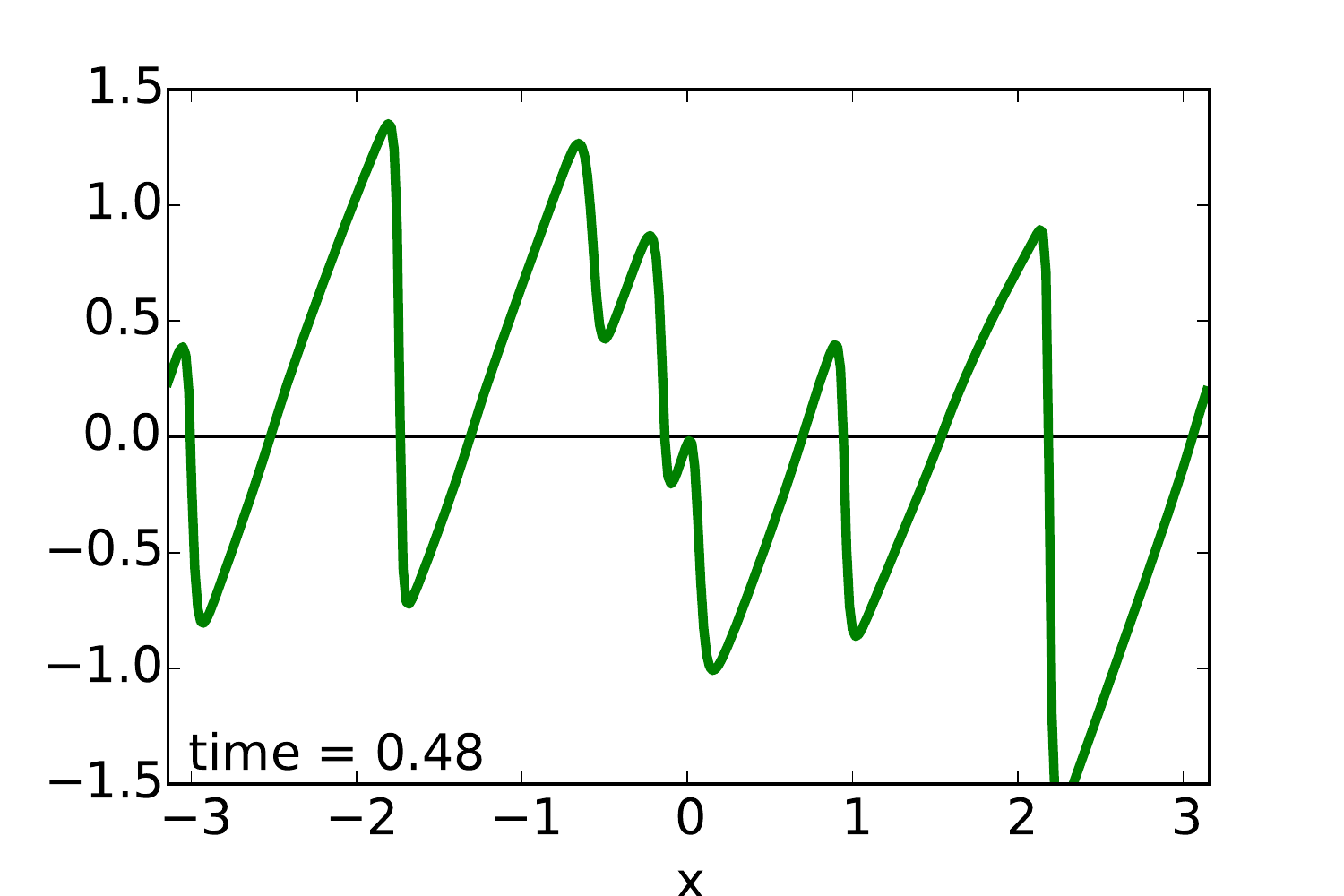}}
\hspace{0.01\textwidth}
\subfloat[\small $t=1.21$]{\includegraphics[width = 0.3\textwidth]{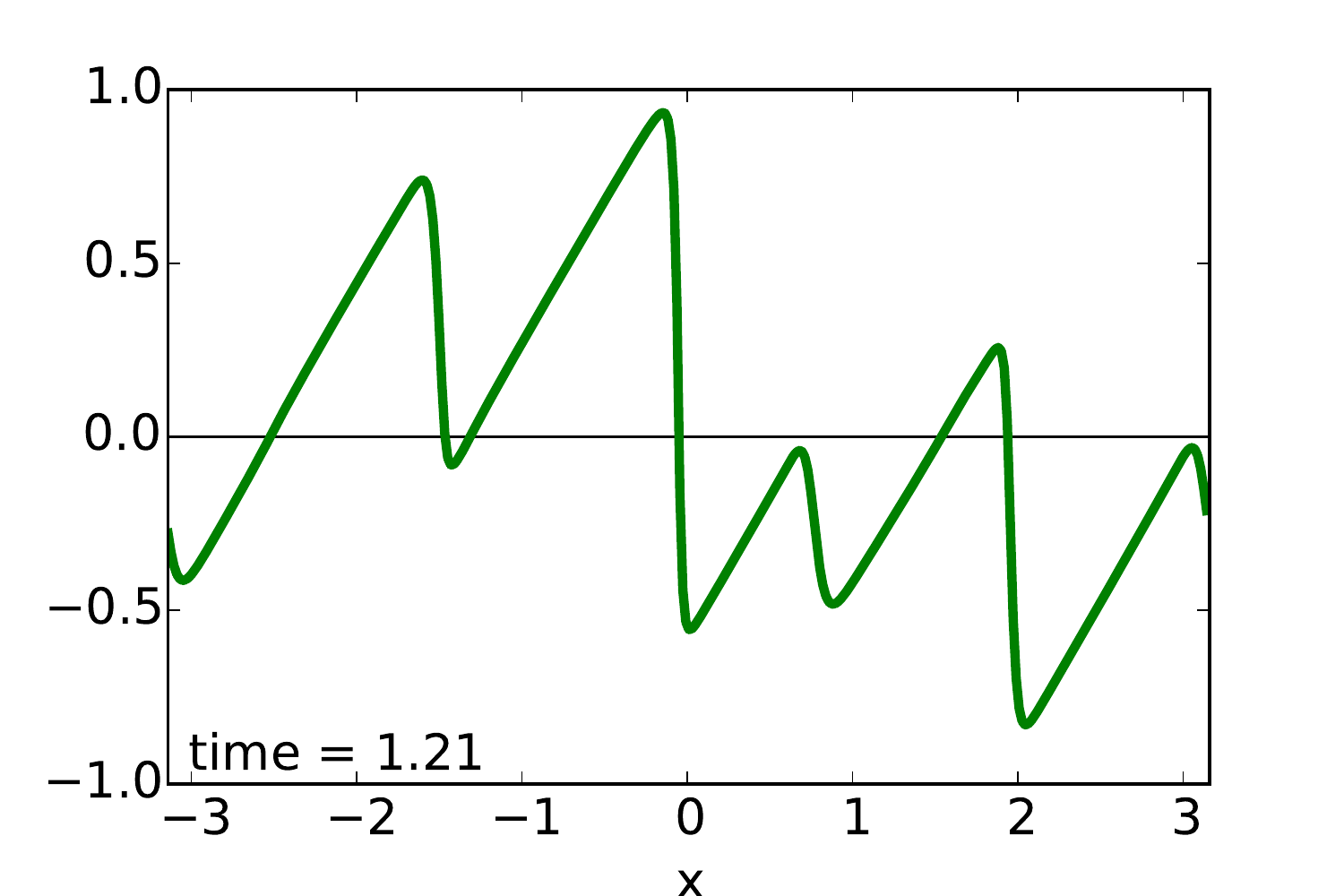}}\\
\subfloat[\small $t=2.42$]{\includegraphics[width = 0.3\textwidth]{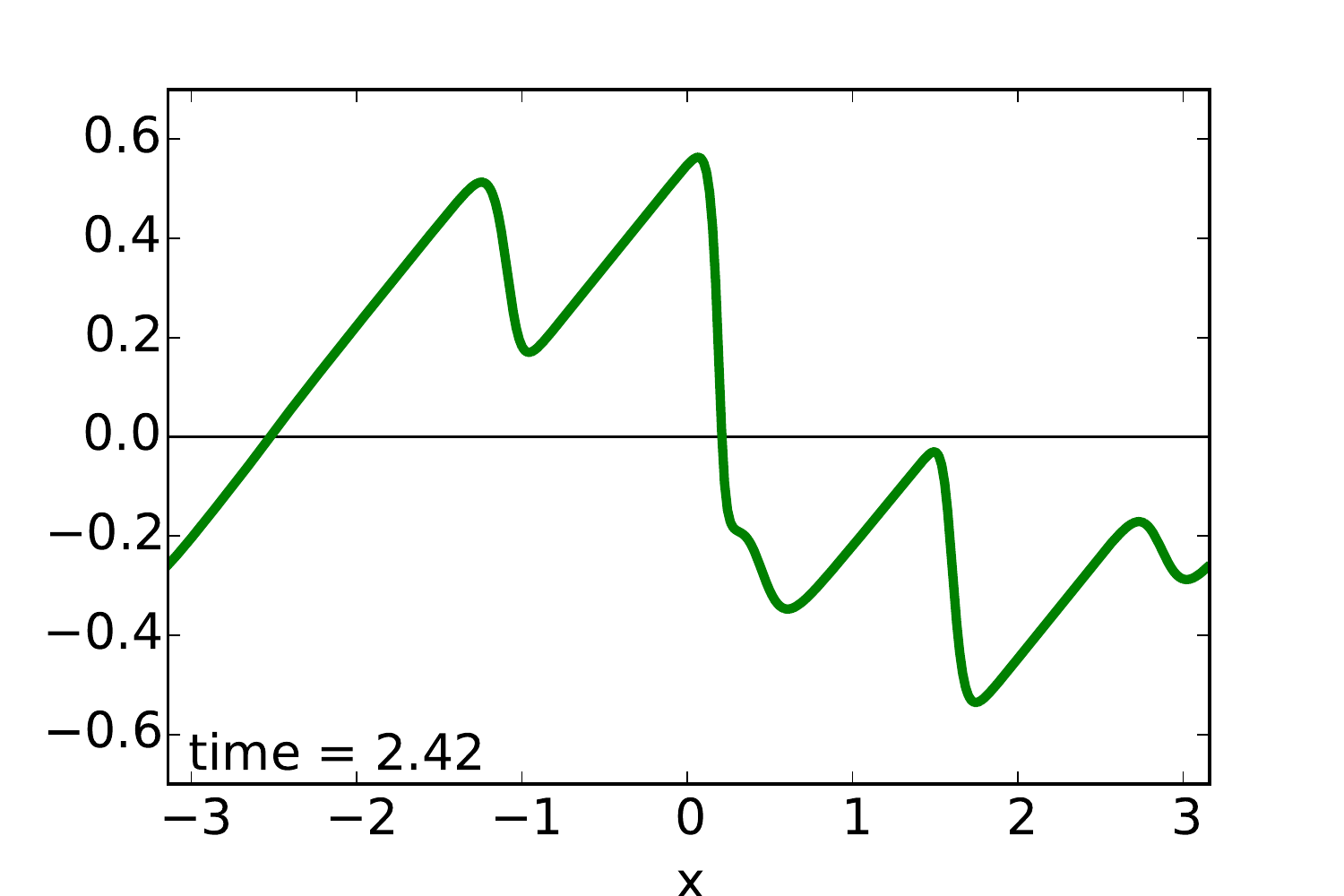}}
\hspace{0.01\textwidth}
\subfloat[\small $t=5.64$]{\includegraphics[width = 0.3\textwidth]{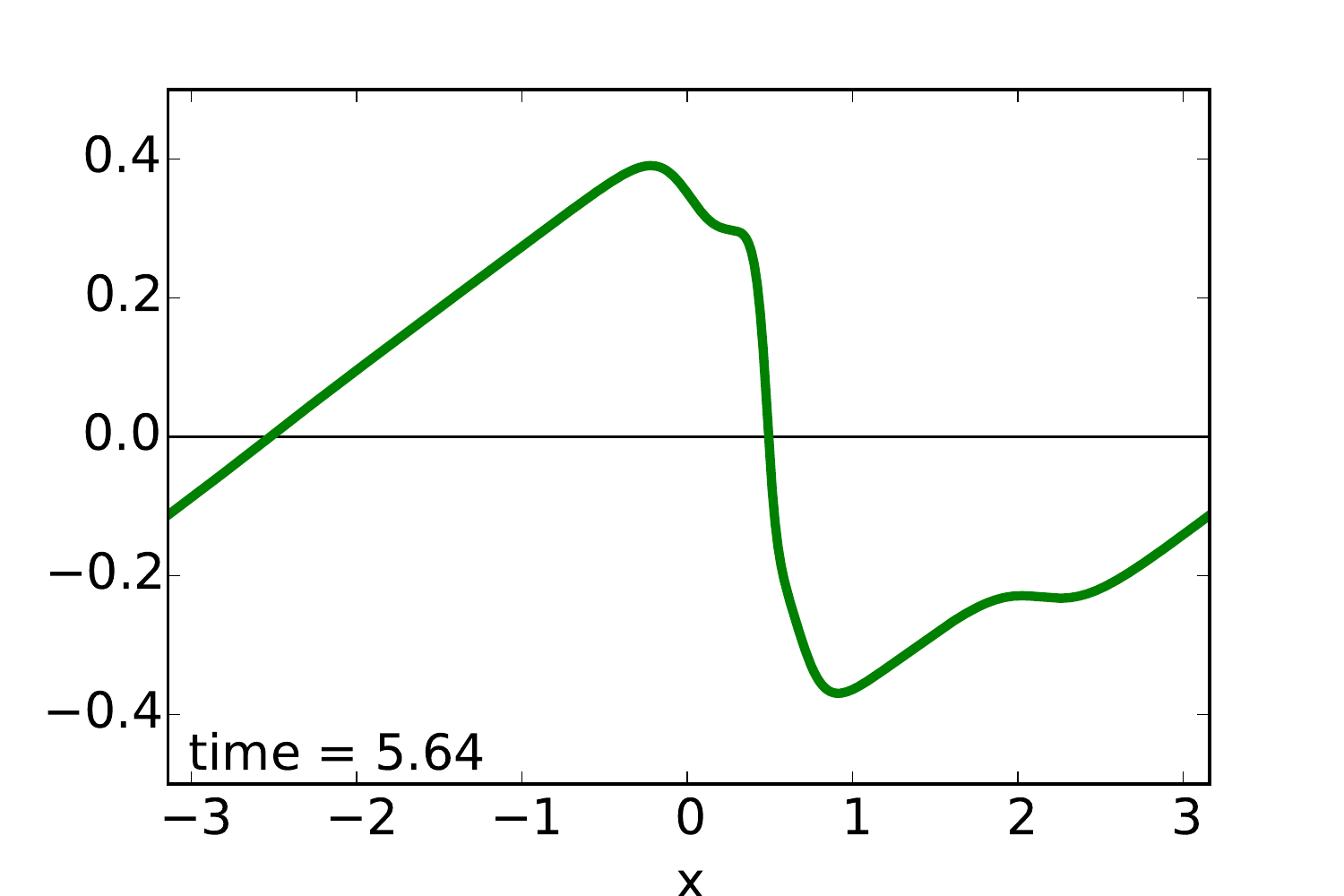}}
\hspace{0.01\textwidth}
\subfloat[\small $t=9.67$]{\includegraphics[width = 0.3\textwidth]{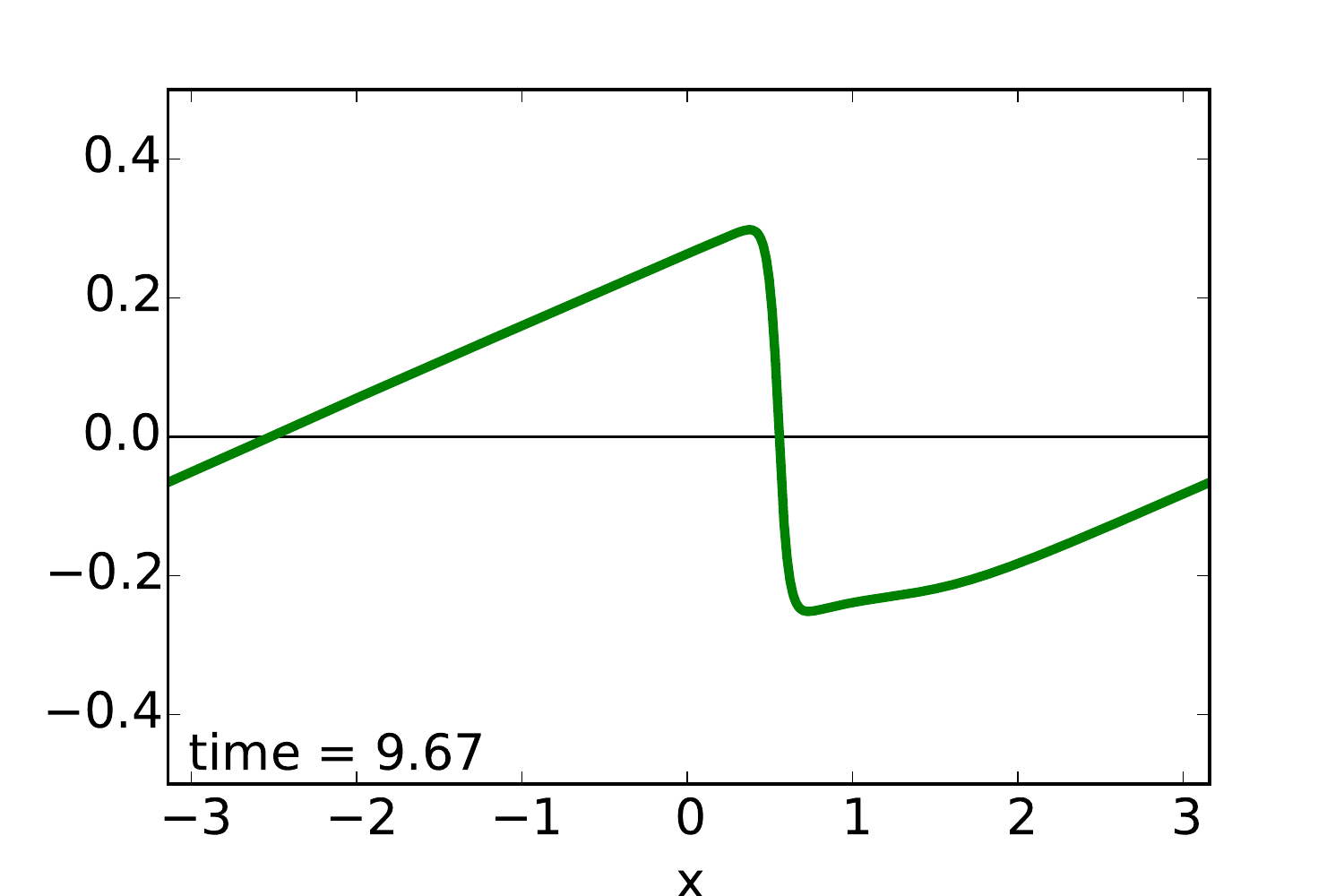}}\\
\subfloat[\small $t=24.17$]{\includegraphics[width = 0.3\textwidth]{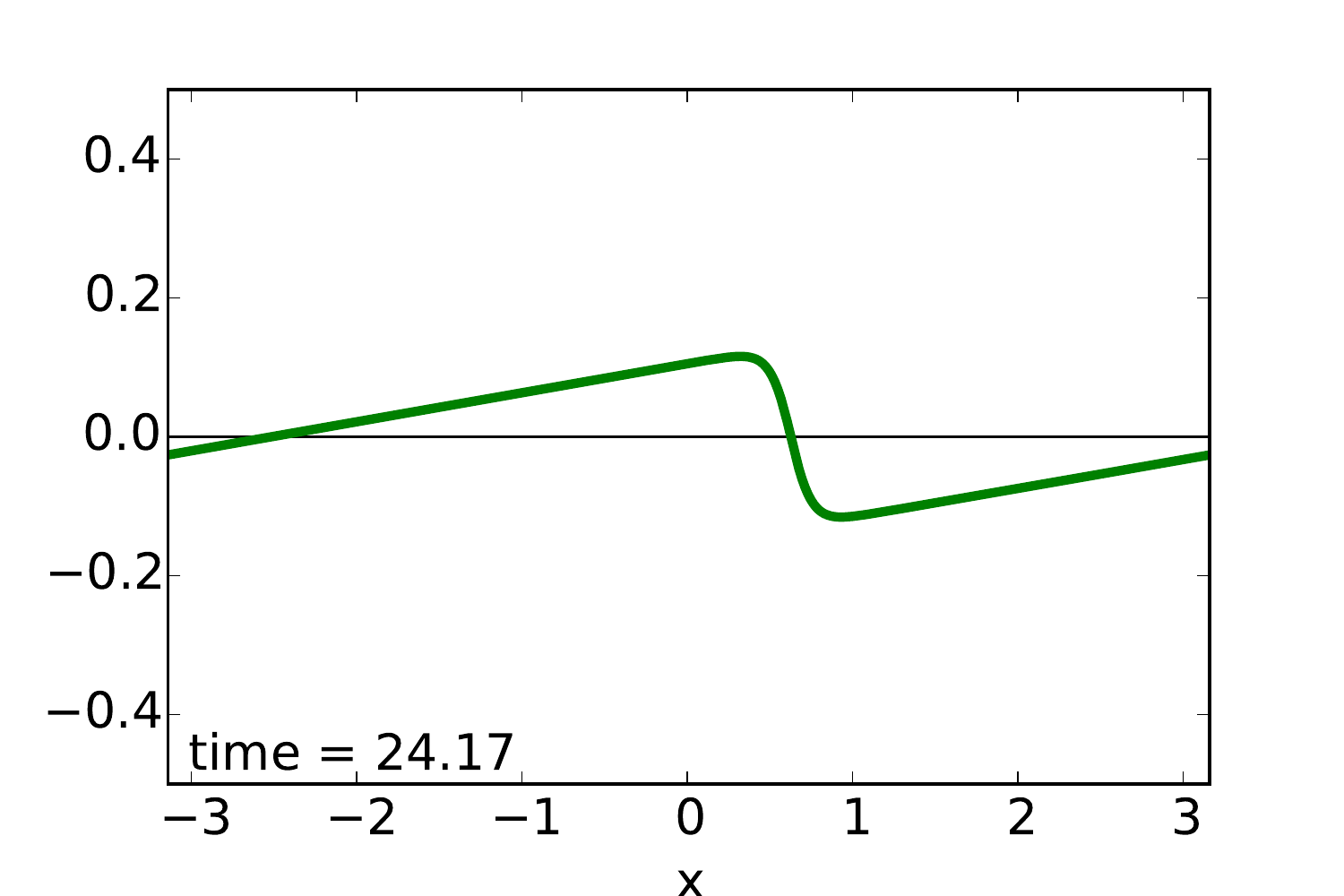}}
\subfloat[\small $t=56.40$]{\includegraphics[width = 0.3\textwidth]{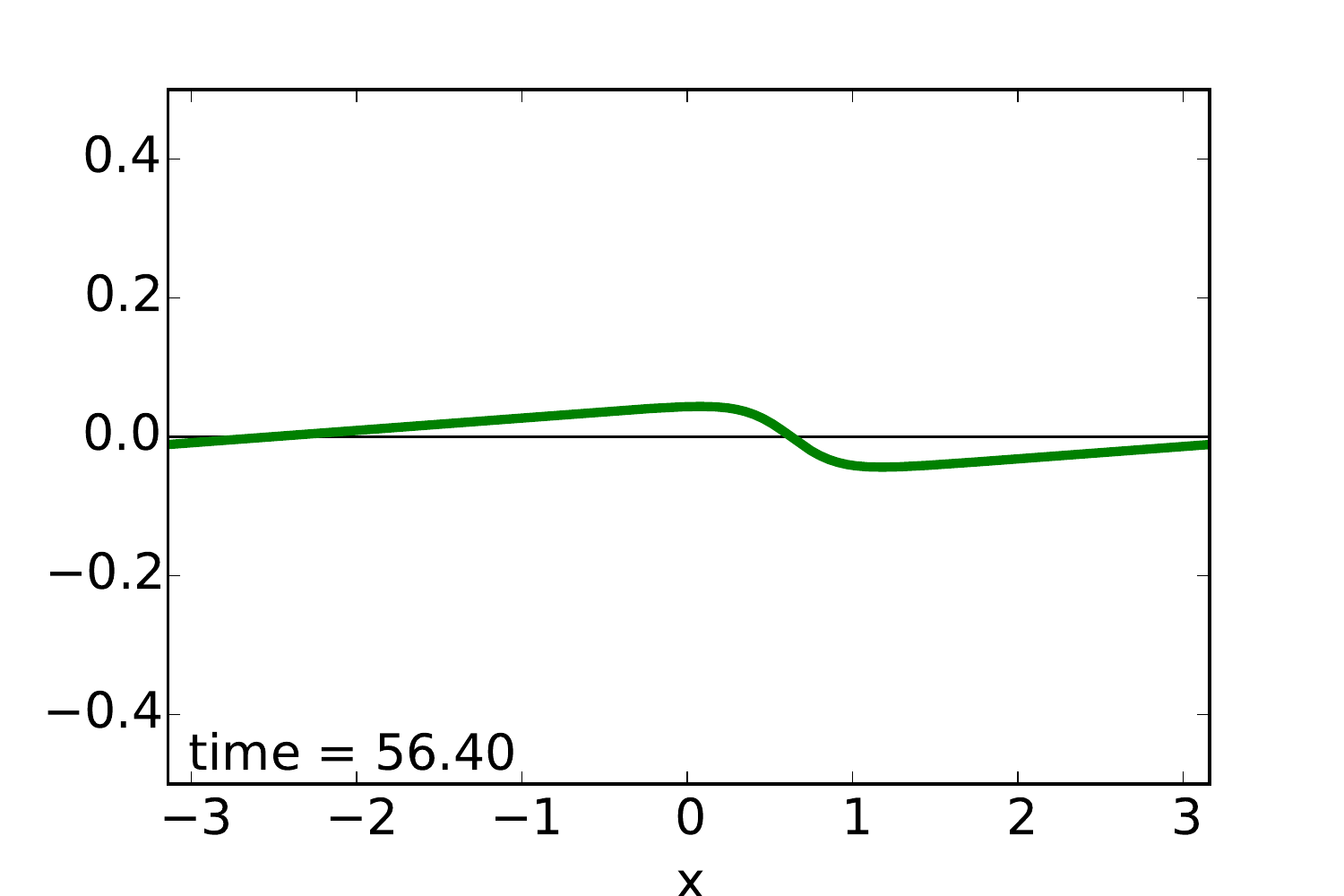}}
\hspace{0.01\textwidth}
\subfloat[\small $t=120.85$]{\includegraphics[width = 0.3\textwidth]{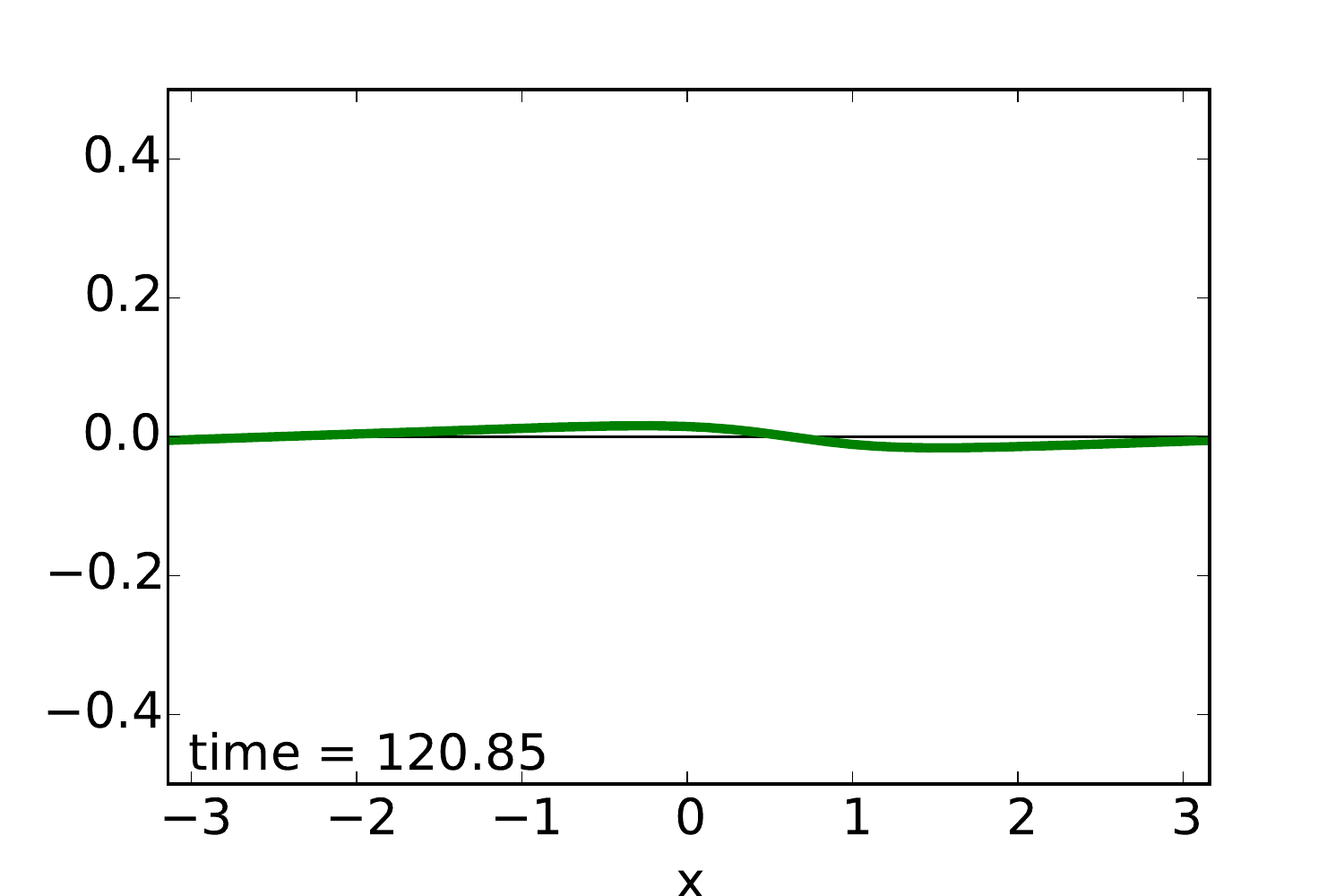}}
\caption{A numerically computed solution to (\ref{eq:Burgers}) with $\nu=0.008$ and random initial data. Solution computed in Python using Gudonov's method with $h=2\pi/350$, CFL constant $\lambda = 0.5$, $m=20$ modes for the random initial data, $\overline{u(x,t)}=a_0=0$, and $y_0:=\argmax_{x\in[-\pi,\pi)} \int^x u(y,0)\rmd y \approx -2.53$. We find that $u(x,t)$ rapidly approaches a solution $W_0(x,t;\nu, \Delta x)$ and then converges to $0$ in a manner consistent with the time evolution of $W_0(x-y_0-\pi,t;\nu)$. Our computations are consistent with the discussion in Sections~\ref{sec:Whit} and \ref{sec:CH}, which indicates that $\Delta x$ should be near $y_0+\pi=0.611$. The scale for (a-d) is not the same as for all other figures. Numerical experiments with different initial data evolved in a qualitatively similar fashion to that shown here.}
\label{fig:numerics}
\end{figure}

\subsection{Statement of the main results}\label{sec:outline}
Our main result concerns the spectrum of the linearization of the viscous Burgers equation about one of the solutions $W(x, t_0; \nu, x_0, c)$ at some time $t=t_0$ fixed. We show that the spectrum is such that solutions of (\ref{eq:Burgers}) which, at $t=t_0$ fixed, are near a member of the metastable family $W(x, t_0; \nu, x_0, c)$ can be expected to approach the family at a much 
faster rate than the solutions $W(x, t; \nu, \Delta x, c)$ themselves evolve in time. Although the linearized evolution is non-autonomous, and thus a rigorous verification of the expected approach rate does not follow directly from the spectral information we derive, we explain why we feel that such rates can nevertheless be expected in Section~\ref{sec:justify} below and in more detail in the discussion Section~\ref{sec:discussion}.

The linearization about $W(x, t; \nu, \Delta x, c)$ in the moving frame $x-\Delta x-ct\mapsto x$ takes the form
\begin{align}\label{eq:linearized}
v_t = \nu v_{xx} -(W_0(x,t; \nu)v)_x,
\end{align}
and the resulting eigenvalue problem is
\begin{align}\label{eq:eprob}
\mathcal L(\nu, t)\varphi =& \lambda\varphi,\quad
\mathcal L(\nu, t) \varphi:=\nu\varphi_{xx} -(W_0(x,t; \nu)\varphi)_x,
\end{align}
where $\mathcal L(\nu, t)$ is considered as an operator $\mathcal L(\nu, t): \mathrm H^2_{per}([-\pi,\pi))\to \mathrm L^2_{per}([-\pi,\pi))$ for every fixed $\nu$ and $t$. We use the standard inner product on $\mathrm L^2_{per}([-\pi,\pi))$
\[
\langle u, v\rangle : = \int_{-\pi}^{\pi} u(x)v(x)\rmd x
\]
and norm $\|u\|^2_{\mathrm L^2_{per}} = \langle u, u\rangle$. 
Motivated by the discussion of the solutions $W(x, t; \nu, \Delta x, c)$ and $u^{CH}(x, t; \nu, u_0, c)$ above we define the small parameter $\varepsilon^2:=2\nu t$. Then our main result is as follows. 
\begin{Theorem}\label{evalThm}
There exists $\varepsilon_0 > 0$ such that for all $\nu$, $t$ such that $0<\varepsilon\le \varepsilon_0$ with $\varepsilon = \sqrt{2\nu t}$, the spectrum for (\ref{eq:eprob}) consists entirely of ordered eigenvalues with $\lambda_0=0$ and the remaining eigenvalues contained on the negative real-axis. In particular,  
\begin{alignat}{3}\label{eq:evalues}
\lambda_1 =& -1/t+\calO\left(\varepsilon^{1/2}\rme^{-1/\varepsilon^2}\right),\quad&\lambda_2 =&-2/t+\calO\left(\varepsilon^{-2}\rme^{-1/\varepsilon^2}\right),\nonumber\\
\lambda_3 =&-3/t+\calO\left(\varepsilon^{-7/2}\rme^{-1/\varepsilon^2}\right),\quad&\lambda_4 =&-4/t+\calO\left(\varepsilon^{-6}\rme^{-1/\varepsilon^2}\right).
\end{alignat}
and $\lambda_n\le \lambda_4$ for all $j> 4$. 
\end{Theorem}
Denoting the eigenfunction associated with $\lambda_n$ by $\varphi_n(x-\Delta x-ct; t, \nu)$
we also show  

\begin{Theorem}\label{convergeThm}
Fix $\gamma_0\ll1$ and let $u(x,t;\nu)$ be a solution to (\ref{eq:Burgers}) with mean $\overline u(x,t;\nu) = c$ so that at some fixed time $t=t_0$ $u(x,t_0; \nu) = W(x, t_0; \nu, x_0, c)+v_0(x; t_0, x_0; \nu)$ with $\|v_0\|_{\mathrm H^2_{per}}=\gamma\le\gamma_0$. Then there exists $x_*$ and $t_*$ such that the projection of $v_*(x; t_*, x_*; \nu):=u(x ,t_0; \nu)- W(x, t_*; \nu, x_*, c)$ onto the space spanned by the first three eigenfunctions for (\ref{eq:eprob}) is zero:
\[
\langle v_*(x; t_*, x_*; \nu), \psi_n(x-x_*-ct_*; t_*, \nu) \rangle= 0 \qquad\text{for}\qquad n=0,1,2,
\]
where $\psi_n$ are the unique functions satisfying $\mathcal L^\dag\psi_n = \lambda_n\psi_n$ and $\mathcal L^\dag$ is the adjoint of $\mathcal L$.
\end{Theorem}
\begin{figure}[h]
\centering
\includegraphics[width = 0.34\textwidth]{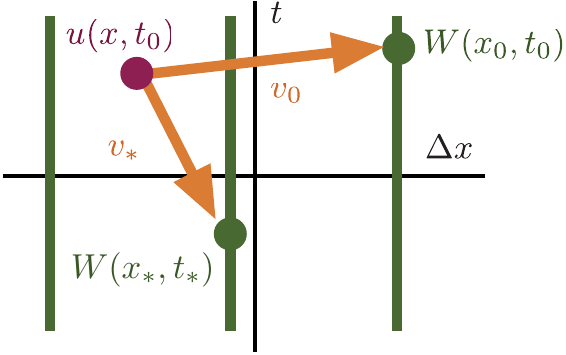}
\caption{$u(x,t;\nu)$ is a solution to (\ref{eq:Burgers}) which at a fixed time $t_0$ is known to be close to a solution $W(x, t_0; \nu, x_0, c)$. We show that by adjusting the parameters $(t_0, x_0)$ slightly we can also write $u(x,t_0;\nu) = W(x,t_*; \nu, x_*, c)+v_*(x; t_*, x_*; \nu)$ where the projection of $v_*$ onto the subspace spanned by the first three eigenfunctions for (\ref{eq:eprob}) is zero.}
\label{fig:vstar}
\end{figure}
See Figure~\ref{fig:vstar}. The inner product $\langle v, w\rangle$ is the standard periodic $\mathrm L^2$ inner product.
\begin{Remark}\label{rmk:smallData}
The discussion in Section~\ref{sec:CH} indicates that the condition $u(x,t_0; \nu) = W(x, t_0; \nu, x_0, c)+v_0(x; t_0, x_0, c; \nu)$ with $\|v_0\|_{\mathrm H^1_{per}}\ll 1$ holds for most initial data provided that $\nu, 1/t \ll1$. 
\end{Remark}
\begin{Remark}
Since (\ref{eq:Burgers}) preserves the mean, by choosing $c$ in $W(x, t_0; \nu, x_0, c)$ so that $\overline u(x,t;\nu) = c$, we ensure that $\overline v_0(x;t_0,x_0;\nu) = 0$ for all time. In the proof of Theorem~\ref{convergeThm} we will show that this implies
\[
\langle v_*(x;t_*,x_*;\nu), \psi_0(x-x_*-ct_*; t_*, \nu) \rangle= 0
\]
independently of $x_*$ and $t_*$. 
\end{Remark}
\subsection{Justification of $W$ as a family of metastable solutions}\label{sec:justify}
Finally, we discuss why the combination of Theorems~\ref{evalThm} and \ref{convergeThm} justifies our identification of the states $W(x,t;\nu, \Delta x,c)$ as a metastable family.
If we attempt to analyze the dynamics of solutions near the metastable family of solutions with the aid of the linearized equation (\ref{eq:linearized}), then the resulting linear equation is non-autonomous and, in general, knowledge about the spectrum of a non-autonomous linearized operator is not sufficient to conclude anything about the linearized evolution.  However, there are examples of parabolic non-autonomous partial differential equations with sufficiently well-behaved nonlinearities for which the ``freezing'' method allows one to estimate the decay rate of solutions in terms of the spectrum of the equations linearized about a solution at a fixed time \cite{Vinograd1983,Promislow2002}. While we have not proven that the freezing method applies to Burgers equations, we feel our results are a first step in rigorously verifying that the frozen time spectrum can serve as a mechanism for understanding the metastable behavior of the family $W(x,t;\nu, \Delta x,c)$ for time of order $\mathcal O(1)$. See the discussion in Section~\ref{sec:discussion} for more details on why we feel the frozen spectrum provides insight into the evolution in this case.

If we think of the spectral picture of the linearized equation (\ref{eq:linearized})
$$
\partial_t v = {\cal L}(\nu,t_0) v = \nu v_{xx} - (W_0(x,t_0;\nu) v)_x\ ,
$$
(where $W_0(x,t_0;\nu)$ is now evaluated at a fixed time $t_0$), then at first glance it looks as if the solutions don't tend toward the family at all, since due to the zero eigenvalue of ${\cal L}(\nu,t_0)$ the linear evolution is not contractive.  However, the point of Theorem~\ref{convergeThm} is that by choosing the parameters $x_*$ and $t_*$ of $W(x,t_*;\nu, x_*,c)$ appropriately, the projection of a solution near $W_0(x,t_0;\nu)$ onto the subspace spanned by the corresponding eigenfunctions $\varphi_n(x-x_*-ct_*; t_*, \nu)$ for $n=0,1,2$ is zero.  Thus, we expect that the linear evolution will 
result in the perturbation decaying toward $W(x,t_*;\nu, x_*,c)$ with a rate governed by third non-zero eigenvalue, which according to Theorem~\ref{evalThm} satisfies
$$
\lambda_3 \approx -\frac{3}{t_0}.
$$
See Figure~\ref{fig:localAttractive}.
Thus, if we write $t = t_0+\tau$ with $t_0 \gg 1$ fixed large enough that $\|v_0\|$ is small as discussed in Remark~\ref{rmk:smallData} and $\tau/t_0 \ll 1$), and then define $p(\tau)$ so that the solution to (\ref{eq:Burgers}) is $u(t_0+\tau) = W(x,t_*;\nu, x_*, c) + p(\tau)$,
 then the size of the perturbation $p(\tau)$ will decay like
$$
\| p(\tau) \|_{\mathrm L^2} \sim \rme^{-\frac{3}{t_0} \tau}\ .
$$
Since
$$
\frac{1}{(t_0 + \tau)^3} = \frac{1}{(t_0)^3 (1+ \tau/t_0)^3} = \frac{\rme^{-3 \ln(1+ \frac{\tau}{t_0})}}{(t_0)^3}
= \frac{\rme^{-\frac{3}{t_0} \tau+\calO\left(\tau^2/t_0^2 \right)}}{(t_0)^3},
$$
for $\tau/t_0$ small enough we have
$$
\| p(\tau ) \|_{\mathrm L^2} \sim \frac{1}{t^3}\ .
$$
Since the evolution along the family behaves like $1/t$, as can be seen from equation (\ref{est:W}) 
\[
W_0(x, t; \nu) = \frac 1t \left[x-\pi\tanh\left(\frac{\pi x}{2\nu t} \right)+\calO\left(\rme^{-1/\nu t }\right)\right],
\]
solutions approach the family at a rate that is much faster than the evolution along the family justifying our characterization of these states as metastable. 
\begin{figure}[h]
\centering
\includegraphics[width = 0.4\textwidth]{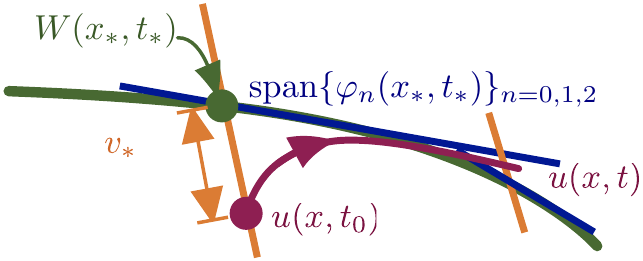}
\caption{A schematic representation for why Theorems~\ref{evalThm} and \ref{convergeThm} indicate that $W(x,t;\nu,\Delta x, c)$ is a metastable family for Burgers equation (\ref{eq:Burgers}). In particular, choosing the initial condition to have projection zero onto the span of $\{\varphi_0,\varphi_1,\varphi_2\}$, the evolution of the semi-flow generated by $\mathcal L(\nu,t_0)$ will contract toward this subspace with a rate $\rme^{-3\tau/t_0}$. For a discussion of why we believe this reflects the decay of the actual linearized evolution, see the discussion in Section~\ref{sec:discussion}.}
\label{fig:localAttractive}
\end{figure}

\section{Eigenvalue problem}\label{sec:evalue}
In this section we prove Theorems~\ref{evalThm} and \ref{convergeThm}. In order to aid understanding of our arguments we have summarized our notation in Tables~\ref{tab:notation}-\ref{tab:notationFast} in Appendix~\ref{app:notation}. Without loss of generality we let $c=0$ and $\Delta x = 0$ (otherwise make the substitution $y= x-\Delta x-ct$). If we consider the eigenvalue equation for the linear operator (\ref{eq:eprob}) with 
$\lambda= 0$ we have
\[
\partial_x^2 \varphi_0 - \frac1\nu (W_0(x,t;\nu) \varphi_0)_x = 0.
\]
Integrating this equation twice we find 
\begin{align}\label{eq:e0}
\varphi_0(x; t, \nu):=\exp\left[\frac1\nu\int^x W_0(s, t; \nu)\rmd s\right]=\frac{C}{\left[\psi^W(x, t; \nu)\right]^2}
\end{align}
is an exact eigenfunction for (\ref{eq:eprob}) with $\lambda=0$, where the function $\psi^W(x,t; \nu)$ was defined in (\ref{eq:psi}). 
To find the rest of the spectrum we define the transformation
\begin{align}\label{eq:tildeTrans}
\varphi(x;t,\nu) = \mathcal T(x;t,\nu)\widetilde\varphi(x; t,\nu)\quad\text{where}\quad \mathcal T(x;t,\nu):=\exp\left[\frac1{2\nu}\int^x W_0(s,t;\nu)\rmd s\right]=\frac {\widetilde C}{\psi^W(x,t;\nu)}
\end{align}
Without loss of generality we choose $\widetilde C=1$. A straightforward computation shows that $\lambda$ is an eigenvalue for (\ref{eq:eprob}) with associated eigenvector $\varphi(x; t, \nu)$ if, and only if, $\lambda$ is an eigenvalue for the self-adjoint problem (\ref{eq:etildeprob}) 
\begin{align}\label{eq:etildeprob}
\widetilde{\mathcal L}(\nu, t)\widetilde\varphi =& \lambda\widetilde\varphi,\quad
\widetilde{\mathcal L}(\nu, t) \widetilde\varphi:=\nu\widetilde\varphi_{xx} -\frac12\left[\partial_xW_0(x, t; \nu)+\frac1{2\nu}W_0^2(x, t; \nu)\right]\widetilde\varphi
\end{align}
with associated eigenfunction $\tilde\varphi$ given by (\ref{eq:tildeTrans}), where we again consider $\widetilde{\mathcal L}(\nu, t)$ as an operator
\[
\widetilde{\mathcal L}(\nu, t): \mathrm H^2_{per}([-\pi,\pi))\to \mathrm L^2_{per}([-\pi,\pi))
\]
for every fixed $\nu$ and $t$. In particular, since the transformation $\varphi\mapsto\widetilde\varphi$ is bounded with bounded inverse, the spectra of $\mathcal L$ and $\widetilde{\mathcal L}$ are identical. Owing to Sturm-Liouville theory for periodic self-adjoint scalar eigenvalue problems (c.f. \cite[Thm 2.1, 2.14]{MangusWinkler}), the eigenvalues for (\ref{eq:etildeprob}) are ordered $\lambda_0 > \lambda_1\ge\lambda_2 > \lambda_3 \ge\lambda_4 > \ldots$. Furthermore, the eigenfunctions $\widetilde\varphi_{2n-1}$ and $\widetilde\varphi_{2n}$ have exactly $2n$ zeros in $x\in[-\pi, \pi)$; since the transformation (\ref{eq:tildeTrans}) is strictly positive, the eigenfunctions $\varphi_{2n-1}$ and $\varphi_{2n}$ for (\ref{eq:eprob}) have exactly $2n$ zeros in $x\in[-\pi, \pi)$ as well. From (\ref{eq:e0}) we see that $\varphi_0(x; t, \nu) > 0$ has no zeros in $x\in[-\pi, \pi)$ since $W$ is continuous; hence, all other eigenvalues $\lambda_n$ are contained on the negative real axis. The following Proposition completes the proof of Theorem~\ref{evalThm}.
\begin{Proposition}\label{evalProp}
Let $\varepsilon:=\sqrt{2\nu t}$. There exists $0<\varepsilon_0\ll 1$ such that for all $\varepsilon\le\varepsilon_0$ the next four eigenvalues for (\ref{eq:etildeprob}) after $\lambda_0 = 0$ are 
\begin{alignat}{3}\label{eq:etildeValues}
\lambda_1 =& -1/t+\calO\left(\varepsilon^{1/2}\rme^{-1/\varepsilon^2}\right),\quad&\lambda_2 =&-2/t+\calO\left(\varepsilon^{-2}\rme^{-1/\varepsilon^2}\right),\nonumber\\
\lambda_3 =&-3/t+\calO\left(\varepsilon^{-7/2}\rme^{-1/\varepsilon^2}\right),\quad&\lambda_4 =&-4/t+\calO\left(\varepsilon^{-6}\rme^{-1/\varepsilon^2}\right).
\end{alignat}
Furthermore, defining $I_s(\varepsilon):=[\varepsilon^{3/2}, 2\pi-\varepsilon^{3/2}]$, $I_f(\varepsilon):= [-\varepsilon^{3/2}, \varepsilon^{3/2}]$, there exists $0<C(\varepsilon_0)<\infty$ such that the following estimates of the first two associated eigenfunctions hold for all $\varepsilon\le\varepsilon_0$
\begin{subequations}\label{eq:tildephi}
\begin{align}
\widetilde\varphi_1:&\left\{ 
\begin{array}{ccc}
\sup_x\left|\rme^{(x-\pi)^2/2\varepsilon^2}\widetilde\varphi_1(x; t,\nu) +1\right| \le C(\varepsilon_0) \varepsilon^{3/2}&: &x\in I_s(\varepsilon) \\
\sup_x\left|\frac{\varepsilon^2}{2\pi^2}\rme^{\pi^2/2\varepsilon^2}\sech\left(\frac{\pi x}{\varepsilon^2}  \right)\widetilde\varphi_1(x; t,\nu)-\left[\sech^2\left(\frac{\pi x}{\varepsilon^2} \right)\left(1+\frac{x^2}{2\varepsilon^2}+\frac{\varepsilon^2}{2\pi^2} \right)-\frac{\varepsilon^2}{2\pi^2}\right] \right|\le C(\varepsilon_0)\varepsilon^{5/2} &: &x\in I_f(\varepsilon)
\end{array}
\right\}\\
\widetilde\varphi_2 : & \left\{ \begin{array}{ccc}
\sup_x\left|\frac{\varepsilon}{x-\pi} \rme^{(x-\pi)^2/2\varepsilon^2}\widetilde\varphi_2(x; t,\nu) +1\right|\le C(\varepsilon_0) \varepsilon & : &x\in I_s(\varepsilon)\\
\sup_x\left|\frac\varepsilon{2\pi} \rme^{\pi^2/2\varepsilon^2}\widetilde\varphi_2(x; t,\nu)-\left[\sinh\left(\frac{\pi x}{\varepsilon^2}\right)+\frac{\pi x}{\varepsilon^2}\sech\left(\frac{\pi x}{\varepsilon^2}\right)\right] \right|\le C(\varepsilon_0)\varepsilon& : &x \in I_f(\varepsilon)
\end{array}\right\}
\end{align}
\end{subequations}
\end{Proposition}
See Figure~\ref{fig:tildePhi} for a representation of $I_s(\varepsilon)$ and $I_f(\varepsilon)$. These intervals $I_{s,f}$ arise naturally from the fact that $\widetilde{\mathcal L}$ is a singularly perturbed operator and we will discuss them in more detail in Section~\ref{sec:overview}. 
In Section~\ref{sec:overview} we provide intuition for Proposition~\ref{evalProp} through a formal matched asymptotic argument. We compute the eigenfunctions $\varphi_n(x; t, \nu)$ associated with each $\lambda_n$ and show that $\varphi_{1,2}(x; t, \nu)$ have two zeros in $x\in[-\pi, \pi)$ and $\varphi_{3,4}(x; t, \nu)$ have four zeros in $x\in[-\pi, \pi)$. For the interested reader we make these arguments rigorous in Section~\ref{sec:rigorous}.

Estimates (\ref{eq:tildephi}) can then be transformed into estimates on the adjoint eigenfunctions for (\ref{eq:eprob}) via (\ref{eq:tildeTrans}) as follows. 
Let $\mathcal L^\dag$ represent the adjoint of $\mathcal L$ and $\psi_n$ its eigenvector associated with $\lambda_n$ so that $\mathcal L^\dag\psi_n = \lambda_n\psi_n$. Using the fact that $\varphi_n=\mathcal T\widetilde\varphi_n$ as described in equation (\ref{eq:tildeTrans}),
$
\widetilde{\mathcal L} = \mathcal T^{-1}\mathcal L\mathcal T
$
and that the operators $\widetilde{\mathcal L}$, $\mathcal T$, and $\mathcal T^{-1}$ are all self-adjoint we find that
$
 \mathcal T\mathcal L^\dag\mathcal T^{-1}\widetilde\varphi_n = \lambda_n\widetilde\varphi_n,
$ or, in other words, $\psi_n = \mathcal T^{-1}\widetilde\varphi_n$. We remark that since $\mathcal T(x;t,\nu)$ is even, $\psi_n(x;t,\nu)$ has the same parity as $\widetilde\varphi_n(x;t,\nu)$. In particular, we will show that $\psi_n$ and $\widetilde\varphi_n$ are even for $n=0,1$ and odd for $n=2$. 

Using the same types of computations as were used to derive (\ref{est:W}) we can derive analogous estimates on the transformation function $\mathcal T(x;t,\nu)=(\psi^W)^{-1}(x,t;\nu)$, namely
\begin{align*}
\mathcal T^{-1}:&\left\{ 
\begin{array}{ccc}
\sup_x\left|\rme^{(x-\pi)^2/2\varepsilon^2}  \mathcal T^{-1}(x;\nu,t) - 1\right|\le C(\varepsilon_0) \rme^{-1/\sqrt\varepsilon} &:& x\in I_s(\varepsilon)\\
\sup_x\left|\frac12\rme^{x^2/2\varepsilon^2}\rme^{\pi^2/2\varepsilon^2}\sech\left(\frac{\pi x}{\varepsilon^2}\right)\mathcal T^{-1}(x;\nu,t)-1 \right| \le C(\varepsilon_0)\rme^{-1/\varepsilon^2}&:&x\in I_f(\varepsilon)
\end{array}
\right\}.
\end{align*}
Thus, the following Proposition is an immediate corollary to Proposition~\ref{evalProp} and the fact that 
\[
\widetilde\varphi_0(x; t, \nu) = \frac1{\psi^W(x, t; \nu)}=\mathcal T(x;\nu,t).
\]
\begin{Proposition}\label{prop:efunc}
Let $\varepsilon:=\sqrt{2\nu t}$. There exists $0<\varepsilon_0\ll 1$ and $0<C(\varepsilon_0)<\infty$ such that for all $\varepsilon\le\varepsilon_0$ the first three eigenfunctions for (\ref{eq:eprob}) are
\begin{subequations}\label{eq:phi}
\begin{align}
\psi_0(&x;t,\nu)= \frac1{\sqrt{2\pi}} \\
\psi_1:&\left\{ 
\begin{array}{ccc}
\sup_x\left|\varepsilon\rme^{(x-\pi)^2/\varepsilon^2}\psi_1(x; t,\nu) +1\right| \le C(\varepsilon_0) \varepsilon^{3/2}&: &x\in I_s(\varepsilon) \\
\sup_x\left|\frac{\varepsilon^3}{4\pi^2}\rme^{\pi^2/\varepsilon^2}\rme^{x^2/2\varepsilon^2}\psi_1(x; t,\nu)-1\right|\le C(\varepsilon_0)\varepsilon &: &x\in I_f(\varepsilon)
\end{array}
\right\}\\
\psi_2 : & \left\{ \begin{array}{ccc}
\sup_x\left|\frac{\varepsilon^3}{x-\pi} \rme^{(x-\pi)^2/\varepsilon^2}\psi_2(x; t,\nu) +1\right|\le C(\varepsilon_0) \varepsilon & : &x\in I_s(\varepsilon)\\
\sup_x\left|\frac{\varepsilon^3}{4\pi}\rme^{\pi^2/\varepsilon^2}\rme^{x^2/2\varepsilon^2}\sech\left(\frac{\pi x}{\varepsilon^2} \right)\psi_2(x; t,\nu)-\left[\sinh\left(\frac{\pi x}{\varepsilon^2}\right)+\frac{\pi x}{\varepsilon^2}\sech\left(\frac{\pi x}{\varepsilon^2}\right)\right] \right|\le C(\varepsilon_0)\varepsilon& : &x \in I_f(\varepsilon)
\end{array}\right\}
\end{align}
\end{subequations}
\end{Proposition}

We remark that in going from Proposition~\ref{evalProp} to Proposition~\ref{prop:efunc} we have introduced a scaling constant which make the Implicit Function Theorem argument in the proof below as simple as possible.
We recall that the eigenfunctions in Proposition~\ref{prop:efunc} are given in the moving frame $x-\Delta x-ct\mapsto x$; thus to get eigenfunctions for the linearization about $W(x, t_0; \nu, x_0, c)$ in a stationary frame we replace $x$ in Proposition~\ref{prop:efunc} with $x-\Delta x-ct$. 

Using Proposition~\ref{prop:efunc} we prove Theorem~\ref{convergeThm}.

\begin{proof}{\bf (of Theorem~\ref{convergeThm})}
We first consider the inner product with $\psi_0(x;t,\nu)$. Since (\ref{eq:Burgers}) preserves the mean of solutions and the mean of $W(x,t;\nu,\Delta x,c)=c$ it is true that the mean $\overline v_0=0$ for all time. Next, using the fact that $v_*$ is given by 
\[
v_*(x; t_*, x_*; \nu) := W(x, t_0; \nu, x_0, c)+v_0(x; t_0, x_0; \nu)- W(x, t_*; \nu, x_*),
\]
we find
\begin{align*}
\langle v_*(x; t_*, x_*; \nu), \psi_0(x-x_*-ct_*; t_*, \nu)\rangle =& \frac1{\sqrt{2\pi}}\int_{-\pi}^\pi v_*(x; t_*, x_*; \nu) \rmd x\\
=&\sqrt{2\pi}\overline v_0 = 0.
\end{align*}

It remains to consider the inner products with $\psi_1$ and $\psi_2$. Let $\Omega\subset\mathrm H_{\mathrm{per}}^2$, $I_1\subset \R$, $I_2\subset\R$ such that $0\in\Omega$, $x_0\in I_1$, and $t_0\in I_2$. We apply the Implicit Function Theorem to $\mathcal F: \Omega\times I_1\times I_2\to\R^2$
\begin{align*}
\mathcal F(v_0; x_*, t_*; \nu, c):=&\pmat{
\langle v_*(x; t_*, x_*; \nu), \psi_1(x-x_*-ct_*; t_*, \nu) \rangle\\
\langle v_*(x; t_*, x_*; \nu), \psi_2(x-x_*-ct_*; t_*, \nu) \rangle
}\\
=&\pmat{
\langle W_0(x-x_0-ct_0, t_0; \nu)-W_0(x-x_*-ct_*, t_*; \nu), \psi_1(x-x_*-ct_*; t_*, \nu) \rangle\\
\langle W_0(x-x_0-ct_0, t_0; \nu)-W_0(x-x_*-ct_*, t_*; \nu), \psi_2(x-x_*-ct_*; t_*, \nu) \rangle
}\\&
+\pmat{
\langle v_0(x, t_0; x_0; \nu), \psi_1(x-x_*-ct_*; t_*, \nu) \rangle\\
\langle v_0(x, t_0; x_0; \nu), \psi_2(x-x_*-ct_*; t_*, \nu) \rangle
}
\end{align*}
and show that $\mathcal F(v_0; x_*,t_*;\nu)=0$ near $(v_0; x_*, t_*) = (0; x_0, t_0)$ for every $\varepsilon:=\sqrt{2\nu t_0}$ small enough. We will show that $\mathcal F$ is uniformly bounded in $\varepsilon$, so that the subspaces $\Omega$, $I_1$, and $I_2$ can be chosen independent of $\varepsilon$. 

Due to Cauchy-Schwartz 
\[
\langle v_0(x; t_0, x_0; \nu), \psi_n(x-x_*-ct_*; t_*, \nu) \rangle \le \|v_0\|_{\mathrm L^2_{per}}\|\psi_n\|\le\|v_0\|_{\mathrm H^1_{per}}.
\]
Thus, $\mathcal F(v_0; x_0, t_0; \nu, c)=0$ for $v_0\equiv0$. In order to show that the matrix
\[
\pmat{|&|\vspace{-0.8em}\\
\frac{\rmd\mathcal F}{\rmd x_*}& \frac{\rmd\mathcal F}{\rmd t_*} \vspace{-0.8em}\\
|&|}\Bigg|_{(x_*, t_*; v_0) = (x_0, t_0; 0)}
\]
is invertible we use the facts that
\begin{align*}
&\frac\rmd{\rmd x_*} \langle W_0(x-x_0-ct_0, t_0; \nu)-W_0(x-x_*-ct_*, t_*; \nu), \psi_n(x-x_*-ct_*; t_*, \nu) \rangle\big|_{(x_*, t_*)=(x_0, t_0)} \\
&\qquad= \langle \left[\partial_xW_0\right](x-x_0-ct_0, t_0; \nu), \psi_n(x-x_0-ct_0; t_0, \nu) \rangle\\
&\frac\rmd{\rmd t_*} \langle W_0(x-x_0-ct_0, t_0; \nu)-W_0(x-x_*-ct_*, t_*; \nu), \psi_n(x-x_*-ct_*; t_*, \nu) \rangle\big|_{(x_*, t_*)=(x_0, t_0)} \\
&\qquad= c\langle \left[\partial_xW_0\right](x-x_0-ct_0, t_0; \nu), \psi_n(x-x_0-ct_0; t_0, \nu) \rangle\\
&\qquad\quad-\langle \left[\partial_tW_0\right](x-x_0-ct_0, t_0; \nu), \psi_n(x-x_0-ct_0; t_0, \nu) \rangle.
\end{align*}
Since $\partial_xW_0(x, t; \nu)$ and $\psi_1(x; t,\nu)$ are even functions and $\partial_tW_0(x, t;\nu)$ and $\psi_2(x; t, \nu)$ are odd functions centered about $x=n\pi$, $n\in \Z$ we have that
\begin{align*}
0&=\langle \left[\partial_tW\right](y, t_0; \nu), \psi_1(y; t_0, \nu) \rangle\\
&=\langle \left[\partial_xW\right](y, t_0; \nu), \psi_2(y; t_0, \nu) \rangle \\
&=\langle 1, \psi_2(y; t_0, \nu) \rangle
\end{align*}
where $y:=x-x_0-ct_0$. In fact, 
\[
0 = \langle 1, \psi_n(y; t_0, \nu) \rangle\qquad\forall j\neq 0
\]
since, integrating the eigenfunction equation (\ref{eq:eprob}) from $y=-\pi$ to $\pi$ and using periodicity we get
\[
0 = \lambda_n\int_{-\pi}^\pi \psi_n(y; t_0, \nu)\rmd y,
\]
where $\lambda_n = 0$ only for $n=0$. Finally,  using
the asymptotic expansions for the derivatives of $W_0(x, t; \nu)$, equations (\ref{est:W}),
\begin{align}\label{est:W2}
\partial_xW_0(x,t; \nu) = & \frac1t\left[1-\frac{\pi^2}{2\nu t}\sech^2\left(\frac{\pi x}{2\nu t} \right)+\calO\left(\frac1t\rme^{-1/\nu t }\right)\right]\nonumber\\
\partial_tW_0(x,t; \nu) = &  \frac1{t^2}\left[-x+\pi\tanh\left(\frac{\pi x}{2\nu t} \right)+\frac{\pi^2 x}{2\nu t}\sech^2\left(\frac{\pi x}{2\nu t} \right)+\calO\left(\frac1t\rme^{-1/\nu t }\right)\right]
\end{align}
we get that
\begin{align}\label{eq:innerprod}
\langle \left[\partial_xW_0\right](y, t_0; \nu), \psi_1(y; t_0, \nu) \rangle =& -\frac{\sqrt\pi}{t_0}\left[1+\calO\left(\varepsilon^{3/2}\right) \right]\quad\text{and}\nonumber\\
\langle \left[\partial_tW_0\right](y, t_0; \nu), \psi_2(y; t_0, \nu) \rangle =& \frac {\sqrt\pi}{2t_0^2} \left[1+\calO\left(\varepsilon\right)\right]
\end{align}
where $\varepsilon=\sqrt{2\nu t_0}$.
We claim that the same scaling holds for the inner products with $v_0$ so that $\mathcal F$ is indeed uniformly bounded for all small $\varepsilon$, which we show at the end of this proof. 

Additionally, using the fact that $v\in \mathrm H^1_{\mathrm{per}}$ and integrating by parts we have 
\begin{align*}
\frac\rmd{\rmd x_*} \langle v_0(x, t_0; x_0; \nu), \psi_n(x-x_*-ct_*; t_*, \nu) \rangle =& -\langle v_0(x, t_0; x_0; \nu), \partial_x\psi_n(x-x_*-ct_*; t_*, \nu) \rangle \\
=& \langle \partial_xv_0(x, t_0; x_0; \nu), \psi_n(x-x_*-ct_*; t_*, \nu) \rangle \\
\le& \|\partial_x v_0\|_{\mathrm L^2_{per}} \le \|v_0\|_{\mathrm H^1_{per}}
\end{align*}
and similarly for the $t_*$ derivative. Thus 
\begin{align*}
&\pmat{|&|\vspace{-0.8em}\\
\frac{\rmd\mathcal F}{\rmd x_*}& \frac{\rmd\mathcal F}{\rmd t_*} \vspace{-0.8em}\\
|&|}\Bigg|_{(x_*, t_*; v_0) = (x_0, t_0; 0)} 
\\
&=\pmat{
\langle \left[\partial_xW_0\right](y), \psi_1(y) \rangle&\langle c\left[\partial_xW_0\right](y)- \left[\partial_tW_0\right](y), \psi_1(y) \rangle\\
\langle \left[\partial_xW_0\right](y), \psi_2(y) \rangle&\langle c\left[\partial_xW_0\right](y)- \left[\partial_tW_0\right](y), \psi_2(y) \rangle}\\
&= \pmat{
 -\frac{\sqrt\pi}{t_0}\left[1+\calO\left(\varepsilon^{3/2}\right) \right]& -\frac{c\sqrt\pi}{t_0}\left[1+\calO\left(\varepsilon^{3/2}\right) \right]\\
0&-\frac {\sqrt\pi}{2t_0^2}\left[1+\calO\left(\varepsilon\right)\right]}\\
&=: A(\varepsilon)
\end{align*}
which is invertible since $\det(A(\varepsilon))=\frac\pi{2t_0^3}\left[1+\calO\left(\varepsilon\right) \right]$ which, for all $\varepsilon$ sufficiently small, is not equal to zero. We observe, in particular, that $\det(A(\varepsilon)) = \calO(1)$, which implies that the difference $\|v_*-v_0\|$ is small for all $\varepsilon\ll1$.

It remains to show that there exists a $C<\infty$ such that $\left|\langle v_0, \psi_1(y; t, \nu) \rangle\right|\le C$ and $\left|\langle v_0, \psi_2(y; t, \nu) \rangle\right|=C$. The first estimate follows from the fact that
\[
\left|\int_{-\pi}^\pi vw\rmd x \right|\le \|v\|_{\mathrm L^\infty} \left|\int_{-\pi}^\pi w\rmd x \right|
\]
and the expansion for $\psi_1$ in Proposition~\ref{prop:efunc}. For the second estimate, we first decompose $v_0=v_{\mathrm{even}}^0+v_{\mathrm{odd}}^0$ into its even an odd components. We note that this is possible since $v_0$ is periodic; in fact
\begin{align*}
v_{\mathrm{even}}^0(x) = \frac12\left(v_0(x)+v_0(2\pi-x) \right) \quad\text{and}\quad v_{\mathrm{odd}}^0(x)=\frac12\left(v_0(x)-v_0(2\pi-x)\right).
\end{align*}
Then
\begin{align*}
\langle v_0, \psi_2(y; t, \nu) \rangle = \langle v_{\mathrm{even}}^0, \psi_2(y; t, \nu) \rangle+\langle v_{\mathrm{odd}}^0, \psi_2(y; t, \nu) \rangle = \langle v_{\mathrm{odd}}^0, \psi_2(y; t, \nu) \rangle.
\end{align*}
Using the expansion for $\psi_2$ given in Proposition~\ref{prop:efunc}, which in particular shows it is exponentially localized near $x=\pi+2n\pi$, we find that there exists a $C<\infty$ such that
\begin{align*}
|\langle v_{\mathrm{odd}}, \psi_2(y; t, \nu) \rangle|\le C\left\|\frac{v_{\mathrm{odd}}^0(x)}{x-\pi}\right\|_{\mathrm L^\infty}\left|\int_0^{2\pi}\frac{(x-\pi)^2}{\varepsilon^3}\rme^{-(x-\pi)^2/\varepsilon^2}\rmd x\right| \le\widetilde C\left\|\frac{v_{\mathrm{odd}}^0(x)}{x-\pi}\right\|_{\mathrm L^\infty}
\end{align*}
for some appropriate $\widetilde C$. Using the fact that
\begin{align*}
\frac{v_{\mathrm{odd}}^0(x)}{x-\pi}=\frac12\frac{v_0(x)-v_0(2\pi-x)}{x-\pi}=\frac12\frac{\int_{2\pi-x}^xv_0'(y)\rmd y}{x-\pi}\le C\|v_0'\|_{\mathrm L^\infty}\le C\|v_0\|_{\mathrm H^2_{\mathrm{per}}}
\end{align*}
we obtain the desired estimate. 
\end{proof}

Thus it remains to prove Proposition~\ref{evalProp}. We give a formal asymptotic analysis argument in Section~\ref{sec:overview}, which provides the intuition behind the relevant scaling. In Section~\ref{sec:rigorous} we prove the proposition rigorously.
\subsection{Overview and formal asymptotics}\label{sec:overview}
In this section we give a formal asymptotic analysis argument to provide intuition for our proof of Proposition~\ref{evalProp} and the form of the eigenfunctions (\ref{eq:tildephi}). The rigorous proof makes up the majority of this work and is given in Section~\ref{sec:rigorous}. We focus on the $n=1$, $2$ cases since all of the technical difficulties arise in these cases. Let $x\in[-\pi, \pi)$; then, using estimates (\ref{est:W}), the definition $\varepsilon^2:=2\nu t$, and formally dropping the higher order $\calO(\rme^{-1/\nu t})$ terms, the eigenfunction problem (\ref{eq:etildeprob}) is
\begin{align}\label{eq:etildeprob2}
\varepsilon^2\partial_{xx}\widetilde\varphi_n -\left[1- \frac{\pi^2}{\varepsilon^2}\sech^2\left(\frac{\pi x}{\varepsilon^2}\right)+\frac1{\varepsilon^2}\left(x-\pi \tanh\left(\frac{\pi x}{\varepsilon^2}\right)\right)^2\right]\widetilde\varphi_n=2t\lambda_n\widetilde\varphi_n.
\end{align}
Let $\widehat\lambda_n=2t\lambda_n$; rescaling space as $\zeta:= x/\varepsilon$ (which, for reasons which will become clear shortly, we call the ``slow scale") regularizes the problem, so that (\ref{eq:etildeprob2}) becomes
\begin{align}\label{eq:eslow}
\partial_{\zeta\zeta}\widehat\varphi_n -\left[1-\frac{\pi^2}{\varepsilon^2}\sech^2\left(\frac{\pi \zeta}\varepsilon\right)+\left(\zeta-\frac\pi\varepsilon\tanh\left(\frac{\pi \zeta}\varepsilon \right)\right)^2\right]\widehat\varphi_n =\widehat\lambda_n\widehat\varphi_n.
\end{align}
The functions $\tanh(\cdot)$ and $\sech(\cdot)$ have highly localized derivatives with
\begin{align*}
\sech\left(y \right)=\calO(\rme^{-y})\quad \text{and}\quad\tanh\left(\pm y\right)=\pm1+\calO(\rme^{-y})\quad\text{for } |y|\sim\infty.
\end{align*}
Thus, for $|\zeta|\in[\sqrt\varepsilon, \pi/\varepsilon]$, the terms $\frac1\varepsilon\sech(\pi\zeta/\varepsilon)$ and $\frac1\varepsilon[\pm1-\tanh(\pi\zeta/\varepsilon)]$ are $\calO(\frac1\varepsilon \rme^{-1/\sqrt\varepsilon})$. Then formally taking the limit $\varepsilon\to0$ of (\ref{eq:eslow}) results in the limiting eigenvalue problem
\begin{align*}
\partial_{\zeta\zeta}\widehat\varphi_n -[1+(\zeta+\pi/\varepsilon)^2]\widehat\varphi_n =&\widehat\lambda_n\widehat\varphi_n,\quad \text{for }\zeta <0\quad\text{and}\nonumber\\
\partial_{\zeta\zeta}\widehat\varphi_n -[1+(\zeta-\pi/\varepsilon)^2]\widehat\varphi_n =& \widehat\lambda_n\widehat\varphi_n,\quad \text{for }\zeta >0.
\end{align*}
We re-center the problem by defining $\xi:=\zeta-\pi/\varepsilon$ and the fact that $\widetilde\varphi_n(x-2\pi) = \widetilde\varphi_n(x)$ to get
\begin{align}\label{eq:eslow0}
\partial_{\xi\xi}\widehat\varphi_n -[1+\xi^2]\widehat\varphi_n =&\widehat\lambda_n\widehat\varphi_n
\end{align}
for $\xi \in  [-\pi/\varepsilon+\sqrt\varepsilon, \pi/\varepsilon-\sqrt\varepsilon]$ (which corresponds with $x\in I_s(\varepsilon)$ in Proposition~\ref{evalProp}). Equation (\ref{eq:eslow0}) has explicit eigenvalues $\widehat\lambda_n = -2n$  with associated eigenfunctions
\begin{align*}
\widehat\varphi_n(\xi) =& H_{n-1}(\xi)\rme^{-\xi^2/2}
\end{align*}
where $H_n(\xi)$ are the physicist's Hermite polynomials, the first few of which are
\[
H_0(y) = 1,\quad H_1(y) = 2y,\quad H_2(y) = 2(2y^2-1),\quad H_3(y) = 4y(2y^2-3).
\]
The slow variables, however, do not capture the behavior of the eigenfunctions for $|\xi|\ll\sqrt\varepsilon$ where the terms $\frac1\varepsilon\sech(\pi\xi/\varepsilon)$ and $\frac1\varepsilon[\pm1-\tanh(\pi\xi/\varepsilon)]$ are non-negligible. On the other hand, introducing the faster space scale $z:=x/\varepsilon^2$ (which we henceforth refer to as the ``fast scale"), equation (\ref{eq:etildeprob}) becomes
\begin{align}\label{eq:efast}
\partial_{zz}\widecheck\varphi_n -\left[\varepsilon^2+\pi^2-2\pi^2\sech^2\left(\pi z\right)+\varepsilon^4z^2-2\pi \varepsilon^2 z\tanh\left(\pi z \right)\right]\widecheck\varphi_n = \varepsilon^2\widehat\lambda_n\widecheck\varphi_n.
\end{align}
Hence, for $z\in[-1/\sqrt\varepsilon, 1/\sqrt\varepsilon]$ (which corresponds with $x\in I_f(\varepsilon)$ in Proposition~\ref{evalProp}), the terms $\varepsilon^2z$ are $\calO(\varepsilon^{3/2})$. Again formally taking the limit $\varepsilon\to0$ results in the limiting eigenvalue problem
\begin{align}\label{eq:efast0}
\partial_{zz}\widecheck\varphi_n +\pi^2[2\sech^2(\pi z)-1]\widecheck\varphi_n =& 0.
\end{align}
Equation (\ref{eq:efast0}) has two linearly independent solutions
\begin{align*}
P(z) =& \sech(\pi z)\quad\text{and}\quad Q(z) = \sinh(\pi z)+\pi z\sech(\pi z).
\end{align*}
We set $\widecheck\varphi_2(z;\widehat\lambda_n) = Q(z)$, anticipating that the fast eigenfunction does not depend, to leading order, on the eigenvalues $\widehat\lambda_n$. As we will show below, however, the matching occurs on the terms which exponentially grow like $\rme^{\pi z}$; thus, since $\sech(\pi z)$ is exponentially decaying, for $\widecheck\varphi_1$ we need to include the $\calO(\varepsilon^2)$ correction so that $\widecheck\varphi_1(z; \widehat\lambda_n)=P(z)+\varepsilon^2P_1(z; \widehat\lambda_n)$ where
\[
P_1(z;\widehat\lambda_n) = \frac{\widehat\lambda_n}{\pi^2}\cosh(\pi z)+\left(\frac{z^2}2+c\right)\sech(\pi z)
\]
solves
\[
\partial_z^2 P_1(z;\widehat\lambda_n) + \pi^2[2\sech^2(\pi z) -1] P_1(z; \widehat\lambda_n) = \left[1+\widehat\lambda_n-2\pi z\tanh(\pi z)  \right]P(z; \widehat\lambda_n).
\]
$P_1(x)$ now includes the exponentially growing term $\cosh(\pi z)$. 
The fast variables are complementary to the slow variables in the sense that now they do not capture the behavior of the eigenfunctions for $|z|\gg1/\sqrt\varepsilon$ where the terms $\varepsilon^2 z$ and $\varepsilon^4 z^2$ are non-negligible.

Our decomposition of the interval $[-\varepsilon^{3/2}, 2\pi-\varepsilon^{3/2}] = I_s(\varepsilon)\cup I_f(\varepsilon)$ now becomes clear. For $x\in I_s(\varepsilon)$, we expect the slow-variable eigenfunctions $\widehat\varphi$ to give a good approximation to $\widetilde\varphi$, whereas for $x\in I_f(\varepsilon)$ we expect the fast-variable eigenfunctions $\widecheck\varphi$ to give a good approximation. See Figure~\ref{fig:tildePhi}. 

We formally construct eigenfunctions $\widetilde\varphi_n(x)$ for (\ref{eq:etildeprob2}) by pasting a slow and a fast solution together; due to symmetry considerations, we glue $\widehat\varphi_n((x-\pi)/\varepsilon)$ with $\widecheck \varphi_1(x/\varepsilon^2; \widehat\lambda_n)$ for $n$ odd and to $\widecheck\varphi_2(z; \widehat\lambda_n)$ for $n$ even. The formal asymptotic analysis procedure is as follows. We add the formal eigenfunctions for (\ref{eq:eslow}) and (\ref{eq:efast}) with relative scaling $C_n$. We determine $C_n$ by requiring $\widehat\varphi_n((x-\pi)/\varepsilon) = C_n\widecheck\varphi_n(x/\varepsilon^2)$ in the overlap region $|x|\approx \varepsilon^{3/2}$. We then subtract the overlap at the matching point $x=\varepsilon^{3/2}$; we define the overlap function $\breve\varphi_n:=\widehat\varphi_n(\sqrt\varepsilon-\pi/\varepsilon) = C_n\widecheck\varphi_n(1/\sqrt\varepsilon)$. We consider $x\in[0,\pi]$; the analysis for $x\in[-\pi,0]$ is completely analogous by symmetry. 
The resulting eigenfunctions are of the form
\begin{align*}
\widetilde\varphi_1(x; t, \nu) =& \rme^{-(x-\pi)^2/2\varepsilon^2}+C_1\left[1+\frac{x^2}{2\varepsilon^2}+\varepsilon^2c\right]\sech\left(\frac{\pi x}{\varepsilon^2}\right)-C_1\frac{\varepsilon^2}{\pi^2}\cosh\left(\frac{\pi x}{\varepsilon^2}\right) - \breve\varphi_1\\
\widetilde\varphi_2(x; t, \nu) =& \frac{x-\pi}{\varepsilon}\rme^{-(x-\pi)^2/2\varepsilon^2}+C_2\sinh\left(\frac{\pi x}{\varepsilon^2}\right) +C_2\frac{\pi x}{\varepsilon^2}\sech\left(\frac{\pi x}{\varepsilon^2}\right) - \breve\varphi_2
\end{align*}
We define the spatial variable
\[
\eta := \frac x{\varepsilon^{3/2}}=\frac\zeta{\sqrt\varepsilon} = \sqrt\varepsilon z
\]
which captures the behavior of $\widetilde\varphi_n$ in the overlap region. 
Then, for $0<\eta=\calO(1)$, the matching conditions $C_n\widehat\varphi_n(x/\varepsilon) = \widecheck\varphi_n(x/\varepsilon^2)$ are 
\begin{align*}
\rme^{-\pi^2/2\varepsilon^2}\rme^{\eta\pi/\sqrt\varepsilon}\rme^{-\varepsilon\eta^2/2} =& C_1\left(1+\frac{\varepsilon\eta^2}2\right)\frac{2}{\rme^{\pi\eta/\sqrt\varepsilon}+\rme^{-\pi\eta/\sqrt\varepsilon}}-C_1\frac{\varepsilon^2}{2\pi^2}\left(\rme^{\pi\eta/\sqrt\varepsilon}+\rme^{-\pi\eta/\sqrt\varepsilon} \right)\\
\frac{\left(\pi+\varepsilon\sqrt\varepsilon\eta \right)}\varepsilon \rme^{-\pi^2/2\varepsilon^2}\rme^{\eta\pi/\sqrt\varepsilon}\rme^{-\varepsilon\eta^2/2} =&\frac{1}{2}\left(\rme^{\pi\eta/\sqrt\varepsilon}-\rme^{-\pi\eta/\sqrt\varepsilon} \right)+C_2\frac{\pi\eta}{\sqrt\varepsilon}\frac{2}{\rme^{\pi\eta/\sqrt\varepsilon}+\rme^{-\pi\eta/\sqrt\varepsilon}}.
\end{align*}
which to leading order becomes
\begin{align*}
\rme^{-\pi^2/2\varepsilon^2}\rme^{\eta\pi/\sqrt\varepsilon}=& -C_1\frac{\varepsilon^2}{2\pi^2}\rme^{\pi\eta/\sqrt\varepsilon}\quad\text{and}\quad \frac{\pi}\varepsilon \rme^{-\pi^2/2\varepsilon^2}\rme^{\eta\pi/\sqrt\varepsilon} =C_2\frac{1}{2}\rme^{\pi\eta/\sqrt\varepsilon}
\end{align*}
and is satisfied by $C_1 = \frac{-2\pi^2}{\varepsilon^2}\rme^{-\pi^2/2\varepsilon^2}$ and $C_2 = \frac{2\pi}{\varepsilon}\rme^{-\pi^2/2\varepsilon^2}$ with overlap
\[
\breve\varphi_1= \rme^{-\pi^2/2\varepsilon^2}\rme^{\pi x/\varepsilon^2}
\quad\text{and}\quad\breve\varphi_2=\frac\pi\varepsilon \rme^{-\pi^2/2\varepsilon^2}\rme^{\pi x/\varepsilon^2}
\]
We emphasize that the matching for both eigenfunctions was done using the coefficients in front of the exponentially growing terms $\rme^{\eta\pi/\sqrt\varepsilon}$ and is why we needed to include the first order correction term in $\widecheck\varphi_1(z)$. Putting everything together, and subtracting the overlap we get
\begin{align*}
\widetilde\varphi_1(x; t, \nu) =& \rme^{-(x-\pi)^2/2\varepsilon^2}-\rme^{-\pi^2/2\varepsilon^2}\left\{\frac{2\pi^2}{\varepsilon^2}\left[1+\frac{x^2}{2\varepsilon^2}+\varepsilon^2c\right]\sech\left(\frac{\pi x}{\varepsilon^2}\right)-2\cosh\left(\frac{\pi x}{\varepsilon^2}\right)\right\}-\rme^{-\pi^2/2\varepsilon^2}\rme^{\pi x/\varepsilon^2}\\
\widetilde\varphi_2(x; t, \nu) =& \frac1\varepsilon\left[ (x-\pi) \rme^{-(x-\pi)^2/2\varepsilon^2}+ 2\pi \rme^{-\pi^2/2\varepsilon^2}\sinh\left(\frac{\pi x}{\varepsilon^2}\right)-\pi \rme^{-\pi^2/2\varepsilon^2}\rme^{\pi x/\varepsilon^2}\right].
\end{align*}
The analysis for $x\in[-\pi,0]$ is completely analogous and the results can be extended to $x\in\R$ by periodicity. The asymptotic results agree with (\ref{eq:tildephi}). A schematic of the resulting eigenfunctions $\widetilde\varphi_1$ through $\widetilde\varphi_4$ is shown in Figure~\ref{fig:tildePhi}. 

\begin{figure}[h]
\centering
\subfloat[\small $\widetilde\varphi_1(x;\varepsilon)$ where $\widehat\varphi_1(\xi)\approx \rme^{-\xi^2/2}$ and $\widecheck\varphi_1(z; \widehat\lambda_1)\approx P(z)+\varepsilon^2P_1(z; -2)$]{\includegraphics[width = 0.45\textwidth]{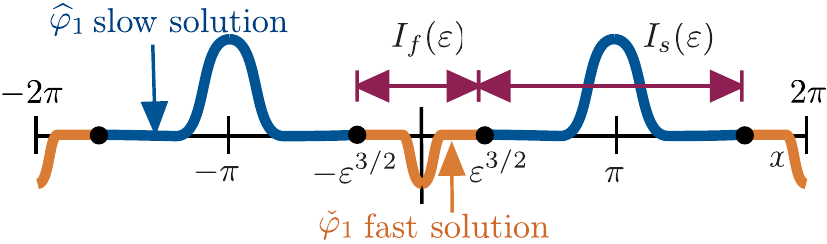}\label{fig:tildePhi1}}
\hspace{0.05\textwidth}
\subfloat[\small $\widetilde\varphi_2(x;\varepsilon)$ where $\widehat\varphi_2(\xi)\approx \xi \rme^{-\xi^2/2}$ and $\widecheck\varphi_2(z; \widehat\lambda_2)\approx Q(z)$]{\includegraphics[width = 0.45\textwidth]{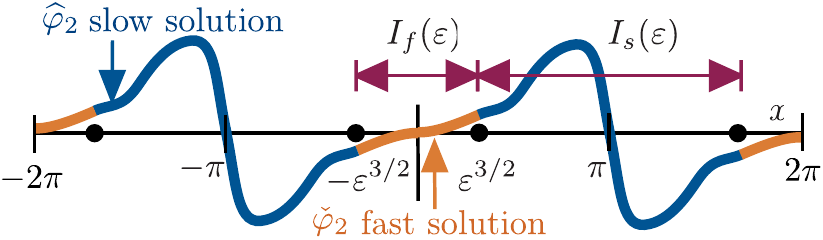}\label{fig:tildePhi2}}\\
\subfloat[\small $\widetilde\varphi_3(x;\varepsilon)$ where $\widehat\varphi_3(\xi)\approx (2\xi^2-1)\rme^{-\xi^2/2}$ and $\widecheck\varphi_1(z; \widehat\lambda_3)\approx P(z)+\varepsilon^2P_1(z; -6)$]{\includegraphics[width = 0.45\textwidth]{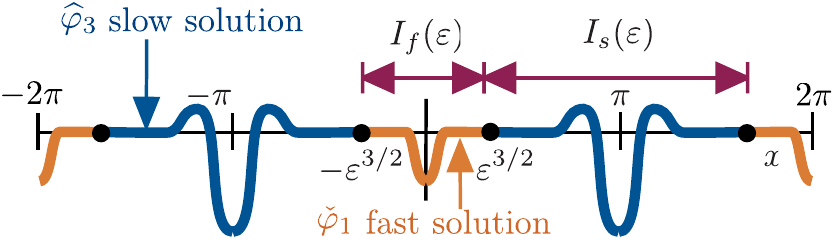}\label{fig:tildePhi3}}
\hspace{0.05\textwidth}
\subfloat[\small $\widetilde\varphi_4(x;\varepsilon)$ where $\widehat\varphi_4(\xi)\approx \xi(2\xi^2-3) \rme^{-\xi^2/2}$ and $\widecheck\varphi_2(z; \widehat\lambda_4)\approx Q(z)$]{\includegraphics[width = 0.45\textwidth]{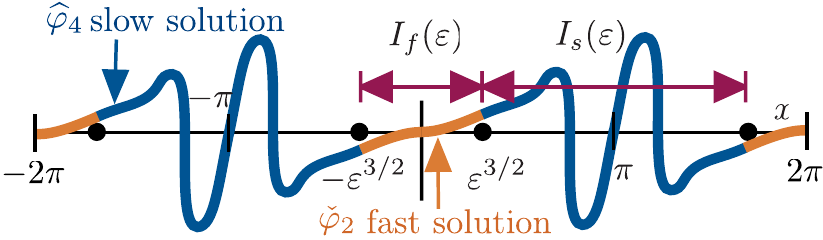}\label{fig:tildePhi4}}\\
\caption{Eigenfunctions for (\ref{eq:etildeprob}) constructed by gluing a slow solution $\widehat\varphi_n$ to a fast solution $\widecheck\varphi_n$. Due to symmetry considerations, we glue $\widehat\varphi_n$ to $\widecheck\varphi_1$ for $n$ odd and to $\widecheck\varphi_2$ for $n$ even. Figures not drawn to scale; in fact, the magnitude of $\widecheck\varphi_n$ is exponentially small relative to the magnitude of $\widehat\varphi_n$.}
\label{fig:tildePhi}
\end{figure}

We make a few observations. First, to leading order, the eigenvalues $\lambda_n=\widehat\lambda_n/2t = -n/t$ are given by the slow eigenvalue problem (\ref{eq:eslow}). Secondly, the contribution to $\widetilde\varphi_n(x)$ from the fast eigenfunctions $\widecheck\varphi_n(x/\varepsilon^2)$ is exponentially smaller than the contribution from the slow eigenfunctions $\widehat\varphi_n(x/\varepsilon)$. However, as we have already remarked, undoing transformation (\ref{eq:tildeTrans}), which is exponentially localized in $x\in I_f(\varepsilon)$, the behavior of eigenfunctions (\ref{eq:phi}) for (\ref{eq:eprob}) in $x\in I_f(\varepsilon)$ becomes relevant. Thus it is essential that we carefully construct the eigenfunctions in both the slow and the fast variables. 

In Sections~\ref{sec:slow}-\ref{sec:matching} we make the above formal arguments rigorous by computing the eigenfunctions for (\ref{eq:etildeprob}). 
In Sections~\ref{sec:slow} and \ref{sec:fast} we rigorously compute the eigenfunction in each of the spatial regimes, $I_s(\varepsilon)$ and $I_f(\varepsilon)$ respectively, using the spatial scaling motivated by the arguments above. We then rigorously match these solutions at the overlap point $x=\pm \varepsilon^{3/2}$ in Section~\ref{sec:matching}.
\section{Rigorous analysis of the eigenvalue problem}\label{sec:rigorous}
In Section~\ref{sec:overview} we provided a formal matched asymptotic analysis argument which gives the intuition behind Proposition~\ref{evalProp}, the key proposition for the proof of Theorems~\ref{evalThm} and \ref{convergeThm}. We anticipate that many readers will find the formal arguments sufficient. However, for the interested reader we provide in this section the rigorous analysis which shows that the results in Proposition~\ref{evalProp} are indeed valid. The proof of this result is technical and relies on many notations. In order to aid understanding of our arguments we have summarized our notation in Tables~\ref{tab:notation}-\ref{tab:notationFast} in Appendix~\ref{app:notation}.
\subsection{Slow variables}\label{sec:slow}
In this section we compute the eigenfunctions for (\ref{eq:etildeprob}) for $x\in I_s(\varepsilon)$. Motivated by the formal asymptotic analysis in Section~\ref{sec:overview} we define the slow variable $\xi:=(x-\pi)/\varepsilon$. We call the eigenfunctions in these coordinates $\widehat\varphi_n(\xi)$; they are defined for $\xi\in[-\pi/\varepsilon+\varepsilon^{1/2}, \pi/\varepsilon - \varepsilon^{1/2}]=:\widehat I_s(\varepsilon)$ and satisfy 
\begin{align}\label{eq:eslow2}
\partial_{\xi\xi}\widehat\varphi_n -\left[\widehat W_\xi(\xi; \varepsilon)+\widehat W^2(\xi; \varepsilon)\right]\widehat\varphi_n =\widehat\lambda_n\widehat\varphi_n
\end{align}
where $\widehat\lambda_n:=2t\lambda_n$ and for any $t\in\R^+$
\begin{align}\label{eq:Nhat}
\widehat W(\xi; \varepsilon):=&\frac t{\varepsilon}W_0(\varepsilon\xi+\pi, t;\nu) = \left[\xi-\frac{2\pi}\varepsilon\frac{\sum_{n\in\Z}n\widehat\exp_n(\xi; \varepsilon)}{\sum_{n\in\Z}\widehat\exp_n(\xi; \varepsilon)}\right],\nonumber\\
\widehat W_\xi(\xi; \varepsilon):=&t\left[\partial_xW_0\right](\varepsilon\xi+\pi, t; \nu)=\left[1-\frac{4\pi^2}{\varepsilon^2}\left(\frac{\sum_{n\in\Z}n^2\widehat\exp_n(\xi; \varepsilon)}{\sum_{n\in\Z}\widehat\exp_n(\xi; \varepsilon)}-\left(\frac{\sum_{n\in\Z}n\widehat\exp_n(\xi; \varepsilon)}{\sum_{n\in\Z}\widehat\exp_n(\xi; \varepsilon)}\right)^2 \right)\right],\nonumber\\
\text{and } \qquad \widehat\exp_n(\xi; \varepsilon):=&\Bigg\{\begin{array}{rcl}
\exp[-2n\pi(n\pi-\varepsilon\xi)/\varepsilon^2] &:& n\ge 0\\
\exp[2n\pi(-n\pi+\varepsilon\xi)/\varepsilon^2] &:& n\le 0
\end{array}
\end{align}
The form of $\widehat\exp_n(\xi; \varepsilon)$ follows from the same type of computations as for (\ref{eq:expbounds}) in Proposition~\ref{prop:estW}
\begin{align*}
\exp\left[\frac{-(\varepsilon\xi+2\pi-2n\pi)^2}{2\varepsilon^2}\right]=\exp\left[\frac{-(\varepsilon\xi-2\pi(n-1))^2}{2\varepsilon^2}\right]=\exp\left[\frac{-\xi^2}2\right]\exp\left[\frac{-2\pi(-\varepsilon\xi(n-1) +(n-1)^2)}{\varepsilon^2}\right],
\end{align*}
factoring out the dominant mode $\exp[-\xi^2/2]$ from the numerator and denominator and shifting $n$. We remark that even though $\widehat W_\xi(\xi;\varepsilon)$ is determined by an appropriate transformation of $\partial_xW_0(x, t;\nu)$, it is also true that $\partial_\xi \widehat W(\xi;\varepsilon) = \widehat W_\xi(\xi; \varepsilon)$; hence our notation. 

Motivated by the formal analysis we re-write (\ref{eq:eslow2}) as
\begin{align*}
\partial_{\xi\xi}\widehat\varphi_n -\left[1+\xi^2+\widehat{\mathcal N}(\xi; \varepsilon)\right]\widehat\varphi_n =(-2n+\widehat\Lambda_n)\widehat\varphi_n
\end{align*}
with $\widehat\Lambda_n:=\widehat\lambda_n +2n$ and $\widehat{\mathcal N}(\xi; \varepsilon):=\widehat W_\xi(\xi; \varepsilon)+\widehat W^2(\xi; \varepsilon)-(1+\xi^2)$,
which is equivalent to the first order system
\begin{align}\label{eq:slowFirst}
\partial_\xi \widehat U_n = \widehat{\mathcal A}_n(\xi)\widehat U_n+\widehat{\mathcal N}_n(\widehat U_n, \xi; \varepsilon, \widehat\Lambda_n)
\end{align}
where $\widehat U_n:=(\widehat \varphi_n, \widehat \psi_n)^T$ with $\widehat\psi_n:=\partial_\xi\widehat\varphi_n$,
\[
\quad \widehat{\mathcal A}_n:= \pmat{0&1\\1+\xi^2-2n&0}, \quad\text{and}\quad \widehat{\mathcal N}_n(\widehat \varphi_n, \widehat \psi_n, \xi; \varepsilon, \widehat\Lambda_n):=\pmat{0\\\left(\widehat{\mathcal N}(\xi; \varepsilon)+\widehat\Lambda_n\right)\widehat \varphi_n}.
\]
\begin{Lemma}\label{lemma:Nhat}
Fix $\widehat\varepsilon_1>0$. There exists $0<\widehat C(\widehat\varepsilon_1)<\infty$ such that for all $\varepsilon\le\widehat\varepsilon_1$ and $\xi\in\widehat I_s(\varepsilon)$,
\begin{subequations}
\begin{align}
\left|\widehat{\mathcal N}(\xi; \varepsilon)\right|\le& \frac {\widehat C(\widehat\varepsilon_1)}{\varepsilon^2}\exp[-\pi^2/\varepsilon^2]\exp[-(\pi-\varepsilon\xi)^2/\varepsilon^2]\exp[\xi^2] \label{eq:Nslow1}\\
\le&\frac {\widehat C(\widehat\varepsilon_1)}{\varepsilon^2}\exp[-2\pi/\sqrt\varepsilon].\label{eq:Nslow2}
\end{align}
\end{subequations}
\end{Lemma}
\begin{proof}
Define $r:=\exp[-2\pi(\pi-\varepsilon|\xi|)/\varepsilon^2]$. Then, due to (\ref{eq:Nhat}), $0<\widehat\exp_n(\xi; \varepsilon) \le r^{|n|}$ with $r\le\exp[-2\pi/\sqrt\varepsilon]<1$; furthermore, since $\widehat{\exp}_0(\xi;\varepsilon)=1$ for all $\xi$ and $\varepsilon$, $\sum_{n\in\Z}\widehat\exp_n(\xi; \varepsilon)\ge 1$. Thus there exists $0<\widehat C(\widehat\varepsilon_1)<\infty$ such that for all $\varepsilon\le\widehat\varepsilon_1$
\begin{align*}
\left|\widehat{\mathcal N}(\xi; \varepsilon)\right| =&\left|\widehat W_x(\xi; \varepsilon)+\widehat W^2(\xi; \varepsilon)-(1+\xi^2)\right|\nonumber\\
=&\left|\frac{8\pi^2}{\varepsilon^2}\left(\frac{\sum_{n\in\Z}n\widehat\exp_n(\xi; \varepsilon)}{\sum_{n\in\Z}\widehat\exp_n(\xi; \varepsilon)}\right)^2-\frac{4\pi^2}{\varepsilon^2}\frac{\sum_{n\in\Z}n^2\widehat\exp_n(\xi; \varepsilon)}{\sum_{n\in\Z}\widehat\exp_n(\xi; \varepsilon)}-\frac{4\pi\xi}\varepsilon\frac{\sum_{n\in\Z}n\widehat\exp_n(\xi; \varepsilon)}{\sum_{n\in\Z}\widehat\exp_n(\xi; \varepsilon)}\right|\nonumber\\
\le& \frac{4\pi}{\varepsilon^2} \left[2\left(\sum_{n\in\Z}|n|r^{|n|}\right)^2+\sum_{n\in\Z}n^2r^{|n|}+\varepsilon|\xi|\sum_{n\in\Z}|n|r^{|n|}\right]\nonumber\\
\le& \frac{4\pi}{\varepsilon^2} \left[2\left(\frac{2r}{(1-r)^2}\right)^2 + \frac{2r(1+r)}{(1-r)^3} + \varepsilon|\xi|\frac{2r}{(1-r)^2} \right]\nonumber\\
\le&\frac{\widehat C(\widehat\varepsilon_1)r}{\varepsilon^2}= \frac {\widehat C(\widehat\varepsilon_1)}{\varepsilon^2}\exp[-2\pi^2/\varepsilon^2]\exp[2\pi\xi/\varepsilon]\nonumber\\
=& \frac {\widehat C(\widehat\varepsilon_1)}{\varepsilon^2}\exp[-\pi^2/\varepsilon^2]\exp[-(\pi-\varepsilon\xi)^2/\varepsilon^2]\exp[\xi^2] \\
\le&\frac {\widehat C(\widehat\varepsilon_1)}{\varepsilon^2}\exp[-2\pi/\sqrt\varepsilon],
\end{align*}
using the fact that $\varepsilon|\xi|\le\pi-\varepsilon^{3/2}$.
\end{proof}
For $n\in\{1,2,3,4\}$ the leading-order evolution equation $\partial_\xi\widehat V_n = \widehat{\mathcal A}_n(\xi)\widehat V_n$ has the two linearly independent solutions $\widehat V_{n,j}(\xi)$, $j\in\{1,2\}$, where
\begin{alignat*}{3}
\widehat V_{1,1}(\xi) :=& \pmat{\rme^{-\xi^2/2}\\-\xi \rme^{-\xi^2/2}}\quad&
\widehat V_{1,2}(\xi) :=& \frac12\pmat{\sqrt\pi \rme^{-\xi^2/2}\erfi(\xi)\\\left[-\sqrt\pi\xi \rme^{-\xi^2}\erfi(\xi)+2\right]\rme^{\xi^2/2}}\\ 
\widehat V_{2,1}(\xi) :=& \pmat{\xi \rme^{-\xi^2/2}\\(1-\xi^2) \rme^{-\xi^2/2}}&
\widehat V_{2,2}(\xi) :=& \pmat{\left[1-\sqrt\pi\xi \rme^{-\xi^2}\erfi(\xi)\right]\rme^{\xi^2/2}\\\left[-\xi+\sqrt\pi (\xi^2-1) \rme^{-\xi^2}\erfi(\xi)\right]\rme^{\xi^2/2}}\\
\widehat V_{3,1}(\xi) :=& \pmat{(2\xi^2-1)\rme^{-\xi^2/2}\\\xi(5-2\xi^2)\rme^{-\xi^2/2}}&
\widehat V_{3,2}(\xi) :=& \frac14\pmat{\left[2\xi+\sqrt\pi(1-2\xi^2)\rme^{-\xi^2}\erfi(\xi)\right]\rme^{\xi^2/2}\\\left[4-2\xi^2+\sqrt\pi(2\xi^2-5)\xi \rme^{-\xi^2}\erfi(\xi)\right]\rme^{\xi^2/2}}\\ 
\widehat V_{4,1}(\xi) :=& \pmat{\xi(2\xi^2-3)\rme^{-\xi^2/2}\\(-2\xi^4+9\xi^2-3) \rme^{-\xi^2/2}}\quad&
\widehat V_{4,2}(\xi) :=& \frac16\pmat{\left[2-2\xi^2+\sqrt\pi\xi(2\xi^2-3)\rme^{-\xi^2}\erfi(\xi)\right]\rme^{\xi^2/2}\\\left[2\xi(\xi^2-4)+\sqrt\pi(-2\xi^4+9\xi^2-3)\rme^{-\xi^2}\erfi(\xi)\right]\rme^{\xi^2/2}},
\end{alignat*} 
as can be verified by explicit computation.
We solve (\ref{eq:slowFirst}) for $\xi\in \widehat I_s(\varepsilon):=[-\pi/\varepsilon+\sqrt\varepsilon,\pi/\varepsilon-\sqrt\varepsilon]$. We expect $\widehat\varphi_n(\xi)$ is close to the formal eigenfunction $H_{n-1}(\xi)\rme^{-\xi^2/2}$; thus, owing to symmetry considerations, we assume that $\widehat U_n(0)\in\Span{\widehat V_{n,1}(0)}$. We then parametrize the corresponding solution to (\ref{eq:slowFirst}) at the matching point $x=\pm\varepsilon^{3/2}$, which corresponds with $\xi=\mp(\pi/\varepsilon-\sqrt\varepsilon)=:\mp\xi_0$.
\begin{Proposition}\label{prop:slowSoln}
Define for every $\varepsilon$ the norm $\|u(\cdot)\|_{\varepsilon} = \sup_{\xi\in \widehat I_s(\varepsilon)} |u(\xi)|$; also define 
\begin{align*}
\breve\Lambda_1:= \frac1{\xi_0}\rme^{\xi_0^2}\widehat\Lambda_1,\quad \breve\Lambda_2:=\frac1{\xi_0^3}\rme^{\xi_0^2}\widehat\Lambda_2, \quad \breve\Lambda_3:= \frac1{\xi_0^5}\rme^{\xi_0^2}\widehat\Lambda_3,\quad \text{and}\quad\breve\Lambda_4:=\frac1{\xi_0^7}\rme^{\xi_0^2}\widehat\Lambda_4.
\end{align*} 
Then there exist constants $\widehat\varepsilon_0$,$\widehat\rho_1$,$\widehat\rho_2 > 0$ such that for all $0\le\varepsilon\le\widehat\varepsilon_0$ the set of all solutions to (\ref{eq:slowFirst}) with $\|u(\cdot)\|_\varepsilon\le\widehat\rho_1$, $\widehat U_n(0)=\widehat d_n\widehat V_{n,1}(0)$ and $|d_n|,|\breve\Lambda_n|\le\widehat\rho_2$ are given by 
\begin{align}\label{eq:slowSoln}
\widehat \varphi_1(\xi; \varepsilon, \breve\Lambda_1)=&\quad \widehat d_1\left[1+\calO(\varepsilon^{-2}\rme^{-2\pi/\sqrt\varepsilon}\ln\varepsilon + |\breve\Lambda_1|)\right]\rme^{-\xi^2/2}\nonumber\\
\widehat \psi_1(\xi; \varepsilon, \breve\Lambda_1)=& -\widehat d_1\left[1+\calO(\varepsilon^{-2}\rme^{-2\pi/\sqrt\varepsilon}\ln\varepsilon + |\breve\Lambda_1|)\right]\xi \rme^{-\xi^2/2}\nonumber\\
\widehat \varphi_2(\xi; \varepsilon, \breve\Lambda_2) =&\quad\widehat d_2\left[1+\calO(\varepsilon^{-2}\rme^{-2\pi/\sqrt\varepsilon}\ln\varepsilon + |\breve\Lambda_2|)\right]\xi \rme^{-\xi^2/2}\nonumber\\
\widehat \psi_2(\xi; \varepsilon, \breve\Lambda_2) =&-\widehat d_2\left[1+\calO(\varepsilon^{-2}\rme^{-2\pi/\sqrt\varepsilon}\ln\varepsilon + |\breve\Lambda_2|)\right](\xi^2-1)\rme^{-\xi^2/2},\nonumber\\
\widehat \varphi_3(\xi; \varepsilon, \breve\Lambda_3)=&\quad \widehat d_3\left[1+\calO(\varepsilon^{-2}\rme^{-2\pi/\sqrt\varepsilon}\ln\varepsilon + |\breve\Lambda_3|)\right](2\xi^2-1)\rme^{-\xi^2/2}\nonumber\\
\widehat \psi_3(\xi; \varepsilon, \breve\Lambda_3)=& -\widehat d_3\left[1+\calO(\varepsilon^{-2}\rme^{-2\pi/\sqrt\varepsilon}\ln\varepsilon + |\breve\Lambda_3|)\right]\xi(2\xi^2-5)\rme^{-\xi^2/2}\nonumber\\
\widehat \varphi_4(\xi; \varepsilon, \breve\Lambda_4) =&\quad\widehat d_4\left[1+\calO(\varepsilon^{-2}\rme^{-2\pi/\sqrt\varepsilon}\ln\varepsilon + |\breve\Lambda_4|)\right]\xi(2\xi^2-3) \rme^{-\xi^2/2}\nonumber\\
\widehat \psi_4(\xi; \varepsilon, \breve\Lambda_4) =&-\widehat d_4\left[1+\calO(\varepsilon^{-2}\rme^{-2\pi/\sqrt\varepsilon}\ln\varepsilon + |\breve\Lambda_4|)\right](2\xi^4-9\xi^2+3) \rme^{-\xi^2/2}
\end{align}
where the coefficients in front of $\breve\Lambda_n$ at the matching point $\xi=\xi_0$ are
\begin{alignat}{3}\label{eq:Melnikov}
\widehat \varphi_1 
:\quad&\frac{\sqrt\pi}{4}\left[1+\calO(\varepsilon^2)\right]\breve\Lambda_1  \qquad &\widehat \psi_1:\quad-&\frac{\sqrt\pi}{4}\left[1+\calO(\varepsilon^2)\right] \breve\Lambda_1\nonumber\\
\widehat \varphi_2
:\quad & \frac{\sqrt\pi}{8}\left[1+\calO(\varepsilon^2)\right]\breve\Lambda_2 \qquad&\widehat \psi_2:\quad-&\frac{\sqrt\pi}{8}\left[1+\calO(\varepsilon^2)\right]\breve\Lambda_2\nonumber\\
\widehat \varphi_3 
:\quad&\frac{\sqrt\pi}{8}\left[1+\calO(\varepsilon^2)\right]\breve\Lambda_1  \qquad &\widehat \psi_3:\quad-&\frac{\sqrt\pi}{8}\left[1+\calO(\varepsilon^2)\right] \breve\Lambda_1\nonumber\\
\widehat \varphi_4
:\quad & \frac{3\sqrt\pi}{16}\left[1+\calO(\varepsilon^2)\right]\breve\Lambda_2 \qquad&\widehat \psi_4:\quad-&\frac{3\sqrt\pi}{16}\left[1+\calO(\varepsilon^2)\right]\breve\Lambda_2.
\end{alignat}
Furthermore, 
\begin{alignat*}{7}
\widehat \varphi_1(-\xi; \cdot)&=&\widehat \varphi_1(\xi; \cdot),\qquad &\widehat \varphi_2(-\xi; \cdot)=-&\widehat \varphi_2(\xi; \cdot), \qquad &\widehat \varphi_3(-\xi; \cdot)=&\widehat \varphi_3(\xi; \cdot), \qquad &\widehat \varphi_4(-\xi; \cdot)=-&\widehat \varphi_4(\xi; \cdot)\\
\widehat \psi_1(-\xi; \cdot)&=-&\widehat \psi_1(\xi; \cdot), \qquad&\widehat \psi_2(-\xi; \cdot)=&\widehat \psi_2(\xi; \cdot), \qquad&\widehat \psi_3(-\xi; \cdot)=-&\widehat \psi_3(\xi; \cdot), \qquad&\widehat \psi_4(-\xi; \cdot)=&\widehat \psi_4(\xi; \cdot).
\end{alignat*}
\end{Proposition}
We remark that the definition of $\breve\Lambda_n$ implies that the eigenvalues for (\ref{eq:etildeprob}) are exponentially close to the eigenvalues for (\ref{eq:eslow0}). This is consistent with our numerical simulations; we will show why this is a valid assumption in Section~\ref{sec:matching}. Note further that (\ref{eq:slowSoln}) shows that the eigenfunctions $\widehat\varphi_n$ are close to the formal eigenfunctions $H_{n-1}(\xi) \rme^{-\xi^2/2}$ as expected from the formal calculations in Section~\ref{sec:overview}. 
\begin{proof}
All solutions to (\ref{eq:slowFirst}) with initial data $\widehat U_n(0) = \widehat d_n\widehat V_{n,1}(0)$ satisfy the fixed point equation
\begin{align}\label{eq:VOCslow}
\widehat U_n(\xi) 
=&\widehat d_n\widehat V_{n,1}(\xi)+\widehat V_{n,1}(\xi)\int_{0}^\xi \langle \widehat W_{n,1}(\tau), \widehat{\mathcal N}_n (\widehat U_n(\tau), \tau; \varepsilon, \widehat\Lambda_n)  \rangle  \rmd\tau + \widehat V_{n,2}(\xi)\int_0^\xi \langle \widehat W_{n,2}(\tau),\widehat{\mathcal N}_n(\widehat U_n(\tau), \tau; \varepsilon, \widehat\Lambda_n)\rangle \rmd\tau
\end{align}
where
\begin{alignat*}{3}
\widehat W_{1,1}(\xi) :=& \frac12\pmat{\left[-\sqrt\pi\xi \rme^{-\xi^2}\erfi(\xi)+2\right]\rme^{\xi^2/2}\\-\sqrt\pi \rme^{-\xi^2/2}\erfi(\xi)}\quad&
\widehat W_{1,2}(\xi) :=& \pmat{\xi \rme^{-\xi^2/2}\\\rme^{-\xi^2/2}}\\
\widehat W_{2,1}(\xi) :=& \pmat{\left[\xi+\sqrt\pi (1-\xi^2) \rme^{-\xi^2}\erfi(\xi)\right]\rme^{\xi^2/2}\\\left[1-\sqrt\pi\xi \rme^{-\xi^2}\erfi(\xi)\right]\rme^{\xi^2/2}}\quad&
\widehat W_{2,2}(\xi) :=& \pmat{(1-\xi^2) \rme^{-\xi^2/2}\\-\xi \rme^{-\xi^2/2}},\\
\widehat W_{3,1}(\xi) :=& \frac14\pmat{\left[2\xi^2-4+\sqrt\pi(5-2\xi^2)\xi \rme^{-\xi^2}\erfi(\xi)\right]\rme^{\xi^2/2}\\\left[2\xi+\sqrt\pi(1-2\xi^2)\rme^{-\xi^2}\erfi(\xi)\right]\rme^{\xi^2/2}}&
\widehat W_{3,2}(\xi) :=& \pmat{\xi(5-2\xi^2)\rme^{-\xi^2/2}\\(1-2\xi^2)\rme^{-\xi^2/2}}\\
\widehat W_{4,1}(\xi) :=& \frac16\pmat{[2\xi(\xi^2-4)+\sqrt\pi(-2\xi^4+9\xi^2-3)\rme^{-\xi^2}\erfi(\xi)]\rme^{\xi^2/2}\\\left[2\xi^2-2+\sqrt\pi\xi(3-2\xi^2)\rme^{-\xi^2}\erfi(\xi)\right]\rme^{\xi^2/2}}\quad&
\widehat W_{4,2}(\xi) :=& \pmat{(2\xi^4-9\xi^2+3) \rme^{-\xi^2/2}\\\xi(2\xi^2-3)\rme^{-\xi^2/2}},
\end{alignat*}
are two linearly independent solutions to the associated adjoint equation $\widehat W_n' = -\widehat{\mathcal A}_n^*(\xi)\widehat W_n$,
which have been normalized so that $\langle \widehat V_{n,i}, \widehat W_{n,j} \rangle_{\R^2}=\delta_{ij}$.
Equation (\ref{eq:VOCslow}) is linear and defined for $\xi\in\R$; thus solutions exist and are bounded on any finite interval.
However, they may not be uniformly bounded in $\varepsilon$ since the interval of integration $\widehat I_s(\varepsilon)$ grows like $1/\varepsilon$. Our first goal, therefore, is to show that the constant bounding the higher order terms in (\ref{eq:slowSoln}) does not grow with $\widehat I_s(\varepsilon)$. 
Motivated by the formal analysis we use the ansatz $\widehat \varphi_n(\xi) =H_{n-1}(\xi)\rme^{-\xi^2/2} \widehat u_n(\xi)$ and $\widehat \psi_n(\xi) = \frac\rmd{\rmd\xi}\left[H_{n-1}(\xi)\rme^{-\xi^2/2}\right]\widehat v_n(\xi)$ to solve (\ref{eq:VOCslow}). We focus on $n=1$, $2$, since all of the technical difficulties arise in these cases; the $n=3$, $4$ cases can be proven completely analogously. The resulting evolution equations for $\widehat u_n$ and $\widehat v_n$ are
\begin{subequations}\label{eq:FPslowBreve}
\begin{align}
\widehat u_1(\xi; \varepsilon, \breve\Lambda_1) 
=&\widehat d_1-\frac{\sqrt\pi}2\int_0^\xi \rme^{-\tau^2}\erfi(\tau)\left(\widehat{\mathcal N}(\tau; \varepsilon)+\widehat\Lambda_1\right)\widehat u_1(\tau;  \varepsilon, \breve\Lambda_1) \rmd\tau \nonumber\\
&\quad+\frac{\sqrt\pi}2 \erfi(\xi) \int_0^\xi \rme^{-\tau^2}\left(\widehat{\mathcal N}(\tau; \varepsilon)+\widehat\Lambda_1\right)\widehat u_1(\tau;  \varepsilon, \breve\Lambda_1) \rmd\tau\nonumber\\
=&:\widehat{\mathcal F}_{1,u}(\widehat u_1; \varepsilon, \widehat d_1, \widehat\Lambda_1)\label{eq:FPslowBreveU1}\\
\widehat v_1(\xi; \varepsilon, \breve\Lambda_1) 
=&\widehat d_1- \frac{\sqrt\pi}2\int_0^\xi \rme^{-\tau^2}\erfi(\tau)\left(\widehat{\mathcal N}(\tau; \varepsilon)+\widehat\Lambda_1\right)\widehat u_1(\tau;  \varepsilon, \breve\Lambda_1) \rmd\tau\nonumber\\
&\quad-\frac1{2\xi} \left[2\rme^{\xi^2}-\sqrt\pi\xi \erfi(\xi)\right]\int_0^\xi \rme^{-\tau^2}\left(\widehat{\mathcal N}(\tau; \varepsilon)+\widehat\Lambda_1\right)\widehat u_1(\tau;  \varepsilon, \breve\Lambda_1) \rmd\tau\nonumber\\
=&:\widehat{\mathcal F}_{1,v}(\widehat u_1; \varepsilon, \widehat d_1, \widehat\Lambda_1)\label{eq:FPslowBreveV1}\\
\widehat u_2(\xi;  \varepsilon, \breve\Lambda_2) =&\widehat d_2+\int_0^\xi \tau \left[1-\sqrt\pi\tau \rme^{-\tau^2} \erfi(\tau)\right]\left(\widehat{\mathcal N}(\tau; \varepsilon)+\widehat\Lambda_2\right)\widehat u_2(\tau;  \varepsilon, \breve\Lambda_2) \rmd\tau \nonumber\\
&\quad-\frac1\xi \left[\rme^{\xi^2}-\sqrt\pi\xi \erfi(\xi)\right]\int_0^\xi \tau^2 \rme^{-\tau^2}\left(\widehat{\mathcal N}(\tau; \varepsilon)+\widehat\Lambda_2\right)\widehat u_2(\tau;  \varepsilon, \breve\Lambda_2) \rmd\tau\nonumber\\
=&:\widehat{\mathcal F}_{2,u}(\widehat u_2; \varepsilon, \widehat d_2, \widehat\Lambda_2)\label{eq:FPslowBreveU2}
\end{align}
All terms in (\ref{eq:FPslowBreveU1})-(\ref{eq:FPslowBreveU2}) are well defined for all $\xi$ since for $\xi$ small we have
\begin{align*}
\int_0^\xi \rme^{-\tau^2}\rmd\tau =& \xi-\frac{\xi^3}3+\calO(\xi^5)\qquad\text{and}\qquad\int_0^\xi \tau^2\rme^{-\tau^2}\rmd\tau = \frac{\xi^3}3+\calO(\xi^5).
\end{align*}
For $\psi_2(\xi)$ we fix $\xi_1>1$ and make the ansatz 
\[
\psi_2(\xi;  \varepsilon, \breve\Lambda_2)=\left\{\begin{array}{ccc}
\rme^{-\xi^2/2}\breve v_2(\xi;  \varepsilon, \breve\Lambda_2) &:& |\xi|\le\xi_1\\
\rme^{-\xi^2/2}\left[\breve v_2(|\xi_1|;  \varepsilon, \breve\Lambda_2)+(1-\xi^2)\widehat v_2(\xi;  \varepsilon, \breve\Lambda_2)\right] &:& |\xi|\ge\xi_1
\end{array}\right\}
\]
where $\breve v_2$ is defined for $|\xi|\le\xi_1$ and $\widehat v_2$ is defined for $|\xi|\ge\xi_1$ and
\begin{align}
\breve v_2(\xi;  \varepsilon, \breve\Lambda_2) =&\widehat d_2(1-\xi^2)+(1-\xi^2)\int_0^\xi \tau \left[1-\sqrt\pi\tau \rme^{-\tau} \erfi(\tau)\right]\left(\widehat{\mathcal N}(\tau; \varepsilon)+\widehat\Lambda_2\right)\widehat u_2(\tau;  \varepsilon, \breve\Lambda_2) \rmd\tau \nonumber\\
&\quad+\left[\xi \rme^{\xi^2}-\sqrt\pi (\xi^2-1) \erfi(\xi)\right]\int_0^\xi \tau^2 \rme^{-\tau^2}\left(\widehat{\mathcal N}(\tau; \varepsilon)+\widehat\Lambda_2\right)\widehat u_2(\tau;  \varepsilon, \breve\Lambda_2)\rmd\tau\nonumber\\
\widehat v_2(\xi;  \varepsilon, \breve\Lambda_2) =&\widehat d_2\frac{\xi_1^2-\xi^2}{1-\xi^2}+\int_{\xi_1}^\xi \tau \left[1-\sqrt\pi\tau \rme^{-\tau} \erfi(\tau)\right]\left(\widehat{\mathcal N}(\tau; \varepsilon)+\widehat\Lambda_2\right)\widehat u_2(\tau;  \varepsilon, \breve\Lambda_2) \rmd\tau \nonumber\\
&\quad-\frac1{\xi^2-1}\left[\xi \rme^{\xi^2}-\sqrt\pi (\xi^2-1) \erfi(\xi)\right]\int_{\xi_1}^\xi \tau^2 \rme^{-\tau^2}\left(\widehat{\mathcal N}(\tau; \varepsilon)+\widehat\Lambda_2\right)\widehat u_2(\tau;  \varepsilon, \breve\Lambda_2)\rmd\tau\nonumber\\
=&:\widehat{\mathcal F}_{2,v}(\widehat u_2; \varepsilon, \widehat d_2, \widehat\Lambda_2)\label{eq:FPslowBreveV2}.
\end{align}
\end{subequations}
Now $\breve v_2(\xi)$ is clearly uniformly bounded with
\[
\breve v_2(\xi;  \varepsilon, \breve\Lambda_2)= \widehat d_2(1-\xi^2) + \calO(\varepsilon^{-2}\rme^{-2\pi/\sqrt\varepsilon}\ln\varepsilon + |\breve\Lambda_2|) \quad\text{for } |\xi|\le\xi_1
\]
and $\widehat{\mathcal F}_{2,v}$ is well-defined for all $\xi\ge\xi_1$. Define 
$
\widehat{\mathcal D}_\varepsilon(\rho):=\{u\in\mathcal C^0(\widehat I_s(\varepsilon)) : \|u\|_\varepsilon\le \rho\}.
$
Our goal is to show there exists $\widehat\rho_1, \widehat\rho_2, \widehat\varepsilon_0\ll1$ small enough such that
\[
\widehat{\mathcal F}_{n,j}(\widehat u; \varepsilon, \widehat d_n, \breve \Lambda_n): \widehat{\mathcal D}_\varepsilon(\widehat\rho_1)\times\{\varepsilon\le\widehat\varepsilon_0\}\times\{|\widehat d_n|, |\breve\Lambda_n|\le \widehat\rho_2\}\to \widehat{\mathcal D}_\varepsilon(\widehat\rho_1)\quad\text{ with } j\in\{u,v\},
\]
whence $\widehat u_n(\xi;  \varepsilon, \breve\Lambda_n)$ and $\widehat v_n(\xi;  \varepsilon, \breve\Lambda_n)$ will be uniformly bounded in $\widehat I_s(\varepsilon)$. 
Using (\ref{eq:Nslow1}) to bound the nonlinearity when multiplied by an exponentially small integrand $\sim \rme^{-\tau^2}$ and (\ref{eq:Nslow2}) to bound the nonlinearity when multiplied by an algebraic integrand $\sim \rme^{-\tau^2}\erfi(\tau)$, and Claim~\ref{claim:C2} below, there exists a $0<C_2(\widehat\varepsilon_1)<\infty$ such that for all $\widehat u_1\in\widehat{\mathcal D}_\varepsilon(\rho)$ and $\varepsilon\le\widehat\varepsilon_2$,
\begin{align*}
\|\widehat{\mathcal F}_{1,u}(\widehat u_1; \varepsilon, \widehat d_1,  \xi_0\rme^{-\xi_0^2}\breve\Lambda_1)\|_\varepsilon \le& |\widehat d_1|+\frac{\sqrt\pi \rho}2 \bigg[\left(\frac {\widehat C(\widehat\varepsilon_1)}{\varepsilon^2}\rme^{-2\pi/\sqrt\varepsilon}+ \frac1\varepsilon \rme^{-\pi^2/\varepsilon^2}\rme^{2\pi/\sqrt\varepsilon}\breve\Lambda_1\right)\int_{0}^{\xi_0} \rme^{-\tau^2}\erfi(\tau)\rmd\tau \nonumber\\
&\quad+ \frac {\widehat C(\widehat\varepsilon_1)}{\varepsilon^2}\rme^{-\pi^2/\varepsilon^2}\erfi(\xi_0) \int_0^{\xi_0}\rme^{-(\pi-\varepsilon\tau)^2/\varepsilon^2}\rmd\tau+\xi_0\rme^{-\xi_0^2}\breve\Lambda_1 \erfi(\xi_0) \int_0^{\xi_0} \rme^{-\tau^2} \rmd\tau\bigg]\\
\le& |\widehat d_1|+\frac{\sqrt\pi \rho\widehat C_2(\widehat\varepsilon_1)}2 \bigg[\left(\frac {\widehat C(\widehat\varepsilon_1)}{\varepsilon^2}\rme^{-2\pi/\sqrt\varepsilon}+ \frac1\varepsilon \rme^{-\pi^2/\varepsilon^2}\rme^{2\pi/\sqrt\varepsilon}\breve\Lambda_1\right)\ln\varepsilon+ \frac {\widehat C(\widehat\varepsilon_1)}{\varepsilon} \rme^{-2\pi/\sqrt\varepsilon}+ \breve\Lambda_1 \bigg].
\end{align*}
It is now straightforward to show that there exist constants $\widehat\rho_1$, $\widehat\rho_2>0$ and $0<\widehat\varepsilon_0\le\widehat\varepsilon_2$ such that $\widehat{\mathcal F}_n(\widehat u_n; \varepsilon, \widehat d_n,  \rme^{-\xi_0^2}\breve\Lambda_n) \in \widehat{\mathcal D}_\varepsilon(\widehat\rho_1)$ for all $\widehat u_n\in\widehat{\mathcal D}_\varepsilon(\widehat\rho_1)$, $|\widehat d_n|, |\breve\Lambda_n|\le \widehat\rho_1$, and $\varepsilon\le\widehat\varepsilon_0$. We remark that the coefficients in $\breve\Lambda_1$ is $\calO(1)$ as a consequence of our choice of scaling of $\widehat\Lambda_1$.

A completely analogous argument holds for $\widehat{\mathcal F}_{1,v}$, $\widehat{\mathcal F}_{2,u}$, and $\widehat{\mathcal F}_{2,v}$, with the following modification
\begin{enumerate}
\item For $\widehat{\mathcal F}_{2,v}$ we use the function space $\widehat{\mathcal D}_\varepsilon(\rho):=\{u\in\mathcal C^0([\xi_1,\xi_0]) : \|u\|_\varepsilon\le \rho\}$.
\item For $\widehat{\mathcal F}_{2,u}$, in order to get the specific form of the $\calO(\varepsilon^{-2}\rme^{-2\pi/\sqrt\varepsilon}\ln\varepsilon + |\breve\Lambda_2|)$ we need  
\begin{align*}
\argmax_{\xi\in\widehat I_s(\varepsilon)} \left|\int_0^\xi \tau \left[1-\sqrt\pi\tau \rme^{-\tau^2} \erfi(\tau)\right]\rmd\tau\right|=
\argmax_{\xi\in\widehat I_s(\varepsilon)} \left|\frac1\xi \rme^{\xi^2}\left[1-\sqrt\pi\xi \rme^{-\xi^2}\erfi(\xi)\right]\int_0^\xi\tau^2\rme^{-\tau^2}\rmd\tau \right|= \pm\xi_0.
\end{align*}
In other words, we need to keep the minus signs and still show that the argmax occurs at the end of the interval $\widehat I_s(\varepsilon)$.
But this is true for all $\varepsilon$ small enough by using the asymptotic expansions shown in Table~\ref{tab:asySlow} to get 
\begin{align*}
&\lim_{\xi\to\infty} \int_0^\xi \tau \left[1-\sqrt\pi\tau \rme^{-\tau^2} \erfi(\tau)\right] = \lim_{\xi\to\infty}\left[\frac12\ln\left(\frac1\xi\right)+\calO\left(1\right) \right]\to-\infty\\
&\lim_{\xi\to\infty}\frac1\xi\left[\rme^{\xi^2}-\sqrt\pi\xi \erfi(\xi)\right]\int_0^\xi \tau^2 \rme^{-\tau^2}\rmd\tau = \lim_{\xi\to\infty} \rme^{\xi^2}\frac1{\xi^3} \left[-\frac{\sqrt\pi}8+\calO(1/\xi^2) \right]\to-\infty,
\end{align*}
and noting that the expressions are bounded on any bounded interval.
\item A similar issue as (ii) arises in $\widehat{\mathcal F}_{2,v}$; a completely analogous argument gives the desired result.
\end{enumerate}

Using the uniform bounds on $\widehat u_n$ we get estimates (\ref{eq:slowSoln}).
Plugging these estimates back into (\ref{eq:FPslowBreve}), again using Claim~\ref{claim:C2} and the asymptotic expansions shown in Table~\ref{tab:asySlow}, we can explicitly integrate the terms multiplying $\breve\Lambda_n$ to leading order at $\xi=\xi_0$ since $\widehat d_n$ is a constant. We obtain (\ref{eq:Melnikov}).

The symmetries then follow from the symmetry of the nonlinear term $\widehat{\mathcal N}(\xi; \varepsilon)$ which is an even function in $\xi$ since $W(x; \varepsilon)$ is odd and $W_x(x; \varepsilon)$ is even in $x$, as we noted in Section~\ref{sec:Whit}. Hence, for all even functions $\widehat u_n(\xi)$, $\widehat{\mathcal F}_n(\widehat u_n; \cdot)$ is even. Thus $\widehat u_n(\xi)$ and $\widehat v_n(\xi)$ are even and the symmetries for $\widehat\varphi_n$ and $\widehat\psi_n$ follow from the symmetries of $H_n(\xi)\rme^{-\xi^2/2}$.
\end{proof}
It remains to prove the following claim.
\begin{Claim}\label{claim:C2} 
Fix $\widehat\varepsilon_1$ as in Lemma~\ref{lemma:Nhat}. Then there exists
$0<\widehat C_2(\widehat\varepsilon_1)<\infty$ such that 
\begin{align*}
&\int_{0}^{\xi_0} \rme^{-\tau^2}\erfi(\tau)\rmd\tau \le \widehat C_2(\widehat\varepsilon_1)\ln\varepsilon\quad\text{and}\quad \erfi(\xi_0) \int_0^{\xi_0} \rme^{-\tau^2} \rmd\tau \le  \widehat C_2(\widehat\varepsilon_1)\varepsilon \rme^{\pi^2/\varepsilon^2}\rme^{-2\pi/\sqrt\varepsilon}
\end{align*}
and, moreover, such that
\begin{align*}
\erfi(\xi_0) &\int_0^{\xi_0}\rme^{-(\pi-\varepsilon\tau)^2/\varepsilon^2}\rmd\tau \le  \widehat C_2(\widehat\varepsilon_1)\varepsilon \rme^{\pi^2/\varepsilon^2}\rme^{-2\pi/\sqrt\varepsilon}.
\end{align*}
\end{Claim}
\begin{proof}
The claim follows from the asymptotic expansions in Table~\ref{tab:asySlow}, the facts that
\begin{align*}
\int_0^{\xi_0} \rme^{-(\pi-\varepsilon\tau)^2/\varepsilon^2}\rmd\tau \le& \int_{-\infty}^\infty \rme^{-(\pi-\varepsilon\tau)^2/\varepsilon^2}\rmd\tau=\int_{-\infty}^\infty \rme^{-\tau^2}\rmd\tau=  \sqrt\pi
\end{align*}
due to symmetry, and the small argument approximation
$
\int_{0}^{\sqrt\varepsilon} \rme^{-\tau^2}\rmd\tau = \sqrt\varepsilon\left[1+\calO\left(\varepsilon \right)\right].
$
\end{proof}
\begin{table}[t]
\centering
\captionsetup{width=.7\linewidth}
\begin{tabular} {c c c }
\toprule
$\erfi(\xi)$ & \,& $\rme^{\xi^2}\frac1{\xi\sqrt\pi}\left[1+ \frac1{2\xi^2}+\calO\left(\frac1{\xi^4} \right)\right]$ \\
$\int_{0}^\xi \rme^{-\tau^2}\rmd\tau$ & \,&$\frac{\sqrt\pi}2-\rme^{-\xi^2}\frac1{2\xi}\left[1- \frac1{2\xi^2}+\calO\left(\frac1{\xi^4} \right)\right]$ \\
$\int_{0}^\xi \tau^2 \rme^{-\tau^2}\rmd\tau$ & \,&$ \frac{\sqrt\pi}4-\rme^{-\xi^2}\frac\xi4\left[2+ \frac1{\xi^2}+\calO\left(\frac1{\xi^4} \right)\right]$ \\
$\int_{0}^\xi \tau^4 \rme^{-\tau^2}\rmd\tau$ & \,&$ \frac{3\sqrt\pi}8-\rme^{-\xi^2}\frac{\xi^3}4\left[2+ \frac3{\xi^2}+\calO\left(\frac1{\xi^4} \right)\right]$ \\
$\int_{0}^\xi \tau^6 \rme^{-\tau^2}\rmd\tau$ & \,&$ \frac{15\sqrt\pi}{16}-\rme^{-\xi^2}\frac{\xi^5}4\left[2+ \frac5{\xi^2}+\calO\left(\frac1{\xi^4} \right)\right]$ \\
$\sqrt\pi\int_{0}^\xi \rme^{-\tau^2}\erfi(\tau)\rmd\tau$ & \,&$-\ln\left(\frac1\xi\right)-\frac12\psi^{(0)}\left(\frac12\right)+\calO\left(\frac1{\xi^2}\right)$ \\
$\int_{0}^\xi \tau[1-\sqrt\pi\tau \rme^{-\tau^2} \erfi(\tau)]\rmd\tau$ & \,&$\frac12\ln\left(\frac1\xi\right)+\frac14\psi^{(0)}\left(-\frac12\right)+\calO\left(\frac1{\xi^2}\right)$ \\
$\int_{0}^\xi \tau^3[1-\sqrt\pi\tau \rme^{-\tau^2} \erfi(\tau)]\rmd\tau$ & \,&$-\frac{\xi^2}4+\frac34\ln\left(\frac1\xi\right)+\frac38\psi^{(0)}\left(-\frac32\right)+\calO\left(\frac1{\xi^2}\right)$ \\
$\int_{0}^\xi \tau^5 [1-\sqrt\pi\tau \rme^{-\tau^2} \erfi(\tau)]\rmd\tau$ & \,&$-\frac{\xi^2(\xi^2+3)}8+\frac{15}8\ln\left(\frac1\xi\right)+\frac{15}{16}\psi^{(0)}\left(-\frac52\right)+\calO\left(\frac1{\xi^2}\right)$ \\
\bottomrule
\end{tabular}
\caption{The asymptotic behavior of all terms  in (\ref{eq:VOCslow}) for $\xi\gg 1$ and $n\in\{1,2,3,4\}$. The integrals and asymptotic expansions were computed using Mathematica. $\psi^{(0)}(x)$ is the digamma function, where $\psi^{(0)}(1/2) = -\gamma-\ln(4)$, $\psi^{(0)}(-1/2)=2-\gamma-\ln(4)$, $\psi^{(0)}(-3/2) = \frac83-\gamma-\ln(4)$, $\psi^{(0)}(-5/2)=\frac{45}{15}-\gamma-\ln(4)$, and $\gamma = \lim_{n\to\infty} \left( \sum_{n=1}^n\frac 1n-\ln n\right)$ is the Euler-Mascheroni constant. }
\label{tab:asySlow}
\end{table}
\subsection{Fast variables}\label{sec:fast}
In this section we compute the eigenfunctions for (\ref{eq:etildeprob}) for $x\in I_f(\varepsilon):=[-\varepsilon^{3/2}, \varepsilon^{3/2}]$. Motivated by the formal asymptotic analysis in Section~\ref{sec:overview} we define the fast variable $z:=x/\varepsilon^2$. We call the eigenfunctions in these coordinates $\widecheck\varphi_n(z)$; they are defined for $z\in[-1/\sqrt\varepsilon, 1/\sqrt\varepsilon]=:\widecheck I_f(\varepsilon)$ and satisfy 
\begin{align}\label{eq:efast2}
\partial_{zz}\widecheck\varphi_n -\left[\widecheck W_z(z; \varepsilon)+\widecheck W^2(z; \varepsilon)\right]\widecheck\varphi_n =\varepsilon^2\widehat\lambda_n\widecheck\varphi_n
\end{align}
where for any $t\in\R^+$
\begin{align*}
\widecheck W(z; \varepsilon):=&tW_0(\varepsilon^2 z, t;\nu),\nonumber\\
\widecheck W_z(z; \varepsilon):=&t\varepsilon^2\left[\partial_xW_0\right](\varepsilon^2 z, t; \nu),
\end{align*}
We remark that even though $\widecheck W_z(z;\varepsilon)$ is obtained through an appropriate transformation of $\partial_xW_0(x, t; \nu)$, it is also true that $\widecheck W_z(z;\varepsilon) = \partial_z\widecheck W(z;\varepsilon)$; hence our notation. 

Motivated by the formal analysis we re-write (\ref{eq:efast2}) as
\begin{align*}
\partial_{zz}\widecheck\varphi_n -\left[\pi^2-2\pi^2\sech^2(\pi z)+\widecheck{\mathcal N}(z; \varepsilon)\right]\widecheck\varphi_n =\varepsilon^2\widehat\lambda_n\widecheck\varphi_n
\end{align*}
with $\widecheck{\mathcal N}(z; \varepsilon):=\widecheck W_x(z; \varepsilon)+\widecheck W^2(z; \varepsilon)-\pi^2[1-2\sech^2(\pi z)]$, which is equivalent to the first order system
\begin{align}\label{eq:fastFirst}
\partial_z \widecheck U_n = \widecheck{\mathcal A}_n(z)\widecheck U_n+\widecheck{\mathcal N}_n(\widecheck U_n, z; \varepsilon, \widehat\Lambda_n)
\end{align}
where $\widecheck U_n:=(\widecheck \varphi_n, \widecheck \psi_n)^T$ with $\widecheck\psi_n:=\partial_z\widecheck\varphi_n$, $\widehat\lambda_n = -2n+\widehat\Lambda_n$ from Section~\ref{sec:slow},
\[
\quad \widecheck{\mathcal A}_n:= \pmat{0&1\\\pi^2[1-2\sech^2(\pi z)]&0}, \quad\text{and}\quad \widecheck{\mathcal N}_n(\widecheck \varphi_n, \widecheck \psi_n, z; \varepsilon, \widehat\Lambda_n):=\pmat{0\\\left(\widecheck{\mathcal N}(z; \varepsilon)+\varepsilon^2\widehat\lambda_n\right)\widecheck \varphi_n}.
\]

\begin{Lemma}\label{lem:Nexp}
Define $\widecheck {\mathcal N}_{\alg}(z; \varepsilon):= \varepsilon^2[1-2\pi z\tanh(\pi z)]+\varepsilon^4z^2$ and $\widecheck {\mathcal N}_{\exp}(z; \varepsilon):=\widecheck {\mathcal N}(z; \varepsilon)- \widecheck {\mathcal N}_{\alg}(z; \varepsilon)$. Then there exists $\widecheck\varepsilon_1>0$ and $0<\widecheck C(\widecheck\varepsilon_1)<\infty$ such that for all $\varepsilon\le\widecheck\varepsilon_1$ and $z\in\widecheck I_f(\varepsilon)$,
\[
\left| \widecheck {\mathcal N}_{\exp}(z; \varepsilon)\right| \le \widecheck C(\widecheck\varepsilon_1)\rme^{-1/\varepsilon^2}
\]
\end{Lemma}
Thus, for all $\varepsilon\le\widecheck\varepsilon_1$, $\widecheck{\mathcal N}(z;\varepsilon)$ is exponentially close to $\widecheck{\mathcal N}_{\alg}(z;\varepsilon)$. In particular, there exists a constant $0<\widecheck C_1(\widecheck\varepsilon_1)<\infty$ such that for all $\varepsilon\le\widecheck\varepsilon_1$ and $z\in\widecheck I_f(\varepsilon)$
\[
\left| \widecheck {\mathcal N}(z; \varepsilon)\right| \le \widecheck C_1(\widecheck\varepsilon_1)\varepsilon^{3/2}
\]
\begin{proof}
The result follows from the definitions of $\widecheck W$ and $\widecheck W_z$ in terms of $W$ and estimates (\ref{est:W}). 
\end{proof}
The leading order evolution equation $\partial_z\widecheck V = \widecheck{\mathcal A}(z)\widecheck V$ has the two linearly independent solutions $\widecheck V_n(z)$, $j\in\{1,2\}$, where
\begin{alignat*}{3}
\widecheck V_1(z) :=& \pmat{-\sech(\pi z)\\\pi\sech(\pi z)\tanh(\pi z)}\quad&\text{and}\quad 
\widecheck V_2(z) :=& \frac1{2\pi}\pmat{\sinh(\pi z)+\pi z\sech(\pi z)\\\pi\big[\cosh(\pi z)+\sech(\pi z)-\pi z\sech(\pi z)\tanh(\pi z)\big]},
\end{alignat*} 
as can be verified by explicit computation. Observe that the leading order terms no longer depends on $n$. Due to symmetry considerations we construct purely even or purely odd eigenfunctions; thus we assume that either $\widecheck U_n(0)\in\Span{\widecheck V_1(0)}$ or $\Span{\widecheck V_2(0)}$. We then parametrize the corresponding solution to (\ref{eq:slowFirst}) at the matching point $x=\pm\varepsilon^{3/2}$, which corresponds with $z=\pm1/\sqrt\varepsilon=:\pm z_0$.
\begin{Proposition}\label{prop:fastSoln}
Define for every $\varepsilon$ the norm $\|u(\cdot)\|_{\varepsilon} = \sup_{z\in \widecheck I_f(\varepsilon)} |u(z)|$.
Then for each for $n\in\N$ there exist constants $\varepsilon_0$,$\widecheck\rho_1$,$\widecheck\rho_2 > 0$ such that for all $0\le\varepsilon\le\varepsilon_0$ the set of all solutions to (\ref{eq:fastFirst}) with $\widehat\lambda_n= -2n+\widehat\Lambda_n$, and which satisfy $\|u(\cdot)\|_\varepsilon\le\widecheck\rho_1$, with $|d_n|,|\widehat\Lambda_n|\le\widecheck\rho_2$ and $\widecheck U_n(0)=\widecheck d_n\widecheck V_1(0)$ are given by 
\begin{subequations}\label{eq:fastSoln}
\begin{align}
\widecheck \varphi_1(z; \varepsilon, \widehat\lambda_n)=&\widecheck d_n\left[-\sech^2(\pi z)\left(1+\frac{\varepsilon^2z^2}2 +\frac{n\varepsilon^2}{\pi^2}\right)+\frac{n\varepsilon^2}{\pi^2}+\calO_n(\varepsilon^{5/2} + \varepsilon^2|\widehat\Lambda_n|)\right]\cosh(\pi z),\nonumber\\
\widecheck \psi_1(z; \varepsilon, \widehat\lambda_n)=& \widecheck d_n\pi\left[\sech^2(\pi z)\left(1-\frac{\varepsilon^2z}\pi\coth(\pi z)+\frac{\varepsilon^2z^2}2+\frac{n\varepsilon^2}{\pi^2} \right)+\frac{n\varepsilon^2}{\pi^2}+\calO_n(\varepsilon^{5/2} + \varepsilon^2|\widehat\Lambda_n|)\right] \sinh(\pi z)
\end{align}
and for $\widecheck U_n(0)=\widecheck d_n\widecheck V_2(0)$ are given by
\begin{align}
\widecheck \varphi_2(z; \varepsilon, \widehat\lambda_n) =&\widecheck d_n\frac1{2\pi}\left[1+\calO_n(\varepsilon + \varepsilon^{3/2}|\widehat\Lambda_n|)\right]\left[\sinh(\pi z)+\pi z\sech(\pi z)\right],\nonumber\\
\widecheck \psi_2(z;  \varepsilon, \widehat\lambda_n) =&\widecheck d_n\frac12\left[1+\calO_n(\varepsilon + \varepsilon^{3/2}|\widehat\Lambda_n|)\right]\left[\cosh(\pi z)+\sech(\pi z)-\pi z\sech(\pi z)\tanh(\pi z)\right].
\end{align}
\end{subequations}
Furthermore, $\widecheck \varphi_1(-z)=\widecheck \varphi_1(z)$, $\widecheck \psi_1(-z)=-\widecheck \psi_1(z)$, $\widecheck \varphi_2(-z)=-\widecheck \varphi_2(z)$, and $\widecheck \psi_2(-z)=\widecheck \psi_2(z)$.
\end{Proposition}
We remark that for all $0<N<\infty$, it is possible to choose $\varepsilon_0$ and $\widehat\rho_2$ small enough (where $\widehat\rho_2$ was chosen in the proof of Proposition~\ref{prop:slowSoln}) such that $|\widehat\Lambda_n|\le\widecheck\rho_2$ whenever $\breve\Lambda_n\le\widehat\rho_2$ for all $n\le N$. We also remark that, unlike in the analogous proposition for the slow variables, Proposition~\ref{prop:fastSoln}, where we computed a different eigenfunction associated with each eigenvalue $\widehat\lambda_n \approx -2n$, here we have only two functions $\widecheck \varphi_1$ and $\widecheck \varphi_2$, which now take $\widehat\lambda_n$ as a parameter. This difference is in accord with the formal analysis which indicates that, at least to leading order, we expect that the fast eigenfunctions to solve the eigenvalue-independent equation
\[
\widecheck\varphi_{zz} +\pi^2[2\sech^2(\pi z)-1]\widecheck\varphi = 0.
\]
\begin{proof}
The argument is completely analogous to Proposition~\ref{prop:slowSoln} so we abbreviate the proof. The symmetries follow from the same argument as in Proposition~\ref{prop:slowSoln}. For the other claims we set up the fixed point equation on the space of bounded functions
\[
\widecheck{\mathcal D}_\varepsilon(\rho):=\{u\in\mathcal C^0(\widecheck I_f(\varepsilon)) : \|u\|_\varepsilon\le \rho\}.
\]
using the Variation of Parameters formula, the normalized adjoint eigenfunctions
\begin{align*}
\widehat W_1(z) :=&\frac1{2\pi}\pmat{-\pi\big[\cosh(\pi z)+\sech(\pi z)-\pi z\sech(\pi z)\tanh(\pi z)\big]\\\sinh(\pi z)+\pi z\sech(\pi z)} \quad\text{and}\quad
\widehat W_2(z) :=&\pmat{\pi\sech(\pi z)\tanh(\pi z)\\\sech(\pi z)}
\end{align*}
and the ansatz
\begin{alignat*}{3}
\widecheck \varphi_1(z;\varepsilon, \widehat\lambda_n) =&\cosh(\pi z) \widecheck u_1(z;\varepsilon, \widehat\lambda_n), \quad& \widecheck \varphi_2(z;\varepsilon, \widehat\lambda_n) =&\frac1{2\pi}\left[\sinh(\pi z)+\pi z\sech(\pi z)\right] \widecheck u_2(z;\varepsilon, \widehat\lambda_n),\\
 \widecheck \psi_1(z;\varepsilon, \widehat\lambda_n) =& \pi\sinh(\pi z)\widecheck v_1(z;\varepsilon, \widehat\lambda_n), \quad& \widecheck \psi_2(z;\varepsilon, \widehat\lambda_n) =&\frac12\left[\cosh(\pi z)+\sech(\pi z)-\pi z\sech(\pi z)\tanh(\pi z)\right]\widecheck v_2(z;\varepsilon, \widehat\lambda_n).
\end{alignat*}
We emphasize that $\widehat u_1$ exponentially grows in $z$, rather than exponentially decaying as the linear eigenfunction $\sech(\pi z)$ might suggest. This ansatz is motivated by the formal asymptotic analysis.
Owing to Claim~\ref{claim:fast} below the following expressions are well defined and bounded on any bounded interval
\begin{subequations}\label{eq:FPfastBreve}
\begin{flalign}
\widecheck u_1(z;\varepsilon, \widehat\lambda_n) 
=&-\widecheck d_1\sech^2(\pi z)&&\nonumber\\
&+\frac1{2\pi} \bigg[-\sech^2(\pi z)\int_{0}^z \left[\sinh(\pi\tau)\cosh(\pi\tau)+\pi\tau \right]\left(\widecheck{\mathcal N}(\tau; \varepsilon)+\varepsilon^2\widehat\lambda_n\right)\widecheck u_1(\tau;\varepsilon, \widehat\lambda_n) \rmd\tau &&\nonumber\\
&\qquad+ \left[\tanh(\pi z)+\pi z\sech^2(\pi z) \right] \int_0^z \left(\widecheck{\mathcal N}(\tau; \varepsilon)+\varepsilon^2\widehat\lambda_n\right)\widecheck u_1(\tau;\varepsilon, \widehat\lambda_n) \rmd\tau\bigg]&&\nonumber\\
=&:\widecheck{\mathcal F}_{1,u}(\widecheck u_1; \varepsilon, \widecheck d_1, -2n+\widehat\Lambda_n)&&\label{eq:FPfastBreveU1}
\end{flalign}
\begin{flalign}
\widecheck v_1(z;\varepsilon, \widehat\lambda_n) 
=&\widecheck d_1\sech^2(\pi z)&&\nonumber\\
&+\frac1{2\pi} \bigg[\sech^2(\pi z)\int_{0}^z \left[\sinh(\pi\tau)\cosh(\pi\tau)+\pi\tau \right]\left(\widecheck{\mathcal N}(\tau; \varepsilon)+\varepsilon^2\widehat\lambda_n\right)\widecheck u_1(\tau;\varepsilon, \widehat\lambda_n) \rmd\tau &&\nonumber\\
&+ \left[\coth(\pi z)-\pi z\sech^2(\pi z) +\frac1{\cosh(\pi z)\sinh(\pi z)}\right] \int_0^z \left(\widecheck{\mathcal N}(\tau; \varepsilon)+\varepsilon^2\widehat\lambda_n\right)\widecheck u_1(\tau;\varepsilon, \widehat\lambda_n) \rmd\tau\bigg]&&\nonumber\\
=&:\widecheck{\mathcal F}_{1,v}(\widecheck u_1; \varepsilon, \widecheck d_1, -2n+\widehat\Lambda_n)&&\label{eq:FPfastBreveV1}
\end{flalign}
\begin{flalign}
\widecheck u_2(z;\varepsilon, \widehat\lambda_n) =&\widecheck d_2&&\nonumber\\
&+\frac1{2\pi}\bigg[-\frac{1}{\cosh(\pi z)\sinh(\pi z)+\pi z}\int_{0}^z \left[\sinh(\pi\tau)+\pi\tau\sech(\pi\tau) \right]^2\left(\widecheck{\mathcal N}(\tau; \varepsilon)+\varepsilon^2\widehat\lambda_n\right)\widecheck u_2(\tau;\varepsilon, \widehat\lambda_n) \rmd\tau &&\nonumber\\
&\qquad+\int_0^z \sech(\pi \tau)\left[\sinh(\pi\tau)+\pi\tau\sech(\pi\tau) \right]\left(\widecheck{\mathcal N}(\tau; \varepsilon)+\varepsilon^2\widehat\lambda_n\right)\widecheck u_2(\tau;\varepsilon, \widehat\lambda_n) \rmd\tau\bigg]&&\nonumber\\
=&:\widecheck{\mathcal F}_{2,u}(\widecheck u_2; \varepsilon, \widecheck d_2, -2n+\widehat\Lambda_n)&&\label{eq:FPfastBreveU2}
\end{flalign}
\begin{flalign}
\widecheck v_2(z;\varepsilon, \widehat\lambda_n) =&\widecheck d_2&&\nonumber\\
&+\frac1{2\pi}\bigg[\frac{\tanh(\pi z)}{\cosh^2(\pi z)+1-\pi z\tanh(\pi z)}\int_{0}^z \left[\sinh(\pi\tau)+\pi\tau\sech(\pi z) \right]^2\left(\widecheck{\mathcal N}(\tau; \varepsilon)+\varepsilon^2\widehat\lambda_n\right)\widecheck u_2(\tau;\varepsilon, \widehat\lambda_n) \rmd\tau &&\nonumber\\
&\qquad+\int_0^z \sech(\pi \tau)\left[\sinh(\pi\tau)+\pi\tau\sech(\pi\tau) \right]\left(\widecheck{\mathcal N}(\tau; \varepsilon)+\varepsilon^2\widehat\lambda_n\right)\widecheck u_2(\tau;\varepsilon, \widehat\lambda_n) \rmd\tau\bigg]&&\nonumber\\
=&:\widecheck{\mathcal F}_{2,v}(\widecheck u_2; \varepsilon, \widecheck d_2, -2n+\widehat\Lambda_n)\label{eq:FPfastBreveV2}.
\end{flalign}
\end{subequations}
Thus $(\widecheck \varphi_n, \widecheck\psi_n)$ satisfies (\ref{eq:fastFirst}) if, and only if, $\widecheck u_n$ and $\widecheck v_n$ satisfy (\ref{eq:FPfastBreve}). Using Lemma~\ref{lem:Nexp} and Claim~\ref{claim:fast} below we find that for all $\widecheck u_n\in\widecheck{\mathcal D}_\varepsilon(\rho)$, $z\in\widecheck I_f(\varepsilon)$ there exists $0<\widecheck C_2(\widecheck\varepsilon_1)<\infty$ such that
\begin{align*}
&\|\widecheck{\mathcal F}_{1,u}(\widecheck u_1; \varepsilon, \widecheck d_1, -2n+\widehat\Lambda_n)\|_\varepsilon \\
&\qquad\le |\widecheck d_1|+\frac\rho{2\pi}\left(\widecheck C_1(\widecheck\varepsilon_1)\varepsilon^{3/2}+\varepsilon^2(-2n+\widehat\Lambda_n) \right) \\
&\qquad\qquad\times\left\|\sech^2(\pi z)\int_{0}^z \left[\sinh(\pi\tau)\cosh(\pi\tau)+\pi\tau \right] \rmd\tau+ \left[\tanh(\pi z)+\pi z\sech^2(\pi z) \right]\int_0^z\rmd\tau\right\|_\varepsilon\\
&\qquad\le |\widecheck d_1|+\frac\rho{2\pi}\left(\widecheck C_1(\widecheck\varepsilon_1)\varepsilon+\varepsilon^{3/2}(-2n+\widehat\Lambda_n) \right)\widecheck C_2(\widecheck\varepsilon_1)(\sqrt\varepsilon+1)
\end{align*}
It is now straightforward to show that there exists constants $\widecheck\rho_1$, $\widecheck\rho_2>0$ and $0<\widecheck\varepsilon_0\le\widecheck\varepsilon_1$ such that $\widecheck{\mathcal F}_{1,u}(\widecheck u_n; \varepsilon, \widecheck d_n, \widecheck\Lambda_n) \in \widecheck{\mathcal D}_\varepsilon(\widecheck\rho_1)$ for all $\widecheck u_n\in\widecheck{\mathcal D}_\varepsilon(\widecheck\rho_1)$, $|\widecheck d_n|, |\widecheck\Lambda_n|\le \widecheck\rho_1$, and $\varepsilon\le\widecheck\varepsilon_0$. A completely analogous argument holds for $\widecheck{\mathcal F}_{1,v}$, $\widecheck{\mathcal F}_{2,u}$, and $\widecheck{\mathcal F}_{2,v}$.
Using this uniform bound on $\widecheck u_n$ in (\ref{eq:FPfastBreve}) and again Claim~\ref{claim:fast} we get the expansions\footnote{The notation $\calO_n$ refers to the fact that the constant may depend on $n$.} 
\begin{alignat*}{3}
\widecheck u_1(z;\varepsilon, -2n+\widehat\Lambda_n)=&\widecheck d_n\left[-\sech^2(\pi z)+\calO_n(\varepsilon + \varepsilon^{3/2}|\widehat\Lambda_n|)\right],\qquad& \widecheck u_2(z;\varepsilon,  -2n+\widehat\Lambda_n) =&\widecheck d_n\frac1{2\pi}\left[1+\calO_n(\varepsilon + \varepsilon^{3/2}|\widehat\Lambda_n|)\right],\nonumber\\
\widecheck v_1(z;\varepsilon, -2n+\widehat\Lambda_n)=& \widecheck d_n\pi\left[-\sech^2(\pi z)+\calO_n(\varepsilon + \varepsilon^{3/2}|\widehat\Lambda_n|)\right],\qquad&\widecheck v_2(z;\varepsilon,  -2n+\widehat\Lambda_n) =&\widecheck d_n\frac12\left[1+\calO_n(\varepsilon + \varepsilon^{3/2}|\widehat\Lambda_n|)\right].
\end{alignat*}
We observe that the leading order terms for $\widecheck u_1(z)$ and $\widecheck v_1(z)$ at the matching point $z=\pm z_0$ are the $\calO(\varepsilon)$ terms since $\sech^2(\pi z_0) = \calO(\rme^{-2\pi/\sqrt\varepsilon})$. Thus we compute the next order terms by plugging the expansion for $\widecheck u_1(z)$ back into (\ref{eq:FPfastBreveU1}) and integrating explicitly using the form of $\widecheck N_\alg$ and
\begin{align*}
&\int_{0}^z \left[\sinh(\pi\tau)\cosh(\pi\tau)+\pi\tau \right]\sech^2(\pi\tau) \rmd\tau=z\tanh(\pi z)\\
&\int_{0}^z \left[\sinh(\pi\tau)\cosh(\pi\tau)+\pi\tau \right]\sech^2(\pi\tau)2\pi\tau\tanh(\pi\tau) \rmd\tau=\pi z^2\tanh^2(\pi z)\\
&\int_{0}^z \left[\sinh(\pi\tau)\cosh(\pi\tau)+\pi\tau \right]\sech^2(\pi\tau)\tau^2 \rmd\tau\\&=\frac1{3\pi^3}\left(6\pi z\Li_2(-\rme^{-2\pi z})+3\Li_3(-\rme^{-2\pi z})-2\pi^3z^3-6\pi^2z^2\ln(1+\rme^{-2\pi z})+3\pi^3z^3\tanh(\pi z)+\frac{9\zeta(3)}4\right)\\
&\int_0^z\sech^2(\pi\tau)\rmd\tau = \frac1\pi\tanh(\pi z)\\
&\int_0^z\sech^2(\pi\tau)2\pi\tau\tanh(\pi\tau)\rmd\tau = \frac1{\pi}\left(\tanh(\pi z)-\pi z\sech^2(\pi z)\right)\\
&\int_0^z\sech^2(\pi\tau)\tau^2\rmd\tau =\frac1{\pi^3}\left(\Li_2(-\rme^{-2\pi z})-\pi^2z^2-2\pi z\ln(1+\rme^{-2\pi z})+\pi^2z^2\tanh(\pi z)+\frac{\pi^2}{12}\right)
\end{align*}
where $\zeta(z)$ is the Riemann zeta function. We get (\ref{eq:fastSoln}).
\end{proof}
\begin{table}[t]
\centering
\captionsetup{width=.7\linewidth}
\begin{tabular} {c c c c c}
\toprule
 & &$z\to 0$& & $z\to\infty$  \\ 
\hline
$\Li_2(-\rme^{-2\pi z})$ & \,& $-\frac{\pi^2}{12}+2\pi z\ln(2)-\pi^2z^2+\frac{\pi^3z^3}3+\calO(z^4)$ & \,& $\rme^{-2\pi z}[-1+\calO(\rme^{-2\pi z})]$\\
$\Li_3(-\rme^{-2\pi z})$ & \,& $-\frac{3\zeta(3)}{4}+\frac{\pi^3 z}6-\pi^2z^2\ln(4)+\frac{2\pi^3z^3}3+\calO(z^4)$ & \,& $\rme^{-2\pi z}[-1+\calO(\rme^{-2\pi z})]$\\
$\ln(\rme^{-2\pi z}+1)$ & \,&$ \ln(2)-\pi z+\frac{\pi^2z^2}2+\calO(z^3)$ & \,& $\rme^{-2\pi z}[1+\calO(\rme^{-2\pi z})]$\\
$\cosh(\pi z)$ & \,& $1+\frac{\pi^2z^2}2+\calO(z^4)$ & \,& $\frac12\rme^{\pi z}\left(1+\calO\left(\rme^{-2\pi z}\right)\right)$\\
$\sinh(\pi z)$ & \,&$\pi z+\frac{\pi^3z^3}6+\calO(z^5)$ & \,& $\frac12\rme^{\pi z}\left(1+\calO\left(\rme^{-2\pi z}\right)\right)$\\
$\tanh(\pi z)$ & \,&$\pi z-\frac{\pi^3z^3}3+\calO(z^5)$ & \,& $1+\calO\left(\rme^{-2\pi z}\right)$\\
$\sech(\pi z)$ & \,&$1-\frac{\pi^2z^2}2+\calO(z^4)$ & \,& $\rme^{-\pi z}\left(2+\calO\left(\rme^{-2\pi z}\right)\right)$\\
$\csch(\pi z)$ & \,&$\frac1{\pi z}-\frac{\pi z}6+\calO(z^3)$ & \,& $\rme^{-\pi z}\left(2+\calO\left(\rme^{-2\pi z}\right)\right)$\\
$\coth(\pi z)$ & \,&$\frac1{\pi z}+\frac{\pi z}3+\calO(z^3)$ & \,& $1+\calO\left(\rme^{-2\pi z}\right)$\\
\bottomrule
\end{tabular}
\caption{The asymptotic behavior of relevant functions for the integrals in (\ref{eq:FPfastBreve}). $\Li_n(x)$ is the polylogarithm function and $\zeta(z)$ is the Riemann zeta function. Expansions computed using Mathematica.}
\label{tab:fastTerms}
\end{table}
It remains to prove the following claim.
\begin{Claim}\label{claim:fast}
All integrals in (\ref{eq:FPfastBreve}) are well defined and bounded on any bounded interval. Furthermore, there exists $\widecheck\varepsilon_2>0$ such that the maximum of each of the following integrals for $|z|\le z_0$ occurs at $z=\pm z_0:=\pm1/\sqrt\varepsilon$ for all $\varepsilon\le\widecheck\varepsilon_2$
\begin{enumerate}
\item $\max_{|z|\le z_0} \left|\left[\tanh(\pi z)+\pi z\sech^2(\pi z) \right]\int_0^z\rmd\tau\right| = z_0+\calO(z_0^2\rme^{-2\pi z_0})$
\item $\max_{|z|\le z_0} \left|\left[\coth(\pi z)-\pi z\sech^2(\pi z) +\frac1{\cosh(\pi z)\sinh(\pi z)}\right] \int_0^z\rmd\tau \right|= z_0+\calO(z_0^2\rme^{-2\pi z_0}) $
\item $\max_{|z|\le z_0} \left|\int_0^z \tanh(\pi \tau)\rmd\tau\right| = z_0-\frac{\ln2}\pi+\calO(\rme^{-2\pi z_0})$
\end{enumerate}
and so that the following integrals are bounded uniformly in $z_0$
\begin{enumerate}
\setcounter{enumi}{3}
\item $\left|\sech^2(\pi z)\int_{0}^z \left[\sinh(\pi\tau)\cosh(\pi\tau)+\pi\tau \right] \rmd\tau\right|$
\item $\left|\frac{1}{\cosh(\pi z)\sinh(\pi z)+\pi z}\int_{0}^z \left[\sinh^2(\pi\tau)+\pi\tau\tanh(\pi\tau) \right] \rmd\tau\right|$
\item $\left|\frac{\tanh(\pi z)}{\cosh^2(\pi z)+1-\pi z\tanh(\pi z)}\int_{0}^z \left[\sinh^2(\pi\tau)+\pi\tau\tanh(\pi z) \right]\rmd\tau\right|$
\end{enumerate}
\end{Claim}
\begin{proof}
To show that the integrals are well defined we need to check that they are finite for all $z$ bounded. This is clear for (i), (iii) and (iv) since each of these expressions at $z=0$ equals zero. For (vi) we observe that $\cosh^2(\pi z)+1-\pi z\tanh(\pi z)$ is never zero since $\cosh^2(0)+1-0\cdot\tanh(0)=2>0$ and at $\pi z=2$
\begin{align*}
\frac\rmd{\rmd z}\left[\cosh^2(\pi z)+1-\pi z\tanh(\pi z)\right]\bigg|_{\pi z=2}=& \pi\left[2\cosh(\pi z)\sinh(\pi z)-\tanh(\pi z)-\pi z\sech^2(\pi z)\right]\bigg|_{\pi z=2}\\
 =& \frac\pi4\left[\sinh(4\pi z)-4\pi z\right]\sech^2(\pi z)\bigg|_{\pi z=2}> 0
\end{align*}
since $\sinh(x)-x \ge 0$ for all $x\ge 0$ (this can be seen since $\sinh(0) = 0$ and $\frac\rmd{\rmd x}\sinh(x) = \cosh(x) \ge 1$).
Thus it remains to consider (ii) and (v), which may develop a singularity at $z=0$.

(ii) We explicitly evaluate the integral to obtain $f_2(z):=\left[\coth(\pi z)-\pi z\sech^2(\pi z) +\sech(\pi z)\csch(\pi z)\right]z$.
Using the asymptotic expansions in Table~\ref{tab:fastTerms} we get
\begin{align*}
\lim_{z\to0} f_2(z) = \lim_{z\to0}\left[\frac2\pi +\calO(z^2)\right] = \frac2\pi
\end{align*}

(v) We explicitly integrate to obtain
\begin{align*}
f_5(z)=&:\frac{1}{\cosh(\pi z)\sinh(\pi z)+\pi z}\int_{0}^z \left[\sinh^2(\pi\tau)+\pi\tau\tanh(\pi\tau) \right] \rmd\tau \\
=&\frac{1}{\cosh(\pi z)\sinh(\pi z)+\pi z}\left[ \frac{\sinh(2\pi z)}{4\pi}-\frac z2-\frac{\pi}{24}-\frac{\Li_2(-\rme^{-2\pi z})}{2\pi}+\frac{\pi z^2}{2}+z\ln(\rme^{-2\pi z}+1) \right]
\end{align*}
where $\Li_2(x)$ is the polylogarithm function.
Using the expansions in Table~\ref{tab:fastTerms} we find 
\begin{align*}
\lim_{z\to0} f_5(z) = \lim_{z\to0} \frac{1}{2\pi z+\calO(z^3)}\left[ \frac{2\pi^2z^3}3+\calO(z^5) \right]= \lim_{z\to0} \frac{1}{1+\calO(z^2)}\left[ \frac{\pi z^2}3+\calO(z^4) \right]=0.
\end{align*}

Away from $z=0$ we use the fact that each of the six expressions is even; thus, without loss of generality, we assume $z\ge 0$. 
For (i)-(iii) we first show that the maximum occurs at $z=z_0$ and then use the large argument asymptotic expansion of the integral to evaluate the maximum. For (iv) - (vi) we explicitly compute the expression in the limit $z\to\infty$ and it is bounded; thus, since we've already shown that each expression is bounded for $z=0$ and they are continuous, they are bounded for all $z>0$. 

(i) We explicitly evaluate the integral to obtain $f_1(z):=\left[\tanh(\pi z)+\pi z\sech^2(\pi z)\right]z$.
Then
\begin{align*}
\lim_{z\to\infty} \frac{f_1(z)}{z} = 1
\end{align*}
so that $f_1(z)\sim z$ as $z\to\infty$; thus, since $z_0 = 1/\sqrt\varepsilon\xrightarrow{\varepsilon\to0}\infty$, there exists $\widecheck\varepsilon_2$ such that 
\[
\max_{0\le z\le z_0} f_1(z) = f_1(z_0)=z_0\left(1+\calO(z_0\rme^{-2\pi z_0})\right)
\]
for all $\varepsilon\le\widecheck\varepsilon_2$, where $f_1(z_0)$ was determined using the asymptotic expansions in Table~\ref{tab:fastTerms}. 

(ii) Follows exactly as (i).

(iii) The fact that the maximum occurs at $z=z_0$ is clear since $\tanh(\pi z)$ is monotone increasing. We integrate explicitly and use the asymptotic expansion for $\ln(1+\rme^{-2\pi z_0})$ for $z_0\gg1$ shown in Table~\ref{tab:fastTerms} to get the asymptotic expansion.

(iv) We explicitly integrate to obtain
\begin{align*}
f_4(z):=&\sech^2(\pi z)\int_{0}^z \left[\sinh(\pi\tau)\cosh(\pi\tau)+\pi\tau \right] \rmd\tau = \sech^2{\pi z}\left[\frac{\cosh^2(\pi z)-1}{2\pi}+\frac{\pi z^2}2 \right]\\
=& \frac1{2\pi}+\frac1{2\pi}\left[-1+\pi^2 z^2\right]\sech^2{\pi z}.
\end{align*}
It is now clear that $\lim_{z\to\infty} f_4(z) = 0$.

(v) Using the expansions in Table~\ref{tab:fastTerms} and $f_5(z)$ defined above, we find 
\begin{align*}
\lim_{z\to\infty} f_5(z) = \lim_{z\to\infty} \frac{4\rme^{-2\pi z}}{1+ \calO(z\rme^{-2\pi z})}\left[ \frac{\rme^{2\pi z}}{8\pi}+\calO(1) \right]= \lim_{z\to\infty} \frac{1}{1+ \calO(z\rme^{-2\pi z})}\left[ \frac{1}{2\pi}+\calO(\rme^{-2\pi z}) \right]=\frac1{2\pi}.
\end{align*}

(vi) Follows exactly as (v).
\end{proof}
At the matching point $z=1/\sqrt\varepsilon$, we will need the following improved estimates on $\widecheck\varphi_1$ and $\widecheck\psi_1$, which can be obtained by substituting (\ref{eq:fastSoln}) back into (\ref{eq:FPfastBreve}) one more time.
\begin{Proposition}\label{prop:fastSoln2}
Let $\varepsilon_0$,$\widecheck\rho_1$,$\widecheck\rho_2 > 0$ be as in Proposition~\ref{prop:fastSoln}. Then the set of all solutions to (\ref{eq:fastFirst}) with $\widehat\lambda_n= -2n+\widehat\Lambda_n$, $\|u(z)\|_\varepsilon\le\widecheck\rho_1$,  $|d_n|,|\widehat\Lambda_n|\le\widecheck\rho_2$ and $\widecheck U_n(0)=\widecheck d_n\widecheck V_1(0)$ are given at the matching point $z_0=1/\sqrt\varepsilon$ by
\begin{align}\label{eq:fastSoln2}
\widecheck \varphi_1(z_0; \varepsilon, \widehat\lambda_n)=&\widecheck d_n\left[\frac{n\varepsilon^2}{\pi^2}-\frac{n\varepsilon^3}{2\pi^2}+\calO_n(\varepsilon^{7/2} + \varepsilon^2|\widehat\Lambda_n|)\right]\cosh(\pi z_0),\nonumber\\
\widecheck \psi_1(z_0; \varepsilon, \widehat\lambda_n)=& \widecheck d_n\pi\left[\frac{n\varepsilon^2}{\pi^2}-\frac{n\varepsilon^3}{2\pi^2}+\calO_n(\varepsilon^{7/2} + \varepsilon^2|\widehat\Lambda_n|)\right] \sinh(\pi z_0)
\end{align}
\end{Proposition}

\subsection{Gluing}\label{sec:matching}
Using the approximations to the eigenfunctions in the slow and fast variables, from Propositions~\ref{prop:slowSoln} and \ref{prop:fastSoln} respectively, we show that there exists a unique global eigenfunction for (\ref{eq:etildeprob})
\begin{align}\label{eq:etildeprob3}
\lambda_n\widetilde\varphi_n:=\nu\partial_{xx}\widetilde\varphi_n -\frac12\left[\partial_xW_0(x, t; \nu)+\frac1{2\nu}W_0^2(x, t; \nu)\right]\widetilde\varphi_n
\end{align}
which can be constructed by gluing a fast eigenfunction to a slow eigenfunction at the overlap point $x=\varepsilon^{3/2}$. Due to symmetry considerations, we glue $\widehat\varphi_n$ to $\widecheck\varphi_1$ for $n$ odd and to $\widecheck\varphi_2$ for $n$ even. The matching conditions can be understood as follows. We need both that the functions $\widehat\varphi_n$ and $\widecheck\varphi_n$ are the same at the matching point as well as their slopes 
\[
\frac{\rmd\widehat\varphi_n((x-\pi)/\varepsilon, \cdot)}{\rmd x} = \frac1\varepsilon\widehat\psi_n((x-\pi)/\varepsilon; \cdot)\qquad\text{and}\qquad \frac{\rmd\widecheck\varphi_n(x/\varepsilon^2; \cdot)}{\rmd x} =\frac1{\varepsilon^2}\widecheck\psi_n(x/\varepsilon^2; \cdot).
\]
Since (\ref{eq:etildeprob3}) is linear, any scalar multiple of $\widehat\varphi_n(\xi; \varepsilon)$ and $\widecheck\varphi_n(x;\varepsilon)$ is an eigenfunction in the appropriate scaling regime; thus, instead of matching the slopes directly we impose the condition that the ratio of the fast eigenfunction and its derivatives is equal to the ratio of the slow eigenfunction and its derivative at the matching point:
\begin{subequations}\label{eq:matching}
\begin{align}
f_{n,1}(\breve\Lambda_n; \varepsilon):=\varepsilon^2\left[\frac{\widehat\psi_n((x-\pi)/\varepsilon; \varepsilon, \breve\Lambda_n)}{\varepsilon\widehat\varphi_n((x-\pi)/\varepsilon; \varepsilon, \breve\Lambda_1)} - \frac{\widecheck\psi_{\mathrm{mod}(n,2)+1}(x/\varepsilon^2; \varepsilon, \widehat\lambda_n)}{\varepsilon^2\widecheck\varphi_{\mathrm{mod}(n,2)+1}(x/\varepsilon^2; \varepsilon, \widehat\lambda_1)}\right]\bigg|_{x=\varepsilon^{3/2}}=0\label{eq:ratios}
\end{align}
where
\begin{align*}
\widehat\lambda_1:= -2+\xi_0\rme^{-\xi_0^2}\breve\Lambda_1,\quad \widehat\lambda_2:=-4+\xi_0^3\rme^{-\xi_0^2}\breve\Lambda_2, \quad \widehat\lambda_3:=-6+ \xi_0^5\rme^{-\xi_0^2}\breve\Lambda_3,\quad \text{and}\quad\widehat\lambda_4:=-8+\xi_0^7\rme^{-\xi_0^2}\breve\Lambda_4.
\end{align*} 
The factor $\varepsilon^2$ in front regularizes the problem and can be thought of as taking the $z$, rather than $x$, derivatives. We observe that (\ref{eq:ratios}) has no explicit dependence on the magnitude of the eigenfunctions. Using the Implicit Function Theorem we will show that there exists a unique fixed point to (\ref{eq:ratios}) near $\varepsilon=\breve\Lambda_n=0$. For this $\breve\Lambda_n$, we ensure that the magnitude of the slow and fast eigenfunction at the same at the matching point by showing that there exists a unique $C_n$ such that
\begin{align}
f_{n,2}(C_n, \breve\Lambda_n(\varepsilon); \varepsilon) := \left[\widehat\varphi_n((x-\pi)/\varepsilon; \varepsilon, \breve\Lambda_n) - C_n\widecheck\varphi_{\mathrm{mod}(n,2)+1}(x/\varepsilon^2; \varepsilon, \widehat\lambda_n)\right]\bigg|_{x=\varepsilon^{3/2}}=0\label{eq:functions}
\end{align}
\end{subequations}
which we will again show is true using the Implicit Function Theorem. We start with condition (\ref{eq:ratios}). Using the expansions (\ref{eq:slowSoln}) and (\ref{eq:fastSoln}) at the matching point $x=\varepsilon^{3/2}$ (equivalently, $\xi = (-\pi+\varepsilon^{3/2})/\varepsilon$ and $z=1/\sqrt\varepsilon$) with coefficients in front of $\breve\Lambda_n$  given by (\ref{eq:Melnikov}) we get
\begin{align*}
\frac1\pi f_{1,1}(\breve\Lambda_1; \varepsilon)=& \frac{\left[1-\frac{\sqrt\pi}4\breve\Lambda_1+\calO(\varepsilon^{-2}\rme^{-2\pi/\sqrt\varepsilon}\ln\varepsilon + \varepsilon^2|\breve\Lambda_1|)\right]\left[1+\calO(\varepsilon^{3/2})\right]}{\left[1+\frac{\sqrt\pi}4\breve\Lambda_1+\calO(\varepsilon^{-2}\rme^{-2\pi/\sqrt\varepsilon}\ln\varepsilon +\varepsilon^2 |\breve\Lambda_1|)\right]}
\\
&\qquad-\frac{\left[1+\calO(\varepsilon^{3/2} +\frac1\varepsilon \rme^{-(-\pi+\varepsilon^{3/2})^2/2\varepsilon^2}|\breve\Lambda_1|)\right] \left[1+\calO(\rme^{-2\pi/\sqrt\varepsilon})\right]}{\left[1+\calO(\varepsilon^{3/2} +\frac1\varepsilon \rme^{-(-\pi+\varepsilon^{3/2})^2/2\varepsilon^2}|\breve\Lambda_1|)\right]}.
\end{align*}
\begin{align*}
\frac1\pi f_{2,1}(\breve\Lambda_2; \varepsilon)=& \frac{\left[1-\frac{\sqrt\pi}8\breve\Lambda_2+\calO(\varepsilon^{-2}\rme^{-2\pi/\sqrt\varepsilon}\ln\varepsilon + \varepsilon^2|\breve\Lambda_2|)\right]\left[1+\calO(\varepsilon^{3/2})\right]}{\left[1+\frac{\sqrt\pi}8\breve\Lambda_2+\calO(\varepsilon^{-2}\rme^{-2\pi/\sqrt\varepsilon}\ln\varepsilon + \varepsilon^2|\breve\Lambda_2|)\right]\left[1+\calO(\varepsilon^{3/2})\right]}
\\
&\qquad-\frac{\left[1+\calO(\varepsilon + \frac1{\sqrt\varepsilon} \rme^{-(-\pi+\varepsilon^{3/2})^2/2\varepsilon^2}|\widehat\Lambda_2|)\right]\left[1+\calO(\frac1{\sqrt\varepsilon}\rme^{-2\pi/\sqrt\varepsilon})\right]}{\left[1+\calO(\varepsilon + \frac1{\sqrt\varepsilon}\rme^{-(-\pi+\varepsilon^{3/2})^2/2\varepsilon^2}|\widehat\Lambda_2|)\right]\left[1+\calO(\frac1{\sqrt\varepsilon}\rme^{-2\pi/\sqrt\varepsilon})\right]}
\end{align*}
\begin{align*}
\frac 1\pi f_{3,1}(\breve\Lambda_3; \varepsilon)=& \frac{\left[1-\frac{\sqrt\pi}8\breve\Lambda_3+\calO(\varepsilon^{-2}\rme^{-2\pi/\sqrt\varepsilon}\ln\varepsilon + \varepsilon^2|\breve\Lambda_3|)\right]\left[1+\calO(\varepsilon^{3/2})\right]}{\left[1+\frac{\sqrt\pi}8\breve\Lambda_3+\calO(\varepsilon^{-2}\rme^{-2\pi/\sqrt\varepsilon}\ln\varepsilon + \varepsilon^2|\breve\Lambda_3|)\right]\left[1+\calO(\varepsilon^{3/2})\right]}
\\
&\qquad-\frac{\left[1+\calO(\varepsilon^{3/2} +\frac1{\varepsilon^3} \rme^{-(-\pi+\varepsilon^{3/2})^2/2\varepsilon^2}|\breve\Lambda_3|)\right] \left[1+\calO(\rme^{-2\pi/\sqrt\varepsilon})\right]}{\left[1+\calO(\varepsilon^{3/2} +\frac1{\varepsilon^3} \rme^{-(-\pi+\varepsilon^{3/2})^2/2\varepsilon^2}|\breve\Lambda_3|)\right]}
\end{align*}
\begin{align*}
\frac1\pi f_{4,1}(\breve\Lambda_4; \varepsilon)=& \frac{\left[1-\frac{3\sqrt\pi}{16}\breve\Lambda_4+\calO(\varepsilon^{-2}\rme^{-2\pi/\sqrt\varepsilon}\ln\varepsilon + \varepsilon^2|\breve\Lambda_4|)\right]\left[1+\calO(\varepsilon^{3/2})\right] }{\left[1+\frac{3\sqrt\pi}{16}\breve\Lambda_4+\calO(\varepsilon^{-2}\rme^{-2\pi/\sqrt\varepsilon}\ln\varepsilon + \varepsilon^2|\breve\Lambda_4|)\right]\left[1+\calO(\varepsilon^{3/2})\right]}
\\
&\qquad-\frac{\left[1+\calO(\varepsilon + \frac1{\varepsilon^2\sqrt\varepsilon} \rme^{-(-\pi+\varepsilon^{3/2})^2/2\varepsilon^2}|\widehat\Lambda_4|)\right]\left[1+\calO(\frac1{\sqrt\varepsilon}\rme^{-2\pi/\sqrt\varepsilon})\right]}{\left[1+\calO(\varepsilon + \frac1{\varepsilon^2\sqrt\varepsilon}\rme^{-(-\pi+\varepsilon^{3/2})^2/2\varepsilon^2}|\widehat\Lambda_4|)\right]\left[1+\calO(\frac1{\sqrt\varepsilon}\rme^{-2\pi/\sqrt\varepsilon})\right]}
\end{align*}

It is clear that $f_{n,1}(0; 0) = 0$ and 
\[
\frac{\rmd f_{n,1}}{\rmd\breve\Lambda_n}\bigg|_{(\breve\Lambda_n; \varepsilon)=(0;0)} \neq0
\]
so that the hypotheses of the Implicit Function Theorem are satisfied. Expanding the unique function $\breve\Lambda_n(\varepsilon)$ in orders of $\varepsilon$ we find
\[
\breve\Lambda_1 = \calO(\varepsilon^{3/2}),\qquad\breve\Lambda_2 = \calO(\varepsilon),\qquad\breve\Lambda_3 = \calO(\varepsilon^{3/2}),\quad\text{and}\quad\breve\Lambda_4 = \calO(\varepsilon).
\]
Next we solve (\ref{eq:functions}) using the expansions for $\breve\Lambda_n(\varepsilon)$ and obtain the expressions
\begin{align*}
\rme^{-\pi/\sqrt\varepsilon}f_{1,2}(C_1, \calO(\varepsilon^{3/2}); \varepsilon)&:=\left[1+\calO(\varepsilon^{3/2})\right]\rme^{-\pi^2/2\varepsilon^2}\rme^{-\varepsilon/2}-\frac{C_1\varepsilon^2}{\pi^2}\left[1-\varepsilon/2+\calO(\varepsilon^{3/2})\right]\left[\frac12+\calO(\rme^{-2\pi/\sqrt\varepsilon}) \right]\\
\rme^{-\pi/\sqrt\varepsilon}f_{2,2}(C_2, \calO(\varepsilon); \varepsilon)&:=\frac\pi\varepsilon\left[1+\calO(\varepsilon)\right]\left[-1+\calO(\varepsilon^{3/2})\right] \rme^{-\pi^2/2\varepsilon^2}\rme^{-\varepsilon/2} -\frac{C_2}{2\pi}\left[1+\calO(\varepsilon)\right]\left[\frac12+\calO\left(\frac1{\sqrt\varepsilon}\rme^{-2\pi/\sqrt\varepsilon}\right)\right]\\
\rme^{-\pi/\sqrt\varepsilon}f_{3,2}(C_3, \calO(\varepsilon^{3/2}); \varepsilon)&:=\frac{\pi^2}{\varepsilon^2}\left[1+\calO(\varepsilon^{3/2})\right]\left[2+\calO(\varepsilon^{3/2})\right] \rme^{-\pi^2/2\varepsilon^2}\rme^{-\varepsilon/2}-\frac{3C_3\varepsilon^2}{\pi^2}\left[1+\calO(\varepsilon)\right]\left[\frac12+\calO(\rme^{-2\pi/\sqrt\varepsilon}) \right]\\
\rme^{-\pi/\sqrt\varepsilon}f_{4,2}(C_4, \calO(\varepsilon); \varepsilon)&:=\frac{\pi^3}{\varepsilon^3}\left[1+\calO(\varepsilon)\right]\left[-2+\calO(\varepsilon^{3/2})\right] \rme^{-\pi^2/2\varepsilon^2}\rme^{-\varepsilon/2}-\frac{C_4}{2\pi}\left[1+\calO(\varepsilon)\right]\left[\frac12+\calO\left(\frac1{\sqrt\varepsilon}\rme^{-2\pi/\sqrt\varepsilon}\right)\right].
\end{align*}
We define
\[
2\pi^2\breve C_1:= \varepsilon^2\rme^{\pi^2/2\varepsilon^2}\rme^{\varepsilon/2}C_1,\quad -4\pi^2\breve C_2:= \varepsilon \rme^{\pi^2/2\varepsilon^2}\rme^{\varepsilon/2}C_2,\quad \frac{4\pi^4}3\breve C_3:= \varepsilon^4\rme^{\pi^2/2\varepsilon^2}\rme^{\varepsilon/2}C_3,\quad\text{and}\quad -8\pi^4\breve C_4:= \varepsilon^3\rme^{\pi^2/2\varepsilon^2}\rme^{\varepsilon/2}C_4
\]
and
\begin{align*}
\breve f_{1,2}(\breve C_1; \varepsilon):=& \rme^{(\pi-\varepsilon^{3/2})^2/2\varepsilon^2}f_{1,2}\left(\frac{2\pi^2}{\varepsilon^2}\rme^{-\pi^2/2\varepsilon^2}\rme^{-\varepsilon/2}\breve C_1, \calO(\varepsilon^{3/2}); \varepsilon \right)\\
\breve f_{2,2}(\breve C_2; \varepsilon):=& \varepsilon \rme^{(\pi-\varepsilon^{3/2})^2/2\varepsilon^2}f_{2,2}\left(-\frac{4\pi^2}{\varepsilon}\rme^{-\pi^2/2\varepsilon^2}\rme^{-\varepsilon/2}\breve C_2, \calO(\varepsilon); \varepsilon \right)\\
\breve f_{3,2}(\breve C_3; \varepsilon):=& \varepsilon^2\rme^{(\pi-\varepsilon^{3/2})^2/2\varepsilon^2}f_{3,2}\left(\frac{4\pi^4}{3\varepsilon^4}\rme^{-\pi^2/2\varepsilon^2}\rme^{-\varepsilon/2}\breve C_3, \calO(\varepsilon^{3/2}); \varepsilon \right)\\
\breve f_{4,2}(\breve C_4; \varepsilon):=& \varepsilon^3\rme^{(\pi-\varepsilon^{3/2})^2/2\varepsilon^2}f_{4,2}\left(-\frac{8\pi^4}{\varepsilon^3}\rme^{-\pi^2/2\varepsilon^2}\rme^{-\varepsilon/2}\breve C_4, \calO(\varepsilon); \varepsilon \right).
\end{align*}
Now it is clear that $\breve f_{n,2}\left(1; 0\right)=0$
and
\[
\frac{\rmd \breve f_{n,2}}{\rmd\breve C_n}\bigg|_{(\breve C_n; \varepsilon)=(1;0)} \neq0
\]
so that the hypotheses of the Implicit Function Theorem are again satisfied. Expanding the unique function $\breve C_n(\varepsilon)$ in orders of $\varepsilon$ we find $\breve C_n(\varepsilon) = 1+\calO(\varepsilon)$, and, in particular, $\breve C_1(\varepsilon) = 1+\varepsilon/2+\calO(\varepsilon^{3/2})$. 

Putting everything together, and recalling the definitions $\varepsilon:=\sqrt{2\nu t}$, $I_s(\varepsilon):=[\varepsilon^{3/2}, 2\pi-\varepsilon^{3/2}]$, $I_f(\varepsilon):= [-\varepsilon^{3/2}, \varepsilon^{3/2}]$, we get that 
\begin{alignat}{3}
\lambda_1 =& \frac1{2t}\left(-2+\calO(\xi_0\rme^{-\xi_0^2}\breve\Lambda_1)\right)=-1/t+\calO(\varepsilon^{1/2}\rme^{-1/\varepsilon^2}),\nonumber\\
\lambda_2 =& \frac1{2t}\left(-4+\calO(\xi_0^3\rme^{-\xi_0^2}\breve\Lambda_2)\right)=-2/t+\calO\left(\varepsilon^{-2}\rme^{-1/\varepsilon^2}\right),\nonumber\\
\lambda_3 =& \frac1{2t}\left(-6+\calO(\xi_0^5\rme^{-\xi_0^2}\breve\Lambda_3)\right)=-3/t+\calO\left(\varepsilon^{-7/2}\rme^{-1/\varepsilon^2}\right),\nonumber\\
\lambda_4 =& \frac1{2t}\left(-8+\calO(\xi_0^7\rme^{-\xi_0^2}\breve\Lambda_4)\right)=-4/t+\calO\left(\varepsilon^{-4}\rme^{-1/\varepsilon^2}\right)
\label{eq:etildeValues2}
\end{alignat}
are eigenvalues for (\ref{eq:etildeprob3}) with associated eigenfunctions $\widetilde\varphi_n(x;t,\nu)$ which can be expanded in the intervals $I_s(\varepsilon)$ and $I_f(\varepsilon)$ as follows:
\begin{subequations}\label{eq:tildephi2}
\begin{flalign}
\widetilde\varphi_1: &\left\{ 
\begin{array}{ccc}
\sup_x\left|\rme^{(x-\pi)^2/2\varepsilon^2}\widetilde\varphi_1(x; t,\nu) +1\right| \le C(\varepsilon_0) \varepsilon^{3/2}&: &x\in I_s(\varepsilon) \\
\sup_x\left|\frac{\varepsilon^2}{2\pi^2}\rme^{\pi^2/2\varepsilon^2}\sech\left(\frac{\pi x}{\varepsilon^2}  \right)\widetilde\varphi_1(x; t,\nu)-\left[\sech^2\left(\frac{\pi x}{\varepsilon^2} \right)\left(1+\frac{x^2}{2\varepsilon^2}+\frac{\varepsilon^2}{2\pi^2} \right)-\frac{\varepsilon^2}{2\pi^2}\right] \right|\le C(\varepsilon_0)\varepsilon^{3/2} &: &x\in I_f(\varepsilon)
\end{array}
\right\}&&
\end{flalign}
\begin{flalign}
\widetilde\varphi_2 : & \left\{ \begin{array}{ccc}
\sup_x\left|\frac{\varepsilon}{x-\pi} \rme^{(x-\pi)^2/2\varepsilon^2}\widetilde\varphi_2(x; t,\nu) +1\right|\le C(\varepsilon_0) \varepsilon & : &x\in I_s(\varepsilon)\\
\sup_x\left|\frac\varepsilon{2\pi} \rme^{\pi^2/2\varepsilon^2}\widetilde\varphi_2(x; t,\nu)-\left[\sinh\left(\frac{\pi x}{\varepsilon^2}\right)+\frac{\pi x}{\varepsilon^2}\sech\left(\frac{\pi x}{\varepsilon^2}\right)\right] \right|\le C(\varepsilon_0)\varepsilon& : &x \in I_f(\varepsilon)
\end{array}\right\}&&
\end{flalign}
\begin{flalign}
\widetilde\varphi_3:&\left\{ 
\begin{array}{ccc}
\sup_y\left|\frac{\varepsilon^2}{2\left(x-\pi \right)^2-\varepsilon^2}\rme^{(x-\pi)^2/2\varepsilon^2}\widetilde\varphi_3(x; t, \nu)+1\right|\le C(\varepsilon_0)\varepsilon^{3/2}&: &x\in I_s(\varepsilon) \\
\sup_y\left|\frac{3\varepsilon^4}{4\pi^4}\rme^{\pi^2/2\varepsilon^2}\sech\left(\frac{\pi x}{\varepsilon^2}  \right)\widetilde\varphi_3(x; t, \nu)-\sech^2\left(\frac{\pi x}{\varepsilon^2} \right) \right| \le C(\varepsilon_0)\varepsilon&: &x\in I_f(\varepsilon)
\end{array}
\right\}&&
\end{flalign}
\begin{flalign}
\widetilde\varphi_4: & \left\{ \begin{array}{ccc}
\sup_y\left|\frac{\varepsilon^3}{(x-\pi) \left[2\left(x-\pi \right)^2-3\varepsilon^2\right]}\rme^{(x-\pi)^2/2\varepsilon^2}\widetilde\varphi_4(x; t, \nu)+1\right| \le C(\varepsilon_0)\varepsilon& : &x\in I_s(\varepsilon)\\
\sup_y\left|\frac{\varepsilon^4}{4\pi^3} \rme^{\pi^2/2\varepsilon^2}\csch\left(\frac{\pi x}{\varepsilon^2}\right)\widetilde\varphi_4(x; t, \nu)-1\right|\le C(\varepsilon_0)\varepsilon & : &x \in I_f(\varepsilon)
\end{array}\right\}.&&
\end{flalign}
\end{subequations}
Equations (\ref{eq:etildeValues2}) and (\ref{eq:tildephi2}) are expansions (\ref{eq:etildeValues}) and (\ref{eq:tildephi}), respectively, in Proposition~\ref{evalProp}. Proposition~\ref{evalProp} now follows from following proposition and Sturm-Liouville theory for periodic boundary conditions (c.f. \cite[Thms 2.1, 2.14]{MangusWinkler}), which states that the eigenvalues are strictly ordered $\lambda_0>\lambda_1\ge\lambda_2>\lambda_3\ge\lambda_4>\ldots$ and that an eigenfunction with exactly $2n$ crossings of zero in $x\in[-\pi,\pi)$ is the eigenfunction associated either with $\lambda_{2n-1}$ or with $\lambda_{2n}$.
\begin{Proposition}
Fix $\varepsilon_0\ll1$ such that the eigenfunctions $\widetilde\varphi_n(x; t,\nu)$ are given as in (\ref{eq:tildephi2}) for all $0\le\varepsilon\le\varepsilon_0$ with $\varepsilon:=\sqrt{2\nu t}$. Then $\widetilde\varphi_1(x; t,\nu)$ and $\widetilde\varphi_2(x;t,\nu)$ have exactly two zeros in the interval $x\in[-\pi, \pi)$ and the eigenfunctions $\widetilde\varphi_3(x; t,\nu)$ and $\widetilde\varphi_4(x;t,\nu)$ have exactly four zeros in the interval $x\in[\varepsilon^{3/2}, 2\pi-\varepsilon^{3/2})$ for all $0\le\varepsilon\le\varepsilon_0$.
\end{Proposition}
\begin{proof}
The $n=2,4$ cases are clear since $\sinh(\pi x/\varepsilon)=0$ at $x=0\in I_f(\varepsilon)$, $\frac{x-\pi}\varepsilon$ has a single zero at $x=\pi\in I_s(\varepsilon)$, and $2\left(\frac{x-\pi}\varepsilon\right)^2-3$ has two zeros at $x=\pi\pm\varepsilon\sqrt{3/2}\in I_s(\varepsilon)$, and by making $\varepsilon_0$ potentially smaller so that $-1+\calO(\varepsilon_0) < 0$. The result for $n=1,3$ is then a direct consequence of Sturm-Liouville theory since $\lambda_0>\lambda_1>\lambda_2$ and $\lambda_2>\lambda_3>\lambda_4$. 
\end{proof}
\section{Discussion}\label{sec:discussion}
In this work we have proposed a candidate metastable family for Burgers equation with periodic boundary conditions, which we denote $W(x,t;\nu, x_0, c)$. The metastable family depends on space and time and is parametrized by three parameters: the spatial location $x_0$, the ``initial" time $t_0$ (so that $t=t_0+\tau$), and mean $c_0$. Our choice of metastable family was motivated by our numerical experiments, one example of which is shown in Figure~\ref{fig:numerics}. We furthermore proposed an explanation for the metastable behavior of $W(x,t;\nu, x_0, c)$ based on the spectrum of the operator $\mathcal L$ which results from linearizing the Burgers equation about $W(x,t_*;\nu, x_*, c_*)$.
 In particular, we showed that a solution to the Burgers equation $u(x,t;\nu)$ which is close at some time $t_0$ to a profile in the metastable family (i.e. $u(x,t_0;\nu) = W(x,t_0; \nu, x_0, c_0)+v_0(x; t_0, x_0, c_0;\nu)$ with $\|v_0\|$ small) can be written as a perturbation from a (potentially different) profile $W(x,t_*;\nu, x_*, c_*)$ such that projection of the perturbation of $u(x,t_0;\nu)$ from $W(x,t_*;\nu, x_*, c_*)$ onto the span of the first three eigenfunctions associated with the linearization of the Burgers equation about $W(x,t_*;\nu, x_*, c_*)$ is zero. 
 These results are summarized in Theorems~\ref{evalThm} and \ref{convergeThm}. From a technical perspective, we derived the first five eigenvalues for $\mathcal L$ using Sturm-Liouville theory and ideas from singular perturbation theory. In particular, we show that there are two relevant space regimes which we call the ``slow" and ``fast" space scales; we construct the eigenfunctions in each regime separately and then rigorously glue the functions together using a Melnikov-like computation.

 As noted in Section \ref{sec:justify}, we regard these results as a first step toward showing that once solutions of Burgers equation are close to the family of Whitham solutions, they  subsequently evolve toward it at a rate much faster than the motion along the family itself.  The problem is that the 
 linearized evolution operator in \eqref{eq:linearized}-\eqref{eq:eprob} is non-autonomous and as
 is well known, in general, information on the spectrum of a non-autonomous, linear vector field
 does not immediately lead to bounds on its evolution.  Furthermore, even
 leaving aside the time dependence,  the operator in \eqref{eq:eprob} is highly non-self-adjoint
 which leads to further problems in deducing information about the evolution just from spectral
 data.  Such operators arise frequently in fluid mechanics and a number of different approaches
 have been proposed to deal with these issues (\cite{Beck2013,Gallagher2009,Deng2011,Bedrossian15}.)
 
 In the present case we feel that the spectral information is of greater use than is generally true for 
 two reasons - first, the transformation described in Section \ref{sec:evalue}, which shows that
 there is a bounded and invertible change of variables which conjugates the linearized operator
 \eqref{eq:eprob} to a self-adjoint operator, and second, the method of ``freezing coefficients'' which
 shows that for linear, non-autonomous equations in which the time-change occurs slowly, the
 spectral information does give good insight into the evolution of the solutions \cite{Vinograd1983}.
 In this case, the slow change in the vector-field is a consequence of the slow evolution along the
 family of Whitham solutions.  To provide a few more details of why we feel the solutions of
 Burgers should
 evolve in a fashion similar to
 that predicted by the spectral estimates established here, consider the linearized equation, written
 in self-adjoint form, i.e.
 \begin{equation}\label{eq:tilde}
 \widetilde{u}_t = \widetilde{\cal{L}}(\nu,t) \widetilde{u}\ ,
 \end{equation}
 where $\widetilde{u} = {\mathcal T}^{-1} u$ with ${\mathcal T}$ defined in \eqref{eq:tildeTrans},
 and $\widetilde{\cal L}$ defined in \eqref{eq:etildeprob}.
 
 We have computed the first four eigenvalues in the spectrum of $ \widetilde{\cal{L}}(\nu,t)$ for all $t$
 sufficiently large, so fix $t_0$ and set ${\widetilde{\cal L}}_0 = \widetilde{\cal{L}}(\nu,t_0)$ and define
 $a(\tau) =  \widetilde{\cal{L}}(\nu,t_0+ \tau ) -  \widetilde{\cal{L}}(\nu,t_0)$.

 Then 
 \begin{equation}
 \widetilde{u}_t = \widetilde{\cal{L}}(\nu,t) \widetilde{u} =  {\widetilde{\cal L}}_0\widetilde{u} + a(\tau) \widetilde{u}\ ,
 \end{equation}

 We can write the solution of this equation with the aid of DuHamel's formula as
 \begin{equation}\label{eq:duhamel}
 \widetilde{u}(\tau) = e^{\tau   {\widetilde{\cal L}}_0} \widetilde{u} + \int_0^{\tau} e^{(\tau -\sigma)   {\widetilde{\cal L}}_0} 
a(\sigma) \widetilde{u}(\sigma) d\sigma\ . 
\end{equation}

The leading order term is easy to estimate since we know (thanks to
Theorem 2) that $\widetilde{u}$ is orthogonal to the eigenfunctions $\widetilde{\phi}_0$,
$\widetilde{\phi}_1$, and $\widetilde{\phi}_2$ (of ${\widetilde{\cal L}}_0$).  In fact, thanks to fact that Burger's equation (and also the
linearized equation \eqref{eq:linearized}) preserve the mean value of the solution we can assume
without loss of generality that $\widetilde{u}(\tau)$ is orthogonal to $\widetilde{\phi}_0$ for all $\tau$.
Thus, let $P$ be the orthogonal projection onto the span of $\widetilde{\phi}_1$, and $\widetilde{\phi}_2$
and let $Q$ be its orthogonal complement.  To analyze the integral term in \eqref{eq:duhamel},
we break it up as
\begin{eqnarray}
\int_0^{\tau} e^{(\tau -\sigma)   {\widetilde{\cal L}}_0} 
a(\sigma) \widetilde{u}(\sigma) d\sigma &=& \int_0^{\tau} e^{(\tau -\sigma)   {\widetilde{\cal L}}_0} 
P a(\sigma) P \widetilde{u}(\sigma) d\sigma + \int_0^{\tau} e^{(\tau -\sigma)   {\widetilde{\cal L}}_0} 
Q a(\sigma) Q \widetilde{u}(\sigma) d\sigma \\
&& \qquad + \int_0^{\tau} e^{(\tau -\sigma)   {\widetilde{\cal L}}_0} 
P a(\sigma) Q \widetilde{u}(\sigma) d\sigma
+\int_0^{\tau} e^{(\tau -\sigma)   {\widetilde{\cal L}}_0} 
Q a(\sigma) P \widetilde{u}(\sigma) d\sigma\ .
\end{eqnarray}

At this point, our current estimates are not sufficient to analyze all the terms in this expression
in detail. However, we believe that leading order contribution comes from the first term
on the right-hand side of this expression.  For instance, the last two terms involve projections
$P a(\tau) Q$ and $Q a(\tau) P$ on complementary spectral subspaces and hence are
probably small, at least for $\tau$ small.  Likewise, the second term involves the evolution
of the part of the solution that lies  in the spectral subspace complementary to 
the span of $\widetilde{\phi}_0$,
$\widetilde{\phi}_1$, and $\widetilde{\phi}_2$ and hence is expected to decay like $e^{-\frac{3}{t_0} \tau}$.
Thus, we focus on the first integral expression.  We can write out the spectral projection $P$
in terms of inner products with $\widetilde{\phi}_1$, and $\widetilde{\phi}_2$ and we find
\begin{eqnarray}\label{eq:projections}
\int_0^{\tau} e^{(\tau -\sigma)   {\widetilde{\cal L}}_0} 
P a(\sigma) P \widetilde{u}(\sigma) d\sigma &=&
\left( \int_0^{\tau} e^{-\frac{1}{t_0} (\tau -\sigma) } (\widetilde{\phi}_1, a(\sigma) \widetilde{\phi}_1)
( \widetilde{\phi}_1, \widetilde{u}(\sigma)) d\sigma \right) \widetilde{\phi}_1 \\ \nonumber
&& \qquad +\left( \int_0^{\tau} e^{-\frac{2}{t_0} (\tau -\sigma) } (\widetilde{\phi}_2, a(\sigma) \widetilde{\phi}_2)
( \widetilde{\phi}_2, \widetilde{u}(\sigma)) d\sigma \right) \widetilde{\phi}_2
\end{eqnarray}
Note that in this expression we have used the fact that cross terms involving 
$\widetilde{\phi}_1$, and $\widetilde{\phi}_2$ will vanish by symmetry, and we have made the
approximation that the eigenvalues $\lambda_1$ and $\lambda_2$ are exactly
$-1/t_0$ and $-2/t_0$ for simplicity.

Now consider the inner products $(\widetilde{\phi}_j, a(\sigma) \widetilde{\phi}_j)$ that occur in the
integrands.  From the perturbation theory for linear operators, we know that if we perturb
${\widetilde{\cal L}}_0$ by $a(\tau)$, the first order shift in the eigenvalue $\lambda_j$ should be given
by exactly this inner product.  On the other hand, we know from our calculation of the 
spectrum that that the shift in the eigenvalue is given by
\begin{equation}
\delta \lambda_j = -\frac{ j}{t_0 + \tau} + \frac{ j}{t_0 } \sim \frac{j \tau}{t_0^2}\ .
\end{equation}

Thus, we expect the integrals in \eqref{eq:projections} to behave like
\begin{equation}
\frac{C}{t_0^2} \int_0^{\tau} e^{-\frac{j}{t_0}(\tau -\sigma)} \sigma (\widetilde{\phi}_j, \widetilde{u}(\sigma)) d\sigma
\end{equation}
 Since $(\widetilde{\phi}_j, \widetilde{u}(0)) = 0$ we expect that this inner product is bounded by
 $C \sigma \|\widetilde{u}_0 \|$, at least for $\sigma$ small, and hence the integrals in 
  \eqref{eq:projections} are expected to behave like
 \begin{equation}
\frac{C \|  \widetilde{u}_0 \| }{t_0^2} \int_0^{\tau} e^{-\frac{j}{t_0}(\tau -\sigma)} \sigma^2   d\sigma
\sim \frac{C \|  \widetilde{u}_0 \| }{t_0^2} \tau^3 \ ,
\end{equation} 
for $\tau$ small.

These estimates lead us to expect a bound on solutions of \eqref{eq:duhamel}
of the form
\begin{equation}
\| \widetilde{u}(\tau) \| \le C e^{-\frac{3}{t_0} \tau} + \frac{C \|  \widetilde{u}_0 \| }{t_0^2} \tau^3\ ,
\end{equation}
which for $\tau$ small, {\em but of order one}, is much faster decay than the rate of motion
along the family of Whitham solutions.  After some fixed time $\tau_0$, we stop the 
evolution with the ``frozen'' time operator ${\widetilde{\cal L}}(t_0)$ and restart the
process of tracking solutions of \eqref{eq:tilde} by approximating ${\widetilde{\cal L}}(t)$ by
${\widetilde{\cal L}}(t_0+\tau_0)$.  However, now, the initial condition for the equation
will be much closer to the manifold of Whitham solutions than the original
initial condition for \eqref{eq:tilde}.  We also note the similarity of this approach to the
renormalization method of \cite{Promislow2002} - see Fig. 2.2 of that reference.

Although our current estimates are not sufficient to rigorously establish the bounds in the 
previous paragraph, which we leave as an open problem, we feel that ubiquity of the
type of non-self-adjoint operators exemplified by ${\cal L}$ in fluid mechanics, along with
the paucity of rigorous estimates of their spectral behavior makes the results presented in this
paper of interest, even though they do not conclusively prove that solutions approach the
Whitham family with the expected rate.  In addition, we feel that the methods derived in 
this paper for studying the behavior of multiple eigenvalues of singularly perturbed spectral
problems may be of independent interest.

It also is worth reiterating that our results show that the spectrum for $\mathcal L$ is, to leading-order, independent of the viscosity $\nu$; this result is particularly interesting since our analysis is not valid for the inviscid equation. Furthermore, our results are in contrast to \cite{Beck2013}, in which the authors proposed an analytical description of the ``bar" metastable family for the Navier--Stokes equation with periodic boundary conditions which were observed numerically in \cite{Yin2003}, denoted $\omega^b$. In \cite{Beck2013} the authors provided numerical evidence and analytical arguments which indicate that the real part of the least negative eigenvalue for the operator obtained from linearizing the Navier--Stokes equation about $\omega^b$ is proportional to $\sqrt\nu$; in other words, the metastable behavior of $\omega^b$ does depend on the viscosity. On the other hand, in \cite{Bedrossian15}, Bedrossian, Masmoudi and Vicol show that the solution behavior for the Navier--Stokes equation in a neighborhood of the Couette flow depends on the time-regime: for small enough time scales the solution behavior is governed by the inviscid limit of Navier--Stokes, whereas viscid effects dominate after long enough times. Thus, our results raise the question about whether there is an even earlier time regime for the Navier--Stokes with periodic boundary conditions than that studied in \cite{Beck2013}, and a potentially different metastable family, in which convergence to a metastable family is independent of the viscosity.

\begin{Acknowledgment}
The authors wish to thank M. Beck for numerous discussions of metastability in Burgers equations and the Navier--Stokes equation and to thank C.K.R.T. Jones, T. Kaper, and B. Sandstede for discussions of Lin's method and singularly perturbed eigenvalue problems. We also appreciate the thoughtful comments and helpful suggestions of the anonymous referees. The work of CEW is supported in part by the NSF grant DMS-1311553. 
\end{Acknowledgment}
\newpage
\appendix
\renewcommand{\theequation}{\Alph{section}.\arabic{equation}}
\renewcommand{\theLemma}{\Alph{section}.\arabic{Lemma}}
\section{Notation}\label{app:notation}
\begin{longtable} {L C R}
\toprule
Variable & Description & Defined in\\
\midrule
$\psi^W(x,t;\nu)$&A solution to the periodic heat equation. It is also used to define transformation (\ref{eq:tildeTrans})&Equation (\ref{eq:psi})\\
$W_0(x,t;\nu)$ & An exact solution to the periodic Burgers equation (\ref{eq:Burgers}) constructed from $\psi^W(x,t;\nu)$ via the Cole--Hopf transformation. &Equation (\ref{eq:uW})\\
$W(x,t;\nu, \Delta x, c)$ & The family of metastable solutions, parametrized by $\Delta x$, $t$, and $c$, given by $W(x, t; \nu, \Delta x, c):=c+W_0(x-\Delta x-ct,t; \nu)$& Section~\ref{sec:Whit}\\
$\mathcal L(\nu,t)$ &  The time-dependent linear operator obtained from linearizing (\ref{eq:Burgers}) about the solution family $W_0(x,t;\nu)$& Equation (\ref{eq:eprob})\\
$\widetilde{\mathcal L}(\nu,t)$ & The time-dependent self-adjoint linear operator associated with $\mathcal L(\nu,t)$ after transforming the eigenfunctions $\varphi_n$ into $\widetilde\varphi_n$ via (\ref{eq:tildeTrans}) & Equation (\ref{eq:etildeprob})\\
$\mathcal T(x;t,\nu)$ & The transformation which maps eigenfunctions for $\widetilde{\mathcal L}(\nu,t)$ into eigenfunctions for $\mathcal L(nu,t)$ & Equation (\ref{eq:tildeTrans})\\
$(\lambda_n,\varphi_n(x; t_0,\nu))$ & Solutions to the frozen-time eigenvalue problem $\lambda_n\varphi = \mathcal L(\nu,t_0)\varphi_n$& Equation (\ref{eq:eprob})\\
$(\lambda_n,\widetilde\varphi_n(x;t_0,\nu))$ & Solutions to the associated frozen-time self-adjoint eigenvalue problem $\lambda_n\widetilde\varphi_n =\widetilde{\mathcal L}(\nu,t_0)\widetilde\varphi_n$ & Equation (\ref{eq:etildeprob}). Note: $\varphi_n$ and $\widetilde\varphi_n$ are related via transformation (\ref{eq:tildeTrans})\\
$x_0$, $t_0$ & Initial parameter values such that the frozen time solution $u(x, t_0;\nu)$ to (\ref{eq:Burgers}) is near $W(x,t_0; \nu, x_0, c)$ & Theorem~\ref{convergeThm}, Section~\ref{sec:outline}\\
$x_*$, $t_*$ & Perturbed parameter values so that the frozen time solution $u(x, t_0;\nu)$ to (\ref{eq:Burgers}) is near $W(x,t_*; \nu, x_*, c)$ and the projection of the perturbation onto the subspace spanned by the eigenfunctions corresponding to the first three eigenvalues is zero &Theorem~\ref{convergeThm}, Section~\ref{sec:outline} \\
$\varepsilon:=\sqrt{2\nu t}$ & Small parameter used in singular perturbation arguments & Section~\ref{sec:outline} and again in Proposition~\ref{evalProp} \\
$I_s(\varepsilon)$, $I_f(\varepsilon)$ & The spatial intervals where the slow equation and fast equation dominate, respectively & Proposition~\ref{evalProp}; see also Figure~\ref{fig:tildePhi}\\
\bottomrule
\caption{General notation used throughout this work.}
\label{tab:notation}
\end{longtable}

\newpage

\begin{longtable} {L C R}
\toprule
Variable & Description & Defined in\\
\midrule
$\xi :=\frac{x-\pi}\varepsilon$ &  Slow spatial variable & Beginning of Section~\ref{sec:slow}\\
$\widehat I_s(\varepsilon)$ & The slow interval $I_s(\varepsilon)$ in terms of $\xi$ & Beginning of Section~\ref{sec:slow}  \\
$\widehat W(\xi;\varepsilon)$ &$W_0(x,t;\nu)$ written in terms of $\xi$ and scaled by $\frac t\varepsilon$ & Beginning of Section~\ref{sec:slow}  \\
$\widehat W_\xi(\xi;\varepsilon)$ &$\partial_xW_0(x,t;\nu)$ written in terms of $\xi$ and scaled by $t$ & Beginning of Section~\ref{sec:slow}  \\
$\widehat \varphi_n(\xi)$ & The eigenfunction $\widetilde\varphi_n(x)$ in terms of $\xi$ & Beginning of Section~\ref{sec:slow}  \\
$\widehat \lambda_n := 2t\lambda_n$ & A transformation of the eigenvalue $\lambda_n$ & Beginning of Section~\ref{sec:slow}  \\
$\widehat \Lambda_n := \widehat\lambda_n+2n$ & Perturbation of the eigenvalue $\widehat\lambda_n$ from $-2n$, the eigenvalue anticipated by the formal analysis of the slow variables in Section~\ref{sec:overview}& Before Lemma~\ref{lemma:Nhat}  \\
$\widehat U_n$ & 2-component vector representation of the eigenfunction $\widehat\varphi_n$, used to make the eigenvalue problem first order& Before Lemma~\ref{lemma:Nhat}  \\
$\widehat{\mathcal A}_n(\xi)$ & A $2\times 2$ non-autonomous real matrix giving the leading order terms in the eigenvalue problem for $\widehat U_n$ & Before Lemma~\ref{lemma:Nhat}, part of (\ref{eq:slowFirst})  \\
$\widehat{\mathcal N}_n(\widehat U_n, \xi; \varepsilon, \widehat\Lambda_n)$ & A $2\times 1$ real vector giving the higher order terms in the eigenvalue problem for $\widehat U_n$ & Before Lemma~\ref{lemma:Nhat}, part of (\ref{eq:slowFirst})  \\
$\widehat{\mathcal N}(\xi; \varepsilon)$ & The part of $\widehat{\mathcal N}_n(\widehat U_n, \xi; \varepsilon, \widehat\Lambda_n)$ that comes from the difference between the formal slow-variable potential with $\varepsilon=0$ (\ref{eq:eslow0}) and the potential in the slow-variable eigenvalue problem (\ref{eq:eslow2})& Before Lemma~\ref{lemma:Nhat}, part of (\ref{eq:slowFirst})  \\
$\xi_0 $& The point at which the eigenfunctions in each of the scaling regimes will be matched at in terms of $\xi$ & Before Proposition~\ref{prop:slowSoln} \\
$\breve \Lambda_n$ & Exponential rescaling of the eigenvalue offset $\widehat\Lambda_n$; necessary for an Implicit Function Theorem argument & Proposition~\ref{prop:slowSoln}  \\
$H_{n-1}(\xi)\rme^{-\xi^2/2}$ & Eigenfunction solutions to the formal slow-variable potential with $\varepsilon=0$ (\ref{eq:eslow0}) with $\widehat\lambda_n=-2n$ & After equation (\ref{eq:eslow0}) \\
\bottomrule
\caption{Notation used for the slow variable analysis in Section~\ref{sec:slow}.}
\label{tab:notationSlow}
\end{longtable}

\newpage

\begin{longtable} {L C R}
\toprule
Variable & Description & Defined in\\
\midrule
$z:=\frac x{\varepsilon^2}$ &  Fast spatial variable & Beginning of Section~\ref{sec:fast}\\
$\widecheck I_f(\varepsilon)$ & The fast interval $I_f(\varepsilon)$ in terms of $z$ & Beginning of Section~\ref{sec:fast}  \\
$\widecheck W(z;\varepsilon)$ &$W_0(x,t;\nu)$ written in terms of $z$ and scaled by $t$ & Beginning of Section~\ref{sec:fast}  \\
$\widecheck W_z(z;\varepsilon)$ &$\partial_xW_0(x,t;\nu)$ written in terms of $z$ and scaled by $t\varepsilon^2$ & Beginning of Section~\ref{sec:fast}  \\
$\widecheck \varphi_n(z)$ & The eigenfunction $\widetilde\varphi_n(x)$ in terms of $z$ & Beginning of Section~\ref{sec:fast}  \\
$\widecheck U_n$ & 2-component vector representation of the eigenfunction $\widecheck\varphi_n$, used to make the eigenvalue problem first order& Before Lemma~\ref{lem:Nexp} \\
$\widecheck{\mathcal A}_n(z)$ & A $2\times 2$ non-autonomous real matrix giving the leading order terms in the eigenvalue problem for $\widecheck U_n$ & Before Lemma~\ref{lem:Nexp},  part of (\ref{eq:fastFirst})  \\
$\widecheck{\mathcal N}_n(\widecheck U_n, z; \varepsilon, \widehat\Lambda_n)$ & A $2\times 1$ real vector giving the higher order terms in the eigenvalue problem for $\widecheck U_n$& Before Lemma~\ref{lem:Nexp}, part of (\ref{eq:fastFirst})   \\
$\widecheck{\mathcal N}(z; \varepsilon)$ & The part of $\widecheck{\mathcal N}_n(\widecheck U_n, z; \varepsilon, \widehat\Lambda_n)$ that comes from the difference between the formal fast-variable potential with $\varepsilon=0$ (\ref{eq:efast0}) and the potential in the fast-variable eigenvalue problem (\ref{eq:efast2})& Before Lemma~\ref{lem:Nexp}  \\
$\widecheck{\mathcal N}_{\mathrm{alg}}(z; \varepsilon)$ & The part of $\widecheck{\mathcal N}(z; \varepsilon)$ that behaves algebraically& Lemma~\ref{lem:Nexp} \\
$\widecheck{\mathcal N}_{\mathrm{exp}}(z; \varepsilon)$ & The part of $\widecheck{\mathcal N}(z; \varepsilon)$ that behaves exponentially& Lemma~\ref{lem:Nexp}  \\
$z_0 $& The point at which the eigenfunctions in each of the scaling regimes will be matched at in terms of $z$ & Before Proposition~\ref{prop:fastSoln} \\
$P(z)$, $Q(z)$ & Two linearly independent solutions to the formal fast-variable equation with $\varepsilon=0$ (\ref{eq:efast0}) & After equation (\ref{eq:efast0}) \\
\bottomrule
\caption{Notation used for the fast variable analysis in Section~\ref{sec:fast}.}
\label{tab:notationFast}
\end{longtable}

\bibliographystyle{SIADS}
\bibliography{periodicBurgers}

\end{document}